\documentclass[11pt,a4paper,twoside]{amsart}
\usepackage{amssymb,amsmath,amsthm,graphicx,color
}
\theoremstyle{plain}
\newtheorem{theorem}{Theorem}[section]
\newtheorem{definition}[theorem]{Definition}
\newtheorem{lemma}[theorem]{Lemma}
\newtheorem{corollary}[theorem]{Corollary}
\newtheorem{proposition}[theorem]{Proposition}

\theoremstyle{remark}
\newtheorem{remark}[theorem]{Remark}

\usepackage{hyperref}
\numberwithin{equation}{section}

\newcommand{\C}{\mathbb{C}}
\newcommand{\R}{\mathbb{R}}
\newcommand{\Z}{\mathbb{Z}}

\newcommand{\F}{\mathcal{F}}

\renewcommand{\Re}{\operatorname{Re}}
\newcommand{\I}{\infty}
\newcommand{\abs}[1]{\left\lvert #1\right\rvert}
\newcommand{\norm}[1]{\left\lVert #1\right\rVert}

\newcommand{\Jbr}[1]{\left\langle #1 \right\rangle}

\newcommand{\IN}{\quad\text{in }}

\newcommand{\swIN}{\quad \text{weakly-}*\text{ in }}

\def\({\left(}
\def\){\right)}
\def\<{\left\langle}
\def\>{\right\rangle}
\def\le{\leqslant}
\def\ge{\geqslant}

\def\d{{\partial}}

\newcommand{\eps}{\varepsilon}

\DeclareMathOperator{\supp}{supp}

\begin{document}
\title[Two minimization problems for mass-subcritical NLS]
{Two minimization problems on non-scattering solutions to mass-subcrticial
nonlinear Schr\"odinger equation}
\author%[S. Masaki]
{Satoshi Masaki}
\address{Department of Systems Innovation\\
Graduate School of Engineering Science\\
Osaka university\\
Toyonaka, Osaka 560-8531, Japan}
\email{masaki@sigmath.es.osaka-u.ac.jp}
\date{}
\maketitle
\vskip-5mm
\begin{abstract}
In this paper, we introduce two minimization problems on
non-scattering solutions to %mass-subcritical
 nonlinear Schr\"odinger equation. 
One gives us a sharp scattering criterion, the other is
concerned with minimal size of blowup profiles.
We first reformulate several previous results in terms of these two minimizations.
Then, the main result of the paper is existence of minimizers to the both minimization
problems for mass-subcritical nonlinear Schr\"odinger equations.
To consider the latter minimization, we consider the equation in a Fourier transform of generalized Morrey space.
It turns out that the minimizer to the latter problem possesses a compactness property, which
is so-called almost periodicity modulo symmetry.
\end{abstract}

\section{Introduction}\label{sec:introduction}

In this paper, we consider time global behavior of 
solutions to the following focusing nonlinear Schr\"odinger equation
\begin{equation}\label{eq:NLS}
\left\{
\begin{aligned}
& i\d_t u + \Delta u = - |u|^{2\alpha} u, \quad (t,x) \in I\times \R^{d} \\
& u(t_0,x) = u_0(x),
\end{aligned}
\right.
\end{equation}
where $d\ge1$, $I\subset \R$ is an interval, $t_0 \in I$, and 
$u(t,x):I\times \R^d \to \C$ is an unknown.

We introduce two minimization problems
associated with time global behavior of solutions to \eqref{eq:NLS}.
First, we briefly recall several previous results in terms of the two problems.
Then, we consider some mass-subcritical cases $\alpha < 2/d$
and establish existence of minimizers to the both problems, which is the main result.

\subsection{Two minimization problems}

To be more precise, let us make some notation.
Let $X$ be a Banach space which corresponds to a \emph{state space}.
Suppose that $X$ is so that $e^{it\Delta}$ becomes a linear bounded operator on $X$
and that $e^{it\Delta}$ converges strongly to the identity on $X$ as $|t|\to 0$.
We also suppose 
that \eqref{eq:NLS} is locally well-posed in $X$ in the following sense:
For any given $u_0 \in X$ and $t_0\in \R$,
there exist an interval $I \ni t_0$ and a unique function $u(t,x) \in I \times \R^d
\to \C$ such that $ u(t) \in C(I, X)$ and the equality
\begin{equation}\label{eq:iNLS}
	 u(t) = e^{i(t-t_0)\Delta} u_0 +i \int_{t_0}^t
	e^{i(t-s)\Delta} (|u|^{2\alpha} u)(s) ds \IN X
\end{equation}
holds for all $t\in I$.
Moreover, continuous dependence on the data holds:
If $u_{0,n} \in X$ converges to $u_0$ in $X$ as $n\to\I$ then
a sequence of corresponding solutions $u_n(t)$ with data $u_n(t_0)=u_{0,n}$
converges to $u(t)$ in $L^\I(J,X)$ as $n\to\I$ 
for any $J \subset\subset I$.
We say that $u$ is a maximal-lifespan solution if it cannot be extended to any strictly larger interval.
Let $I_{\mathrm{max}}=I_{\mathrm{max}}(u):= 
( T_{\mathrm{min}}(u), T_{\mathrm{max}}(u)) $ be the maximal interval of a $u$.

We say a solution $u(t)$ \emph{scatters} in $X$ for positive time direction 
(resp.~negative time direction) if $T_{\mathrm{max}}=+\I$
(resp.~$T_{\mathrm{min}}=-\I$) and
if $e^{-it\Delta} u(t)$ has a strong limit in $X$ as $t\to\I$
(resp.~$t\to-\I$).

\subsection*{First minimization problem}
We now introduce the first problem.
Let $\ell: X \to \R_{\ge 0}$ be a \emph{size function}.
Suppose that $\ell$ is continuous.
The first minimization is the following;
\begin{equation*}
	\begin{aligned}
	E_1&{}=E_1 (\alpha, X, \ell) \\
	&{}:= \inf \left\{ \inf_{t \in I_{\mathrm{max}}(u) }  \ell(u(t)) \ \middle|\ 
	\begin{aligned}
	& u(t):  \text{solution to \eqref{eq:NLS} that does not}\\
	& \text{scatter for positive time direction.}
	\end{aligned}
	\right\}
	\end{aligned}
\end{equation*}
When $\ell(\cdot) = \norm{\cdot}_X$, we simply denote $E_1(\alpha ,X)$.
If a size function $\ell(\cdot)$ is invariant under complex conjugation,
that is, if 
\begin{equation}\label{eq:conj}
	\ell(\overline{f}) = \ell(f), \quad \forall f\in X
\end{equation}
is satisfied then it is obvious from
time symmetry of \eqref{eq:NLS}  that
\begin{align*}
		E_1 %(\alpha, X, \ell) \\
	&{}= \inf
	\left\{ \inf_{t \in I_{\mathrm{max}}(u) }  \ell(u(t)) \ \middle|\ 
	\begin{aligned}
	& u(t):  \text{solution to \eqref{eq:NLS} that does not}\\
	& \text{scatter for negative time direction.}
	\end{aligned}
	\right\}\\
	&{}= \inf
	\left\{ \inf_{t \in I_{\mathrm{max}}(u) }  \ell(u(t)) \ \middle|\ 
	\begin{aligned}
	& u(t):  \text{solution to \eqref{eq:NLS} that does not}\\
	& \text{scatter for at least one time direction.}
	\end{aligned}
	\right\}.
\end{align*}
Notice that the validity of small data scattering result in $X$ is expressed as
$E_1(\alpha, X,\ell) > 0$.
In other words, the positivity of $E_1$ suggests that $(X,\ell)$ is a suitable framework
to consider time global behavior of solutions. 
The value $E_1$ gives us a sharp scattering criterion;
\[
	\ell(u_0) < E_1 \Longrightarrow
	u(t) \text{ scatters for positive time direction}.
\]
Further, $u(t)$ scatters for both time direction if $\ell$ satisfies \eqref{eq:conj}.

Also remark that $E_1(\alpha, X)<\I$ is equivalent to existence of a nonscattering solution in $X$.
\medskip

\subsection*{Second minimization problem}
The next problem is 
\begin{align*}
	E_2&{}=E_2(\alpha, X, \ell) \\
	&{}:= \inf \left\{  
	\limsup_{t \uparrow T_{\max}(u)} \ell(u(t)) \ \middle|\ 
	\begin{aligned}
	& u(t):  \text{solution to \eqref{eq:NLS} that does not}\\
	& \text{scatter for positive time direction.}
	\end{aligned}
	\right\}.
\end{align*}
As in $E_1$, a similar infimum value for negative time direction
has the same value under the assumption \eqref{eq:conj}.
Intuitively, $E_2$ is a minimum size of possible ``blowup profiles.''\footnote{
It is known that, in some cases, a solution that does not scatter for positive time direction
tends to an orbit of a static profile, a blowup profile, by a group action, say $G$, as time approaches to the end of maximal time interval
(e.g., a standing wave solution $u(t,x)= e^{it\omega} \phi_\omega (x)$).
If a size function $\ell$ is chosen so that it is invariant under the group action $G$,
 then the size of such a solution tends to that of a corresponding profile.
Of course, another kind of behavior may take place, in general, and so it may not be a definition of $E_2$.
}

% It is known that, in some case, a solution that does not scatter for positive time direction
% tends to 
% an orbit of a static profile by a group action, say $G$, as time approaches to the end of maximal time interval
% (e.g., a standing wave solution $u(t,x)= e^{it\omega} \phi_\omega (x)$).
% If a size function $\ell$ is chosen so that it is invariant under the group action $G$,
% then the size of such a solution tends to that of a corresponding profile.
% An intuitive meaning of $E_2$ is a minimum size of such profiles.
% Of course, this ``definition'' is not correct in general because it may happen that a solution
%neither scatters nor stays close to a profile modulo $G$.

It is obvious by definition that $$0\le E_1 \le E_2 \le \I.$$
Further, if $\ell(\cdot)$ is a time independent quantity, 
such as $L^2$-norm of the solution, then two infimum values coincide.
On the other hand, it may happen that $E_1 < E_2$ (see Theorems \ref{thm:previous2} and \ref{thm:previous3}).
A good point on these minimizations is that $\ell(\cdot)$ needs not to be a time independent quantity.
Hence, it enables us to consider the problem under a setting that no conserved quantity is available.

%
% \bigskip
Besides $E_1=\I$ implies non-existence of a nonscattering solution,
$E_2=\I$ gives us a weaker statement that any bounded (in the sense of a corresponding size function) solution scatters.
This kind of scattering result is extensively studied in the defocusing case (see \cite{KM2,KV2,MMZ,Mu1,Mu2,Mu3}, for instance). 
Although we need some modification on $E_1$ and $E_2$, there is an example of the setting that yields $E_1<E_2=\I$
(see Theorems \ref{thm:previous3} and \ref{thm:previous4}).

In this way, we can obtain somewhat detailed information on dynamics
from a combination of these two values.

\subsection{Evaluation of $E_1$ and $E_2$ for mass-critical and -supercritical cases}
The main interest here is
to find explicit values of $E_1$ and $E_2$.
Further, we seek minimizers.
Besides its own interest, 
to know a minimizer would be a key step of finding the explicit infimum values.
In some mass-critical and -supercritical settings,
a \emph{ground state solution} $Q_\alpha(t,x)$ attains $E_1$ and/or $E_2$,
where  $Q_\alpha(t,x)$ is given by
\[
	Q_\alpha(t,x):=
	\begin{cases}
	e^{it} Q(x), & 0< \alpha < \frac{2}{d-2},\\
	W(x) = \(1+\frac{|x|^2}{d(d-2)}\)^{-\frac{d-2}2}, & \alpha = \frac{2}{d-2}
	\end{cases}
\]
Here, $Q(x)$ is
 a positive radial solution of $-\Delta Q + Q = Q^{2\alpha+1}$
and $W$ solves $-\Delta W = W^{\frac{d+2}{d-2}}$.

Let us now collect several settings that
explicit value of $E_1$ and $E_2$ can be determined.
They are reformulations of previous results.
Let us begin with the mass-critical case.
\begin{theorem}[mass-critical case]\label{thm:previous1}
Let $d\ge1$ and $\alpha = 2/d$. Let $X=L^2(\R^d)$ and $\ell(\cdot)=
\norm{\cdot}_{L^2}$.
Then,
$E_1 = E_2 = \norm{Q}_{L^2(\R^d)}$.
\end{theorem}
This result is just a rephrase of Dodson \cite{Do1} (see also \cite{KTV,KVZ}).
The next case is energy-critical case.
\begin{theorem}[energy-critical case]\label{thm:previous2}
Let $d\ge4$ and  $\alpha = 2/(d-2)$. 
Let $X=\dot{H}^1(\R^d)$ and $\ell(\cdot)=
\norm{\cdot}_{\dot{H}^1}$. Then,
\[
	E_1 = \sqrt{\frac2d} \norm{W}_{\dot{H}^1(\R^d)},\quad
	E_2=\norm{W}_{\dot{H}^1(\R^d)}
\]
% \[
% 	E_2 (2/(d-2), \dot{H}^1(\R^d)) = \norm{W}_{\dot{H}^1(\R^d)},
% \]
% Then, $E_2=\norm{W}_{\dot{H}^1(\R^d)}$.
Moreover, there exists a global radial solution $W_-(t)$ such that
\begin{enumerate}
	\item $W_-(t)$ scatters for negative time direction and
	$\lim_{t\to-\I} \norm{W_-(t)}_{\dot{H}^1(\R^d)}=E_1$;
	\item $W_-(t)$ converges to $W$ exponentially as $t\to\I$, that is,
	there exists positive constants $c$ and $C$ such that
	\[
		\norm{W_-(t) - W}_{\dot{H}^1(\R^d)} \le C e^{-ct}
	\]
	for all $t\ge0$. In particular, $\lim_{t\to\I} \norm{W_-(t)}_{\dot{H}^1(\R^d)}=E_2$.
\end{enumerate}
Furthermore, there is no solution $u(t)$ which does not scatter for positive time direction and
which attains $E_1$ at some finite time. That is,
if $\norm{u(t_0)}_{\dot{H}^1}=E_1$ for some $t_0 \in I_{\max}(u)$ then $u(t)$ scatters for both time directions.
\end{theorem}
% where $W(x) = (1/(1+\frac{2|x|^2}{d-2}))^{1-\frac{d}2}$ is a ground state
% for this case. $W$ solves $-\Delta W = W^{\frac{d+2}{d-2}}$.
% The above equality is essentially due to Killip and Visan \cite{KV}
% and Dodson \cite{D}.
The above theorem follows by summarizing several previous results
\cite{Do2,DM,KM,KV,LZ}.
We give a proof in Appendix \ref{sec:energycrit}.
\begin{remark}
In the 3d cubic case ($d=3$ and $\alpha=1$, $\dot{H}^{1/2}$-critical),
characterizations for $E_1(1,H^1(\R^3),\ell)$ and $E_2(1,H^1(\R^3),\ell)$ 
similar to Theorem \ref{thm:previous2} can be obtained from results in \cite{AN,DHR,DR,HR},
where $\ell(f):= \norm{f}_{L^2}^{1/2} \norm{\nabla f}^{1/2}_{L^2}$.
It would be possible to extend the result to whole inter-critical cases, i.e., between mass-critical 
and energy-critical case.
However, we do not pursue it any more.
\end{remark}

\subsection{Mass-subcritical case and weighted $L^2$ spaces}
Now, we turn to the mass-subcritical case $\alpha < 2/d$, which is the main interest of this paper.
It is well-known that global well-posedness holds in $L^2$ or $H^1$.
These are natural spaces in which the conserved quantities
make sense.
However, %in the mass-subcritical case $\alpha < 2/d$,
these spaces are not suitable to consider time global behavior because
a scaling argument shows that
$E_2(\alpha , L^2) = E_2(\alpha, H^1) = 0$.
Then, we need some other space $X$ to have $E_1>0$.
The main purpose of the paper is 
to see that a hat-Morrey space %$X=\hat{M}^{\alpha}_{q,r}$
is a good candidate for this kind of analysis.

Before this, we briefly recall some previous results in weighted $L^2$ spaces.
Weighted $L^2$ spaces are
frequently used for the analysis of the mass-subcritical case
and are spaces in which small data scattering holds.
In this paper, we consider the following two weighted spaces
\[
	 \F H^1:= \{ f \in \mathcal{S}' \ |\ \F f \in H^1 \}, \quad
	\F \dot{H}^{s}:= \{ f \in \mathcal{S}' \ |\ \F f \in L^{\frac{2d}{d-2s}} \cap \dot{H}^{s} \},
\]
where $0<s<d/2$ and $\F$ stands for the Fourier transform with respect to the 
space variable.

To consider \eqref{eq:NLS} in a weighted $L^2$ space, we generalize the concept of solution
by introducing an \emph{interaction variable} $v(t):=e^{-it\Delta} u(t)$.
Let $X=\F H^1$ or $X=\F \dot{H}^{s}$.
Notice that $e^{it\Delta}$ ($t\neq0$) is not a bounded operator on $X$ any more.
For given $u_0$ such that $v_0 = e^{-it_0 \Delta} u_0 \in X$, we say $u(t,x) \in I \times \R^d
\to \C$ is a solution to \eqref{eq:NLS} on an interval $I\ni t_0$ if
 $v(t)$ belongs to  $C(I, X)$ and satisfies
\begin{equation*}
	v(t) = v_0 +i \int^t_{t_0}
	e^{-is\Delta} [ |e^{is\Delta} v(s)|^{2\alpha} (e^{is\Delta} v(s))] ds \IN X
\end{equation*}
for all $t\in I$. The continuous dependence is also defined as a continuity 
of $v_0 \mapsto v(t)$ as in the previous case.
% The modification is necessary because
It is worth mentioning that 
 it may happen that $u(0) \in X$ but $u(t) \not\in X$ for all $t\neq0$.
In particular, this implies that a time translation symmetry breaks down.

\begin{remark}\label{rmk:modifiedsol}
Remark that if $X$ is so that $e^{it\Delta}$ becomes a linear bounded operator on $X$
and that $e^{it\Delta}$ converges strongly to identity on $X$ as $|t|\to 0$,
then modified notion of well-posedness coincides with the original one.
Indeed, $v(t)=e^{-it\Delta} u(t) \in C(I,X)$ is equivalent to $u(t) \in C(I,X)$ and 
the above modified integral formula is equivalent to \eqref{eq:iNLS}.
In this sense, this modified formulation is a generalization.
\end{remark}

% As mentioned above, small data scattering holds in the weighted spaces.
% This relates with the positivity of the first , where the definition of $E_1$ is slightly different;
We then modify the definition of $E_1$ slightly; fix $t_0=0$ and
\[
	\tilde{E}_1(\alpha,X,\ell):= 
	\inf
	\left\{  \ell(u(0)) \ \middle|\ 
	\begin{aligned}
	& u(t):  \text{solution to \eqref{eq:NLS} that does not}\\
	& \text{scatter for positive time direction.}
	\end{aligned}
	\right\}.
\]
Namely, we measure size of solutions at a specific time because time translation symmetry is broken.
Similarly, the infimum $E_2$ associated with the weighted $L^2$ spaces 
are defined in  a modified form
\[
	\tilde{E}_2= \inf \left\{  
	\limsup_{t \uparrow T_{\max}(u)} \ell(e^{-it\Delta}u(t)) \ \middle|\ 
	\begin{aligned}
	& u(t):  \text{solution to \eqref{eq:NLS} that does not}\\
	& \text{scatter for positive time direction.}
	\end{aligned}
	\right\}.
\]
\begin{remark}
By Remark \ref{rmk:modifiedsol}, we can introduce $\tilde{E}_1$ and $\tilde{E}_2$ 
under the original definition of well-posedness.
In that case, we have $\tilde{E}_1=E_1$ thanks to the time translation symmetry.
Further, if $\ell$ is invariant under $e^{it\Delta}$ then $\tilde{E}_2=E_2$.
\end{remark}

It is known that the weighted spaces are the spaces in which 
 small data scattering holds, i.e., $\tilde{E}_1(\alpha,X)>0$ (see \cite{GOV,NO,M2} and references therein).
Further, in \cite{M1,M2}, existence of a minimizer
to $\tilde{E}_1$ is shown for suitable size functions.
However, it will turn out soon that $\tilde{E}_2$ is not finite.
It would suggest that the weighted $L^2$ spaces are not so good frameworks for
the second minimization.

% the weighted $L^2$ spaces are not suitable for 
% the second minimization problems.
%\begin{remark}\label{rmk:modifiedsol}
% The above notion well-posedness is slightly modified.
% can be replaced by a usual Duhamel formula.
% The reason for modification is that it allows us
% to treat the case where $X$ is a weighted $L^2$ space, which 
% is often used in the mass-subcrtical case $\alpha<2/d$,
% in a unified way

Let us now introduce precise results in the weighted $L^2$ framework.
Let $s_c:= d/2 - 1/\alpha$. Remark that $s_c<0$ if and only if $\alpha < 2/d$.
\begin{theorem}[mass-subcritical case I, \cite{M1,M2}]\label{thm:previous3}
Let $d\ge1$ and $\max(1/d, 2/(d+2)) < \alpha < 2/d$.
Let $X= \F H^1$
and
$$\ell(f) = \ell_{\F H^1}(f) := \norm{|x| f}_{L^2}^{-s_c} \norm{f}_{L^2}^{1+s_c}.$$
Then, $0<\tilde{E}_1 < \ell_{\F H^1} (Q) < \tilde{E}_2=\I$.
Further, there exists a solution $u_c(t)$ that attains $\tilde{E}_1$, i.e., $\ell_{\F H^1}(u_c(0))=\tilde{E}_1$.
\end{theorem}

Notice that $\tilde{E}_2=\I$ implies that 
even a minimizer $u_c(t)$ to $\tilde{E}_1$ satisfies
$\sup_{t\ge0} \norm{ e^{-it\Delta} u_c(t) }_{\F H^1}=\I$.
The proof of $\tilde{E}_2=\I$ is immediate from time decay property that
the boundedness of $\ell_{\F H^1}( e^{-it\Delta} u(t))$ gives.
The infiniteness of $\tilde{E}_2$ can be understood also as a reflection of the fact that
$\ell_{\F H^1}( e^{-it\Delta} u(t))$ is a scattering-solution-oriented value.
% it gives a penalty if $U(-t)u(t)$ ``disperses'' 
For example, even the ground state solution is not bounded;
\[
	\sup_{t\in\R} \ell_{\F H^1} ( e^{-it\Delta} Q_{\alpha}(t) )
	= \I.
\] 
Intuitively, if $u(t)$ does not scatter then $e^{-it\Delta}$ gives some uncancelled ``dispersion effect,''
which penalized by the weighted $L^2$ norm.
% The boundedness with respect to $\ell_{\F H^1}(U(-t)u(t))$ is a strong assumption to 
% exclude nonscattering solutions.

We have a similar result in the case $X=\F \dot{H}^{|s_c|}$.
\begin{theorem}[mass-subcritical case II, \cite{M2}]\label{thm:previous4}
Let $d\ge1$ and $\max(1/d, 2/(d+2)) < \alpha < 2/d$.
Let $X= \F \dot{H}^{|s_c|}$
and $\ell(\cdot) = \norm{\cdot}_{\F\dot{H}^{|s_c|}}$.
Then, $0<\tilde{E}_1 < \ell (Q) < \tilde{E}_2=\I$.
Further, there exists a solution $u_c(t)$ that attains $\tilde{E}_1$, i.e., 
$\ell(u_c(0))=\tilde{E}_1$.
\end{theorem}
In \cite{M2}, the identity $\tilde{E}_2=\I$ is not shown.
A proof will appear elsewhere.
% Although the size of ground state is infinite also in this case;
% \[
% 	\sup_{t\in\R} \norm{ e^{-it\Delta} Q_{\alpha}(t) }_{\F H^1}
% 	= \I,
% \] 
% it is not clear whether $\tilde{E}_2(\alpha,\F \dot{H}^{-s_c} ) = \I$ or not.

 % is often  used as a state space.
% In this paper,  the weighted spaces which we work with are only the following two;
% \[
% 	 \F H^1:= \{ f \in \mathcal{S}' \ |\ \F f \in H^1 \}, \quad
% 	\F \dot{H}^{s}:= \{ f \in \mathcal{S}' \ |\ \F f \in L^{\frac{2d}{d-2s}} \cap \dot{H}^{s} \},
% \]
% One may not expect $u(t) \in X$ in general. 
% Remark that $e^{it\Delta}$ is not bounded operator on the weighted $L^2$
% spaces.

% Let $s_c:= d/2 - 1/\alpha (< 0)$.
% For $\max(1/d, 2/(d+2)) < \alpha < 2/d$, the followings are known.
% \begin{itemize}
% \item Let $X= \F H^1:= \{ f \in \mathcal{S}' \ |\ \F f \in H^1 \}$, where
% $\F$ stands for the Fourier transform.
% We take a size function 
% $$\ell_{\F H^1}(f) = \norm{|x| f}_{L^2}^{-s_c} \norm{f}_{L^2}^{1+s_c}.$$
% Then, it follows that 
% \[
% 	0 < E_1(\alpha , \F H^1, \ell_{\F H^1}) < \ell_{\F H^1} (Q_\alpha)
% 	< E_2(\alpha , \F H^1, \ell_{\F H^1}) = \I.
% \]
% The second inequality is due to \cite{M1,M2}.
% The last equality is Theorem.
% \item In \cite{M2}, it is also shown that
% \[
% 	0 < E_1 (\alpha, \F \dot{H}^{-s_c}) < \norm{Q_\alpha}_{\F \dot{H}^{-s_c}}.
% \]
% It is not clear whether $E_2 (\alpha, \F \dot{H}^{-s_c})=\I$.
% \end{itemize}

\subsection{Mass-subcritical case and hat-Morrey space}

Let us proceed to the main issue.
In this paper, we want to consider \eqref{eq:NLS} in mass-subcritical case
with choosing a state space $X$ and a size function $\ell$ 
so that  the both $E_1$ and $E_2$ becomes positive and finite.
As seen in the previous section, if we choose a weighted $L^2$ space as a
state space, then the finiteness of $E_2$ problem is not clear.
Hence, we will seek another space.

The conclusion is that a hat-Morrey space is a good candidate. 
% As one will see below, this is one of such spaces.
The space, introduced by Bourgain \cite{Bo1}, is used in a refinement of a Stein-Tomas estimate,
a special case of Strichartz' estimate, see \cite{Bo2,MV,CK,BV}.
The definition is as follows.

\begin{definition}
For $j\in \Z$, we let 
$$
\mathcal{D}_j:=\left\{ 2^{-j}( [0,1)^d + k)=\prod_{l=1}^N [k_l 2^{-j},(k_l+1)2^{-j})
\middle| k \in \Z^d \right\}$$
be a set of dyadic cubes with size $2^{-j}$.
Let $\mathcal{D}:= \cup_{j\in\Z} \mathcal{D}_j$. %be a set of all dyadic cubes.
For a cube $\tau \in \mathcal{D}$, we frequently use the notation
$\tau = \tau^{j}_k := 2^{-j}( [0,1)^d + k)
%\prod_{l=1}^d [k_l 2^{-j},(k_l+1)2^{-j})
$ with suitable $j\in\Z$ and $k\in \Z^d$.
Let us introduce a (generalized) Morrey norm by
\[
	\norm{f}_{{M}^{p}_{q,r}}
	:= \norm{|\tau^j_k|^{\frac1{p}- \frac1{q} } \norm{f}_{L^{q}(\tau_j^k)} }_{\ell^{r}_{j\in\Z,\,k\in\Z^d}}
\]
for $1 \le q \le p \le r \le \I$. If $r<\I$ we assume $q<p<r$.
A hat-Morrey norm is also introduced by
\[
	\norm{f}_{\hat{M}^{p}_{q,r}} := \norm{\F f}_{M^{p'}_{q',r}}
	= \norm{|\tau_k^j|^{\frac1{p'}- \frac1{q'} } \norm{\F{f}}_{L^{q'}(\tau_k^j)} }_{\ell^{r}_{j\in\Z,\,k\in\Z^d}}
\]
for $1 \le p \le q \le \I$ and $p'\le r \le \I$. Again, if $r<\I$ we assume $q'<p'<r$.
We define function spaces $M^{p}_{q,r}$ and $\hat{M}^{p}_{q,r}$ as 
sets of functions in $L^{q}_{\mathrm{loc}}(\R^d)$ and $\F L^{q'}_{\mathrm{loc}}(\R^d) $, respectively,
such that the corresponding norm is finite.
\end{definition}
\begin{remark}
We remark that $M^{p}_{q,\I}$ coincide with usual Morrey and so
that the ${M}^{p}_{q,r} $ is a generalization.
 If $r<\I$ then $M^{p}_{p,r}=\hat{M}^{p}_{p,r}  = \{0\}$.
If $r<\I$ then $M^{r}_{q,r}$ and $ M^{r'}_{q,r}$ do not contain ${\bf 1}_{\tau}(x)$ for any $\tau \in \mathcal{D}$, 
where ${\bf 1}_A(x)$ denotes the characteristic function of a set $A \subset \R^d$.
\end{remark}
The hat-Morrey space is a generalization
of a hat-Lebesgue space $\hat{L}^p = \hat{M}^p_{p,\I} = \F L^{p'}$.
It is known that some dispersive estimates, such as the Strichartz' estimates,
are naturally extended to the hat-Morrey and hat-Lebesgue spaces \cite{Bo2,HT,MS2}.
By means of these estimate, well-posedness of nonlinear Schr\"odinger
equation is established in \cite{G,HT,MS2} for $d=1$.
In \cite{G1,GV,MS1}, well-posedness of KdV-type equations are studied
in hat-Lebesgue spaces.

In \cite{MS2}, the first problem $E_1$ for generalized KdV equation is considered
and they show that existence of a special solution which attains this value, in
a suitable sense, under an assumption on a relation between the values of $E_1$
for generalized KdV equation and for nonlinear Schr\"odinger equation.
As one tool for obtaining the result,
well-posedness of \eqref{eq:NLS} in a hat-Morrey space is established for $d=1$.
% \cite{MS2}において、gKdV方程式に関して、
% $E_1$をある意味で達成する元が存在することを示した。
% 空間1次元の場合の hat-Morrey 空間における\eqref{eq:NLS} の適切性は、その際に、
% 道具の一つとして用いられている。
We first generalize this well-posedness for higher dimensions.
% まずは、\cite{MS2}の適切性の結果を一般次元に拡張する。
Although our main interest is mass-subcritical case $\alpha <2/d$,
one does not need this ``restriction'' for a well-posedness result.

\begin{theorem}\label{thm:LWP}
Let $d\ge 1$ and $ \frac2d\cdot\frac1{1+\frac2{d(d+3)}} < \alpha < \frac2d\cdot\frac2{1-\frac2{d(d+3)}}$.
\begin{enumerate}
\item The equation \eqref{eq:NLS} is locally well-posed in $\hat{L}^{d\alpha}$.
\item The equation \eqref{eq:NLS} is locally well-posed in $\hat{M}^{d\alpha}_{q,r}$,
provided
\[
	d \alpha < q< \(1+\frac{2}{d(d+3)}\){d\alpha},
\]
and
\[
	(d\alpha)' < r \le ((d+2)\alpha)^*,
\]
where $a^*=\min (a, 2a/(a-2))$
\end{enumerate}
\end{theorem}

Let us now turn to the minimization problems.
For these problems, we assume that
\begin{equation}\label{asmp:alpha}
	\frac2d \cdot \frac1{1+\frac2{d(d+3)}} < \alpha < \frac2d, \quad
	(d\alpha)' < r < ((d+2)\alpha)^*
\end{equation}
and take $X=\hat{M}^{d\alpha}_{2,r}$ as a state space.
Notice that the first assumption of \eqref{asmp:alpha} is necessary to take $q=2$ in Theorem \ref{thm:LWP} (2),
and the second assumption is exclusion of the end point $r=((d+2)\alpha)^*$, which is necessary for 
concentration-compactness-type arguments.
We remark that $\hat{M}^{d\alpha}_{2,r}$ is one of the spaces on which
$\{e^{it\Delta}\}_t$ forms a one-parameter group of linear bijective isometries,
and that $e^{it\Delta}$ converges strongly to the identity operator 
as $|t|\to0$.

Introduce a size function as follows:
\begin{equation}\label{eq:M_size}
	\ell_{r}(f) := \inf_{\xi \in \R^d} \norm{e^{-i(\cdot )\cdot \xi} f(\cdot)}_{\hat{M}^{d\alpha}_{2,r}}.
\end{equation}
One sees that $\ell_r$ is an equivalent quasi-norm on $\hat{M}^{d\alpha}_{2,r}$ (see Remark \ref{rmk:quasi-norm}).
Since $\ell_r$ satisfies \eqref{eq:conj}, the meaning of $E_1$ and $E_2$
can be strengthen.
An important fact is that the size of ground state is bounded in time
\[
	\ell_r(e^{it} Q) = \ell_r( Q) <\I,
\]
which gives us a desired a priori bound on the second minimization value,
\[
	E_2(\alpha, \hat{M}^{d\alpha}_{2,r}, \ell_r) \le \ell_r(Q) <\I.
\]
\begin{remark}[defocusing case]
In this paper, we only consider focusing equations.
However, the focusing nation is only used for the above a priori bound on $E_2$.
For the defoucsing case, 
if we assume $E_2<\I$ then the same results as in Theorems \ref{thm:main2} and \ref{thm:main3} are obtained
by the same proof.
Further, if we obtain some contradiction from the conclusions in Theorems \ref{thm:main2} and \ref{thm:main3}
then we have $E_2=\I$.
It is needless to say that Theorem \ref{thm:LWP} holds without the boundedness assumption.
\end{remark}

Now, let us introduce the main results of this paper.
\begin{theorem}\label{thm:main1}
Let $d\ge1$ and suppose \eqref{asmp:alpha}.
Then, $0< E_1(\alpha, \hat{M}^{d\alpha}_{2,r}, \ell_r) < \ell_r(Q)$.
\end{theorem}
This theorem says two things.
Firstly, $E_1$ is positive, that is, a small data scattering result holds in $\hat{M}^{d\alpha}_{2,r}$.
Secondly, $E_1$ is strictly smaller than the size of ground state, and so
the ground state solution is not a minimizer to the $E_1$ problem.
As for a minimizer to $E_1$, we have the following;
\begin{theorem}\label{thm:main2}
Let $d\ge1$ and suppose \eqref{asmp:alpha}.
There exists a maximal-lifespan solution $u_{E_1}(t)$ to \eqref{eq:NLS} such that
\begin{enumerate}
\item $u_{E_1}(t)$ does not scatter for positive time direction.
\item $u_{E_1}(t)$ attains $E_1$ in such a sense that one of the following holds;
\begin{enumerate}
\item $\ell(u_{E_1}(0)) = E_1$,
\item $u_{E_1}(t)$ scatters for negative time direction and $$\ell_r\(\lim_{t\to-\I} e^{-it\Delta} u_{E_1}(t)\) = E_1.$$
\end{enumerate}
\end{enumerate}
\end{theorem}
\begin{remark}
In the above theorem, validity of the case (2)-(a) implies that existence of a minimizer as in Theorems \ref{thm:previous3}
and \ref{thm:previous4}. On the other hand, (2)-(b) corresponds to the situation as in 
the energy-critical case (Theorem \ref{thm:previous2}).
Notice that in the energy-critical case, the case (2)-(a) never happens.
\end{remark}
We next state existence of a minimizer to $E_2$ of which flow
is totally bounded modulo dilations, and translations in both physical and Fourier sides.
\begin{theorem}\label{thm:main3}
Let $d\ge1$ and suppose \eqref{asmp:alpha}.
There exists a maximal-lifespan solution $u_{E_2}(t)$ to \eqref{eq:NLS} such that
\begin{enumerate}
\item $u_{E_2}(t)$ does not scatter for both time directions.
\item $\sup_{t\in I_{\max}(u_{E_2}) \cap \{t\ge 0\}} \ell_r(u_{E_2}(t))= \sup_{t\in I_{\max}(u_{E_2}) \cap \{t\le 0\}} \ell_r(u_{E_2}(t))=E_2$.
\item $u_{E_2}(t)$ is almost periodic modulo symmetry
i.e. there exist $y(t), z(t): I_{\max}(u_{E_2}) \to \R^d$,
$N(t): I_{\max}(u_{E_2}) \to 2^\Z$, and $C(\eta)>0$ such that 
\begin{multline}\label{eq:apms}
	\sup_{|w|\le \frac{N(t)}{C(\eta)}} \norm{ (e^{iw\cdot(x-y(t))}-1) u_{E_2}(t) }_{\hat{M}^{d\alpha}_{2,r}} \\
	+ \norm{\F^{-1} {\bf 1}_{|\xi-z(t)| \ge C(\eta) N(t)} \F u_{E_2}(t)}_{\hat{M}^{d\alpha}_{2,r}} \le \eta
\end{multline}
for any $\eta>0$ and for any $t\in I_{\max}(u_{E_2})$.
\end{enumerate}
\end{theorem}
\begin{remark}
The validity of \eqref{eq:apms} is equivalent to pre-compactness (or total boundedness) of the set
\[
	\left\{ 
	\frac{e^{-iN(t)^{-1}x\cdot z(t)}}{N(t)^{1/\alpha}} u_{E_2}\(t,\frac{x}{N(t)}+y(t)\) \middle|
	t \in I_{\max}
	\right\} \subset \hat{M}_{2,r}^{d\alpha}
\]
(see Theorem \ref{thm:totallybounded}, below). $y(t)$ and $z(t)$ correspond to \emph{a spacial center} and
\emph{a frequency center}, respectively.
The meaning of the smallness of the first term in the left hand side of \eqref{eq:apms} is 
close to that of $\norm{ {\bf 1}_{|x-y(t)|\ge C(\eta)/N(t)} u_{E_2}(t) }_{\hat{M}^{d\alpha}_{2,r}}$.
However, the equivalence of two smallnesses is not clear.
\end{remark}
The ground state solution is an example that does not scatter
and is almost periodic modulo symmetry.
Namely, the ground state solution satisfies
the first and third property of Theorem \ref{thm:main3}. 
In mass-critical and energy-critical cases, $E_2$ coincides with the size of ground state (Theorems \ref{thm:previous1} and \ref{thm:previous2}).
In the proofs of these theorems, a solution with almost periodicity modulo symmetry  
plays a crucial role.
The main step of the proof there is to derive a contradiction from the assumption
that $E_2$ is less than the size of ground state via a precise analysis on
almost-periodic-modulo-symmetry
solutions similar to that given in Theorem \ref{thm:main3}.
In view of these facts, one conjecture in our case would be
$E_2(\alpha, \hat{M}^{d\alpha}_{2,r} , \ell_r) = \ell_r(Q)$.
This equality insists that every nonscatter solution $u(t)$ (even $u_{E_1}(t)$ given in Theorem \ref{thm:main2}) satisfies
\[
	\limsup_{t \uparrow T_{\max}} \ell_r (u(t)) \ge \ell_r (Q),
\]
which seems reasonable from the view point of the soliton resolution conjecture.
However, it is not clear even if we believe that there is no almost-periodic-modulo-symmetry solution
``smaller'' than the ground states, as in the mass-critical or energy-critical cases.
One negative reason is that
 we do not know whether the size function $\ell_r$ is chosen well enough to capture such phenomena.
An appropriate choice of a size function (for the above conclusion)
would be given by a variational characterization of $Q$,
which is not known in $\hat{M}^{d\alpha}_{2,r}$.

The rest of the paper is organized as follows.
In Section \ref{sec:preliminaries}, we introduce basic facts and several tools.
In particular, Theorems \ref{thm:LWP} and \ref{thm:main1} are established in this section.
Section \ref{sec:compactness} is devoted to the study of a compactness result,
a linear profile decomposition (Theorem \ref{thm:pd}).
Then, we turn to the minimization problems. We prove Theorem \ref{thm:main3}
in Section \ref{sec:main3}, and Theorem \ref{thm:main2} in Section \ref{sec:main2}.

\section{Preliminaries}\label{sec:preliminaries}

\subsection{Strichartz' estimates}
Strichartz' estimate is a key tool for well-posedness theory.
The estimates is naturally extended in terms of
 hat-Lebesgue spaces and hat-Morrey spaces.
In one dimensional case, this kind of generalization is established in \cite{HT}.

We first introduce Strichartz' estimate in hat-Morrey space.
\begin{proposition}\label{prop:Strichartz}
If 
\[
	% 	r > \frac{(d+3)}{(d+1)} \max(2,q)
	\frac1p < \frac{d+1}{d+3} \min\(\frac12,\frac1q\)
\]
then
\[
	\norm{e^{it\Delta}f}_{L^{p}_{t,x}(\R^{1+d})}
	\le C
	\norm{f}_{\hat{M}^{\frac{dp}{d+2}}_{q,p^*}},
\]
where $p^* = \min(p, \frac{2p}{p-2})$.
\end{proposition}
The proof is similar to \cite[Theorem 1.2]{BV} which corresponds to the case $p=2(d+2)/d$.
The condition of the proposition is necessary because the proof is based on 
the following bilinear restriction estimate.
\begin{proposition}[\cite{T}]
Let $Q,Q'$ be cubes of sidelength $1$ in $\R^d$ such that
\[
	\min \{ |x-y| \ |\ x \in Q, y \in Q'\} \sim 1
\]
and let $f,g$ be functions such that $\hat{f}$ and $\hat{g}$ are supported in $Q$ and $Q'$, respectively.
Then, for all $p>\frac{2(d+3)}{d+1}$ and all $q$ such that
\[
	\frac1{q} > \frac{d+3}{d+1}\frac1p,
\] 
it holds that
\[
	\norm{(e^{it\Delta}f )(e^{it\Delta}g)}_{L^{p/2}_{t,x}} \le C \norm{f}_{\hat{L}^q} \norm{g}_{\hat{L}^q}
\]
with a positive constant $C$ independent of $Q$, $Q'$, $f$, and $g$.
\end{proposition}

Let us proceed to Strichatz' estimates in hat-Lebesgue space.
We have an embedding between hat-Morrey and hat-Lebesgue spaces. 
\begin{proposition}\label{prop:gm_embedding}
We have the following embeddings.
\begin{itemize}
\item If $1 \le q < p < r \le \I $ then $L^p \hookrightarrow M^p_{q,r}$.
\item If $1\le q' <  p' <r \le \I $ then $\hat{L}^p \hookrightarrow \hat{M}^{p}_{q,r}$.
\end{itemize}
\end{proposition}
For the proof, see \cite{Bo2,MVV,BV,MS2}.
The only one dimensional case is considered there, however the modification is obvious.
\begin{proposition}\label{prop:StrichartzhL}
Let $d\ge 1$ and let $1 \le p,q,r \le \I$ be such that
\[
	\frac2p + \frac{d}q = \frac{d}r.
\]
Assume that a triplet $(p,q,d)$ is either $(p,q,d)=(\I,\I,2)$,  $(p,q,d)=(2,\frac{2d}{d-2},d)$ with $d\ge 3$,
or satisfies
\[
	0 \le \frac1q \le \frac12,\quad 0\le \frac1p < \frac12 -\frac1q ,\quad \frac1p \le \frac14, 	
\] 
if $d=1$;
\[
	0 < \frac1q \le \frac12,\quad 0\le \frac1p < \min \( -\frac2{3q} + \frac12 , -\frac{3}{2q} + \frac34 \)
\] 
if $d=2$; and
\[
	0 \le \frac1q \le \frac12,\quad 0\le \frac1p \le \frac{d}{d-2}\frac1q,
\] 
\[
	\frac1p <  -\frac{d}{3}\(\frac1q -\frac{d+1}{2(d+3)}\) + \frac{d+1}{2(d+3)},
\]
\[
	\frac1p < -\frac{d+1}{2}\(\frac1q -\frac{d+1}{2(d+3)}\) + \frac{d+1}{2(d+3)}
\]
if $d\ge 3$. 
Then, it holds that
\begin{equation}\label{eq:hL_hom}
	\norm{e^{it\Delta} g}_{L^{p}_{t}L^q_x(\R^{1+d})}
	\le C \norm{g}_{\hat{L}^{r}(\R^d)}
\end{equation}
and that
\begin{equation}\label{eq:hL_dual}
	\norm{\int_\R e^{-it'\Delta} F(t') dt'}_{\hat{L}^{r'}(\R^d)}
	\le C\norm{F}_{L^{p'}_t L^{q'}_x(\R^{1+d})}
\end{equation}
for some positive constant $C$.
\end{proposition}
\begin{proof}
The second estimate \eqref{eq:hL_dual} follows from \eqref{eq:hL_hom} by duality.
So, let us restrict our attention to the first estimate.
The one dimensional case is due to Hyakuna and Tsutsumi \cite{HT}.
Let us consider the multi dimensional case.
The diagonal case $p=q > \frac{2(d+3)}{d+1}$ is an immediate consequence of Propositions \ref{prop:Strichartz}
and \ref{prop:gm_embedding}.
The off-diagonal case follows by interpolating the diagonal case and well-known $r=2$ cases.
\end{proof}

As for the Strichartz estimate for the hat-Lebesgue space,
one can obtain a dual estimate \eqref{eq:hL_dual}.
Note that a dual space of a hat-Morrey space is not clear;
only a pre-dual space is characterized (see Section \ref{subsec:predual}).
% \begin{corollary}
% Let $\alpha \in (\frac{2(d+3)}{d^2+3d+2},\frac{2(d+3)}{d^2+3d-2})$. 
% Let $t_0\in\R$ and $\R \supset I \ni t_0$.
% Let $\Phi[F](t,x) = \int_{t_0}^t e^{i(t-t')\Delta} F(t') dt'$.
% Then, 
% \[
% 	\norm{\Phi[F]}_{L^\I(I,\hat{L}^{d\alpha})} + \norm{\Phi[F]}_{L^{(d+2)\alpha}_{t,x}(I\times \R^d)} 
% 	\le C \norm{F}_{L^{\frac{(d+2)\alpha}{2\alpha+1}}_{t,x}(I\times \R^d)}
% \]
% for any $F\in L^{\frac{(d+2)\alpha}{2\alpha+1}}_{t,x}(I\times \R^d)$.
% \end{corollary}

Let us now proceed to inhomogeneous estimates.
For $t_0\in\R$ and an interval $I \subset \R$ such that $I \ni t_0$,
let $$\Phi[F](t,x) = \int_{t_0}^t e^{i(t-t')\Delta} F(t') dt'. $$
The first estimate is as follows.
\begin{corollary}\label{cor:hom_dual}
If
\[
	1\le r < \frac{2}{1-\frac{2}{d(d+3)}} 
\]
then
\[
	\norm{\Phi[F]}_{L^\I(I,\hat{L}^{r})} \le C \norm{F}_{L^{\frac{(d+2)r}{2r+d}}_{t,x}(I\times \R^d)}
\]
for any $F\in L^{\frac{(d+2)r}{2r+d}}_{t,x}(I\times \R^d)$.
\end{corollary}
This follows \eqref{eq:hL_dual} and the fact that $e^{it\Delta}$ is an isometry on $\hat{L}^{r}$.
The next one is an inhomogeneous estimate for non-admissible pairs. 
\begin{proposition}[\cite{Ka,Fo,Ko,Vi}]\label{prop:inhom}
Let $1\le p_1,p_2,q_1,q_2 \le \I$. The estimate
\[
	\norm{\Phi[F]}_{L^{p_1}_t(I,L^{q_1}_x)} \le C \norm{F}_{L^{p_2'}_t(I,L^{q_2'}_x)}
\]
holds true if the following three assumptions are fulfilled:
\begin{itemize}
\item (acceptability) For $i=1,2$,
\[
	p_i < d \( \frac12- \frac1{q_i}\) \quad \text{ or } \quad (p_i,q_i) =(\I,2);
\]
\item (scale condition)
\[
	\sum_{i=1,2} \(\frac2{p_i} + \frac{d}{q_i}\) = d;
\]
\item (additional assumption) $q_1,q_2 < \I$ if $d=2$, and
\begin{equation}\label{eq:Kcond}
	\frac1{p_1} + \frac1{p_2} < 1, \quad \frac{d-2}{d} \le \frac{q_1}{q_2} \le \frac{d}{d-2}.
\end{equation}
if $d\ge3$
\end{itemize}
\end{proposition}
Further restricting to the case $p_1=q_1$ and $p_2=q_2$, we obtain
\begin{equation}\label{eq:inhom}
	\norm{\Phi[F]}_{L^{(d+2)\alpha}_{t,x}(I \times \R^d)} \le C \norm{F}_{L^{\frac{(d+2)\alpha}{2\alpha+1}}_{t,x}(I\times \R^d)},
\end{equation}
% if $ \frac2d \cdot \frac{d+1}{d+2}< \alpha < \frac2d \cdot \frac{d+1}{d} $.
provided
\begin{equation}\label{eq:LWPcond}
 \frac2d \cdot \frac{1}{1+\frac1{d+1}}< \alpha < \frac2d \cdot \frac{1}{1-\frac1{1+d}}.
\end{equation}
This estimate is sufficient for our purpose.
We remark that the condition \eqref{eq:LWPcond} comes from the acceptability.
It is known that the condition \eqref{eq:Kcond} can be relaxed slightly (see \cite{Fo,Ko,Vi}).
% We use this version for the sake of simplicity since this is sufficient for the present purpose.
However, we do not recall it since, under the diagonal assumption,
the condition \eqref{eq:Kcond} is already weaker than \eqref{eq:LWPcond}.

\subsection{Well-posedness results}
With the Strihcartz' estimates, we obtain local well-posedness.
The following norm plays an important role in the well-posedness theory.
\begin{definition}[Scattering norm]
For an interval $I\subset \R$ and function $u(t,x):I \times \R^d \to \C$, we 
define a \emph{scattering norm} by
\[
	S_I(u) := \norm{u}_{L^{(d+2)\alpha}_{t,x}(I\times \R^d)}.
\]
\end{definition}
\begin{lemma}\label{lem:LWP}
Let $d\ge 1$ and $\alpha>0$ satisfy \eqref{eq:LWPcond}.
Then, there exists a constant $\delta=\delta(d)>0$ such that
if a function $u_0 \in \mathcal{S}'$ and an interval $I \subset \R$
satisfy $t_0 \in I$ and
\[
% 	\norm{e^{i(t-t_0)\Delta}u_0}_{L^{{(d+2)\alpha}}_{t,x}(I\times \R^d)} \le \delta
	S_I(e^{i(t-t_0)\Delta}u_0) \le \delta
\]
then there exists a unique solution $u(t): I \times \R^d \to \C$
to \eqref{eq:NLS} such that
\[
% 	\norm{u}_{L^{{(d+2)\alpha}}_{t,x}(I\times \R^d)} \le 2 \norm{e^{i(t-t_0)\Delta}u_0}_{L^{{(d+2)\alpha}}_{t,x}(I\times \R^d)}.
	S_I(u) \le 2 S_I(e^{i(t-t_0)\Delta}u_0).
\]
\end{lemma}
\begin{proof}
By the non-admissible Strichartz' estimate \eqref{eq:inhom},
\begin{align*}
	\norm{\Phi[|u|^{2\alpha}u]}_{L^{(d+2)\alpha}_{t,x}(I \times \R^d)}
	&{}\le C \norm{|u|^{2\alpha}u}_{L^{\frac{(d+2)\alpha}{(2\alpha + 1)}}_{t,x}(I \times \R^d)} \\
	&{}= C \norm{u}_{L^{(d+2)\alpha}_{t,x}(I \times \R^d)}^{2\alpha + 1}
\end{align*}
and similarly,
\begin{multline*}
	\norm{\Phi[|u_1|^{2\alpha}u_1] - \Phi[|u_2|^{2\alpha}u_2]}_{L^{(d+2)\alpha}_{t,x}(I \times \R^d)} \\
	\le C (\norm{u_1}_{L^{(d+2)\alpha}_{t,x}(I \times \R^d)} + \norm{u_2}_{L^{(d+2)\alpha}_{t,x}(I \times \R^d)})^{2\alpha}
	\norm{u_1-u_2}_{L^{(d+2)\alpha}_{t,x}(I \times \R^d)}
\end{multline*}
By a standard fixed point argument, we obtain a unique solution $u \in L^{(d+2)\alpha}_{t,x}(I \times \R^d)$ of the 
equation $u(t) = e^{i(t-t_0)\Delta}u_0 + i \Phi[|u|^{2\alpha}u]$ if $\delta$ is sufficiently small.
\end{proof}

\begin{proof}[Proof of Theorem \ref{thm:LWP}]
Let $t_0=0$ for simplicity.
Notice that the assumption
\[
	\frac2d\cdot\frac{1}{1+\frac2{d(d+3)}} < \alpha < \frac2d\cdot\frac{1}{1-\frac2{d(d+3)}}
\]
is stronger than \eqref{eq:LWPcond}.
\smallskip

\noindent{\bf Step 1.}
Since $\alpha > \frac2d\cdot\frac{1}{1+\frac2{d(d+3)}}$,
if $u_0 \in \hat{L}^{d\alpha}$ then
\[
% 	\norm{e^{it\Delta} u_0}_{L^{(d+2)\alpha}_{t,x}(\R\times \R^d)} \le C
	S_\R(e^{it\Delta} u_0) \le C
	\norm{u_0}_{\hat{L}^{d\alpha}}
\]
by Proposition \ref{prop:StrichartzhL}.
The same conclusion is deduced from Proposition \ref{prop:Strichartz} for
$u_0 \in \hat{M}^{d\alpha}_{q,r}$ if $\alpha > \frac2d\cdot\frac{1}{1+\frac2{d(d+3)}}$,
$q< (1+\frac{2}{d(d+3)})d\alpha $, and $r \le ((d+2)\alpha)^*$.

\noindent{\bf Step 2.} Since $e^{it\Delta} u_0 \in L^{(d+2)\alpha}_{t,x}(\R\times \R^d)$,
there exists an interval $I \ni 0$ such that
\[
% 	\norm{e^{it\Delta} u_0}_{L^{(d+2)\alpha}_{t,x}(I\times \R^d)} 
	S_I(e^{it\Delta} u_0)
	\le \delta,
\]
where $\delta$ is the number given in Lemma \ref{lem:LWP}.
Then, the lemma gives a unique solution $u(t)\in L^{(d+2)\alpha}_{t,x} (I \times \R^{d})$.

\noindent{\bf Step 3.}
Let us show that $u(t)$ possesses the desired continuity.
By Corollary \ref{cor:hom_dual},
the Duhamel term obeys
\[
	\norm{\Phi[|u|^{2\alpha}u]}_{L^\I(I,\hat{L}^{d\alpha})}
	\le C \norm{u}_{L^{(d+2)\alpha}_{t,x}(I\times \R^d)}^{2\alpha+1}
\]
as long as $\alpha<\frac2d\cdot\frac{1}{1-\frac2{d(d+3)}}$.
Thus, $u(t)-e^{it\Delta}u_0= i \Phi[|u|^{2\alpha}u] \in C(I, \hat{L}^{d\alpha})$.

Obviously, if $u_0 \in \hat{L}^{d\alpha}$ then
the linear part satisfies $e^{it\Delta} u_0 \in C(I, \hat{L}^{d\alpha})$ and so $u(t) \in C(I.\hat{L}^{d\alpha})$.
On the other hand, if $u_0 \in \hat{M}^{d\alpha}_{q,r}$ then
$e^{it\Delta} u_0 \in C(I, \hat{M}^{d\alpha}_{q,r})$. Since $\hat{L}^{d\alpha} \hookrightarrow \hat{M}^{d\alpha}_{q,r}$ 
follows from the assumptions $q > d\alpha$ and $r > (d\alpha)'$,
we conclude that $u(t) \in C(I, \hat{M}^{d\alpha}_{q,r})$.

Continuous dependence on the data follows by a standard argument.
\end{proof}

\begin{remark}
Let us make a comment on the assumption of Theorem \ref{thm:LWP} on $\alpha$.
The lower bound of $\alpha$ is used in the existence part.
Without this assumption, neither $u_0\in \hat{L}^{d\alpha}$ nor $u_0 \in \hat{M}^{d\alpha}_{q,r}$ 
is sufficient to obtain a solution $u(t) \in L^{(d+2)\alpha}(I \times \R^d)$.
On the other hand, the upper bound of $\alpha$ is used in the persistence-of-regularity part.
Without this assumption, the obtained solution $u(t) \in L^{(d+2)\alpha}(I \times \R^d)$
does not necessarily belong to $C(I,\hat{L}^{d\alpha})$ or $C(I,\hat{M}^{d\alpha}_{q,r})$.
Remark that one can obtain a solution under a weaker assumption \eqref{eq:LWPcond} (and the lower bound on $\alpha$).
\end{remark}
\begin{remark}
If $u_0 \in \hat{L}^{d\alpha}$ then the solution $u(t)$ belongs to some off-diagonal space-time function space
$L^p_{t}(I,L^q_x(\R^d))$, $p\neq q$.
Indeed, $|u|^{2\alpha} u \in L^{\frac{(d+2)\alpha}{2\alpha+1}}_{t,x}(I \times \R^d)$
implies the Duhamel term $i\Phi[|u|^{2\alpha} u]$ belongs to $L^p_{t}(I,L^q_x(\R^d))$
for suitable $(p,q)$ satisfying the assumption of
non-admissible Strichartz's estimate (Proposition \ref{prop:inhom}).
Similarly, the linear part $e^{it\Delta} u_0$ also belongs to some $L^p_{t}(I,L^q_x(\R^d))$
by Proposition \ref{prop:StrichartzhL}. Hence, the solution belongs to the intersection.
A off-diagonal estimate, or a mixed-norm estimate,  similar to Proposition \ref{prop:StrichartzhL}
would be possible also for the hat-Morey space.
However, we do not pursue it in this paper.
\end{remark}

In the rest of this section, we let $X$ be $\hat{L}^{d\alpha}$ or $\hat{M}^{d\alpha}_{q,r}$
that satisfies the assumption of Theorem \ref{thm:LWP}.
We next characterize finite time blowup and scattering in terms of the scattering norm of the solution.
For the proof, see \cite{MS1}.
\begin{proposition}[Blowup and scattering criterion]\label{prop:criterion}
% Let $d\ge 1$ and $\alpha>0$ satisfy \eqref{eq:LWPcond}.
Let $u(t)$ be an $X$-solution of \eqref{eq:NLS} given in Theorem \ref{thm:LWP}. 
\begin{itemize}\item
If $T_{\mathrm{max}}<\I$ then $S_{[0,T)}(u) \to \I$
as $T \uparrow T_{\mathrm{max}}$.
\item
The solution scatters for positive time direction if and only if
$T_{\mathrm{max}}=+\I$ and $S_{[0,\I)}(u) <\I$.
\end{itemize}
A similar statements are true for negative time direction.
\end{proposition}
An immediate consequence is  small data scattering,
\begin{theorem}[small data scattering]
Let $u(t)$ be a nonzero $X$-solution of \eqref{eq:NLS} given in Theorem \ref{thm:LWP}. 
If $S_{\R}(e^{i(t-t_0)\Delta}u_0) \le \delta$ then $u(t)$ scatters in $X$ for both time directions,
where $\delta$ is the constant given in Lemma \ref{lem:LWP}.
\end{theorem}
The first part of Theorem \ref{thm:main1} is a rephrase of this theorem.
One has also a nonscattering result for solution with non-positive energy.
\begin{theorem}
Let $u(t)$ be a nonzero $X$-solution of \eqref{eq:NLS} given in Theorem \ref{thm:LWP}. 
We further assume that $u_0 \in H^1$ and $E[u_0]:= \frac12 \norm{\nabla u_0}_{L^2} - \frac1{p+1} \norm{u_0}_{L^{2\alpha+1}}^{2\alpha+1} \le 0$.
Then, $u(t)$ is global and does not scatter for both time directions.
\end{theorem}
The proof is similar to \cite[Theorem 1.10]{MS1}.
Then, the rest of Theorem \ref{thm:main1} is immediate by looking at energy of $ c Q(x)$ for $0<c\le 1$
(see \cite{MS2,M1,M2}, for instance).

We also use the following stability estimate.
For an interval $I \subset \R$, 
we say a function $u(t) \in C(I,X)$ is an $X$-solution to \eqref{eq:NLS} on $I$ with error 
$e(t) \in L^{(d+2)\alpha/(2\alpha+1)}_{t,x}(I \times \R^d) $ if
$u(t)$ satisfies
\[
	u(t) = e^{i(t-t_0)}u(t_0) + i \Phi[|u|^{2\alpha} u ](t) -i \Phi[e](t) \IN X
\] 
for any $t,t_0 \in I$. Namely, $i\d_t u + \Delta u = - |u|^{2\alpha}u + e$ (at least formally).
\begin{proposition}[Long time stability]\label{prop:stability}
Let $t_0 \in\R$ and $I \subset \R$ be a interval containing $t_0$.
Let $e(t),\tilde{e}(t) \in L^{(d+2)\alpha/(2\alpha+1)}_{t,x}(I \times \R^d) $.
Let $\tilde{u}\in C(I,X)$ be an $X$-solution to \eqref{eq:NLS} with error $\tilde{e}(t)$.
Assume that $\tilde{u}$ satisfies 
\[
	S_I(\tilde{u})\le M,
\]
for some $M>0$. 
Then there exists 
$\varepsilon_{1}=\varepsilon_1(M)$
such that if $u(t_0),\tilde{u}(t_0)\in X$, $e(t)$, and $\tilde{e}(t)$ satisfy
\begin{eqnarray*}
S_I\(e^{i(t-t_0)\Delta}(u(t_0)-\tilde{u}(t_0))\) 
+\|e\|_{L_{t,x}^{\frac{(d+2)\alpha}{2\alpha + 1}}(I\times\R^d)}
+\|\tilde{e}\|_{L_{t,x}^{\frac{(d+2)\alpha}{2\alpha + 1}}(I\times\R^d)}
\le\varepsilon
\end{eqnarray*}
for some $0<\varepsilon<\varepsilon_{1}$, 
then there exists an $X$-solution 
$u\in C(I,X)$ to \eqref{eq:NLS} on the same $I$ with error $e(t)$. 
Further, the following estimates are valid.
\begin{align*}
% \|u-\tilde{u}\|_{L^{3\alpha}_{t,x}(I\times \R)}
S_I(u-\tilde{u})
&\le C\varepsilon,\\
\norm{|u|^{2\alpha}u-|\tilde{u}|^{2\alpha}\tilde{u}}_{L_{t,x}^{\frac{(d+2)\alpha}{2\alpha+1}}
(I\times\R^d)}&\le C\varepsilon,\\
\|u -\tilde{u}\|_{L^\I(I,X)}
	&\le \norm{ u(t_0)-\tilde{u}(t_0) }_{X}+ C\varepsilon.
\end{align*}
\end{proposition}
The proof is standard.
For instance, see \cite{MS2,M1}.

\subsection{Functional analysis}\label{subsec:predual}
In this section, we introduce two functional analysis results on Morrey spaces.

We first give a pre-dual of Morrey space $M^{p}_{q,r}$.
This allows us to consider a weak-$*$ convergence in Morrey and hat-Morrey spaces.
In particular, thanks to the Banach-Alaoglu theorem, a closed unit ball of $M^{p}_{q,r}$
is compact with respect to weak-$*$ topology.
We use an argument similar to \cite{GS}.
\begin{definition}
% Let $1<p'<\I$ and $1<q'<\I$.
Let $ 1\le q \le p \le \I$.
We say a function $g$ is a $(p',q')$-block with respect to a dyadic cube
$Q \in \mathcal{D}$ if
% if there exists a dyadic cube $Q$ such that 
$\supp g \subset \overline{Q}$ and
\[
	\norm{g}_{L^{q'}(\R^d)}= \norm{g}_{L^{q'}(Q)} \le |Q|^{\frac{1}{q'}-\frac1{p'}}.
\]
A function $g$ is simply called a $(p',q')$-block
if there exists a dyadic cube $Q_0$ such that
$g$ becomes a $(p',q')$-block with respect to $Q_0$.
\end{definition}
\begin{definition}
% Let $1<p'<\I$, $1<q'<\I$, and $1<r'<\I$.
Let $1 \le q \le p \le r \le \I$.
The block space $N^{p'}_{q',r'}$ is the set of all measurable functions $g$
for which there exists a decomposition
\[
	g(x) = \sum_{{\bf j}\in \Z^{1+d}} \lambda_{\bf j} A_{\bf j} (x),
\]
where $A_{\bf j}$ is a $(p',q')$-block with respect to $\tau^{j_1}_{j'} = 2^{-j_1}([0,1)^d+ j')$ (${\bf j}=(j_1,j') \in \Z \times \Z^d$),
% \[
% 	\supp A_j \subset \tau^{j}_k = 2^{-j}(k + [0,1)^d)	
% \]
 $\lambda_{\bf j} \in \ell^{r'}_{\bf j}(\Z^{1+d})$,
and the convergence takes place for almost all $x\in \R^d$.
The norm of $g$ is given by
\[
	\norm{g}_{N^{p'}_{q',r'}} = \inf \norm{\lambda_{\bf j}}_{\ell^{r'}_{\bf j}(\Z^{1+d})},
\]
where $\{\lambda_{\bf j}\}$ runs over all admissible expressions above.
\end{definition}

\begin{theorem}\label{thm:predual}
Let $1\le q \le p \le r \le \I$ be such that $M^{p}_{q,r}$ is defined and let $p>1$.
The generalized Morrey space
$M^{p}_{q,r}$ is the dual of the block space $N^{p'}_{q',r'}$ in the following sense.
\begin{enumerate}
\item Let $f \in M^{p}_{q,r}$. Then, for any $g \in N^{p'}_{q',r'}$,
we have $f\cdot g \in L^1(\R^d)$ and the mapping
\[
	N^{p'}_{q',r'} \ni g \mapsto \int_{\R^d} f(x) g(x) dx \in \C
\]
defines a continuous linear functional $L_f$ on $N^{p'}_{q',r'}$.
\item Conversely, any continuous linear functional $L$ on $N^{p'}_{q',r'}$
can be realized as $L=L_f| N^{p'}_{q',r'}$ with a certain $f\in M^{p}_{q,r}$.
If $f_1,f_2 \in M^p_{q,r}$ define the same functional then $f_1=f_2$ almost everywhere.
\end{enumerate}
Furthermore, we have $\norm{L_f}_{M^{p}_{q,r} \to \C} = \norm{f}_{N^{p'}_{q',r'}}$. 
\end{theorem}
\begin{remark}
By the above theorem, the dual of $\hat{M}^{p}_{q,r}$ is given by
$\hat{N}^{p'}_{q',r'} = \{ f \in \mathcal{S}' \ |\ \F f \in N^{p}_{q,r'} \}$
with $(f,g)_{\hat{M}^{p}_{q,r},\hat{N}^{p'}_{q',r'}} = \int_{\R^d} \F f(\xi) \F g(\xi) d\xi$,
for $1 \le q'  \le p' \le r \le \I$ such that $\hat{M}^{p}_{q,r}$ is defined and $p<\I$.
\end{remark}
\begin{proof}
We first prove the first assertion.
For given $g\in N^{p'}_{q',r'}$ and $\eps>0$, there exists a decomposition
\[
	g(x) = \sum_{j\in \Z, k\in \Z^d} \lambda_k^j g_k^j (x),
\]
where each $g_k^j$ is a $(p',q')$-block with respect to $\tau_k^j$
and 
\[
	\norm{\lambda_k^j}_{\ell^{r'}_{j,k}(\Z^{1+d})} \le (1+\eps) \norm{g}_{N^{p'}_{q',r'}} .
\]
Then, twice use of H\"older's inequality yields
\begin{align*}
	\norm{fg}_{L^1(\R^d)}
	\le{}& \sum_{j\in\Z,\,k\in\Z^d} |\lambda_k^j| \int_{\tau_k^j} |f(x) g_k^j(x)| dx \\
	\le{}& \sum_{j\in\Z,\,k\in\Z^d} |\lambda_k^j| \norm{f}_{L^q(\tau_k^j)} \norm{g_k^j}_{L^{q'}} \\
	\le{}& \sum_{j\in\Z,\,k\in\Z^d} |\lambda_k^j| \(|\tau_k^j|^{\frac1p-\frac1q} \norm{f}_{L^q(\tau_k^j)}\) \\
	\le{}& \norm{f}_{M^{p}_{q,r}} \norm{\lambda_k^j}_{\ell^{r'}_{j,k}(\Z^{1+d})}
	\le (1+\eps)\norm{f}_{M^p_{q,r}} \norm{g}_{N^{p'}_{q',r'}}.
\end{align*}

Let us proceed to the second assertion.
Let $L$ be a bounded linear functional on $N^{p'}_{q',r'}$.
For each $j\in \Z$ and $k\in\Z^d$, the mapping
\[
	L^{p'}(\R^d) \ni g \mapsto L_k^j(g) \equiv L(g {\bf 1}_{\tau_k^j}) \in \C
\]
is a bounded linear functional.
Since $p' \in [1,\I)$, we see that $L_k^j$ 
 is realized by an $L^p_{\mathrm{loc}}(\R^d)$-function $f_k^j$ with $\supp f_k^j \subset \overline{\tau_k^j}$.
Since $L_{k_1}^{j_1}(g) = L_{k_2}^{j_2}(g)$ holds by definition for all  $g\in L^{p'}(\R^d)$ with
$\supp g \subset \tau_{k_1}^{j_1} \cap \tau_{k_2}^{j_2}$,
one sees that there is a function $f\in L^p_{\mathrm{loc}}(\R^d)$ such that
$f_k^j (x)= f {\bf 1}_{\tau_k^j}(x)$ for all $j\in \Z$ and $k\in \Z^d$.
Then,
\[
	L_k^j(g) = \int_{\tau_k^j} g(x) f(x) dx
\]
for all $g \in L^{p'}_{\mathrm{loc}}(\R^d)$, $j\in\Z$, and $k \in \Z^d$.
We now show that $f \in M^{p}_{q,r}$.
We define
\[
	g_k^j(x) =
	\left\{
	\begin{aligned}
	&|\tau_k^j|^{\frac1p-\frac1q} \norm{f}_{L^q(\tau_k^j)}^{1-q}|f(x)|^{q-1} {\bf 1}_{\tau_k^j}(x), && 
	\norm{f}_{L^q(\tau_k^j)} > 0\\
	&0 &&\text{otherwise}.
	\end{aligned}
	\right.
\]
Remark that $q<\I$ by the definition of $M^{p}_{q,r}$.
Then, $ g_k^j$ is a $(p',q')$-block with respect to $\tau_k^j$
since $\norm{g_k^j}_{L^{q'}} = |\tau_k^j|^{\frac1p-\frac1q}$. 
Fix a finite set $K \subset \Z^{1+d}$.
Take an arbitrary nonnegative sequence $\{\rho_k^j\}_{j,k} \in \ell^{r'}(\Z^{1+d})$ supported on $K$
and set
\begin{equation}\label{eq:dualpf1}
	g_K = \sum_{(j,k) \in K} \rho_k^j g^j_k \in N^{p'}_{q',r'}.
\end{equation}
We have
\[
	\sum_{(j, k)\in K} \rho_k^j |\tau_k^j|^{\frac1p-\frac1q}  \norm{f}_{L^q(\tau_k^j)}  = \int_{\R^d} |f(x)|g_K(x) dx
	= L(h),
\]
where $h= g_K(x) \overline{\mathrm{sgn}(f(x))}$.
Further, by the decomposition \eqref{eq:dualpf1},
\[
	|L(h)| \le \norm{L}_{N^{p'}_{q',r'} \to \C} \norm{h}_{N^{p'}_{q',r'}}
	\le \norm{L}_{N^{p'}_{q',r'} \to \C} \norm{\rho_k^j}_{\ell^{r'}_{j,k}(\Z^{1+d})}.
\]
Since $r>1$ and since $K$ and $\{\rho_k^j\}$ are arbitrary, we conclude that
\[
	\norm{f}_{M^{p}_{q,r}} \le \norm{L}_{N^{p'}_{q',r'} \to \C}.
\]
Furthermore, $L=L_f|N^{p'}_{q',r'}$ holds.
The uniqueness follows by a standard argument.
\end{proof}

We next characterize total boundedness of a bounded set $K \subset M^{p}_{q,r}$.
\begin{theorem}\label{thm:totallybounded}
Let $1 \le q < p <r <\I$.
A bounded set $K \subset M^{p}_{q,r}$ is totally bounded if and only if for any $\eta>0$ there exists $C(\eta)>0$
such that
\begin{equation}\label{eq:tbdd}
	 \norm{f {\bf 1}_{\{|x|\ge C(\eta)\}} }_{M^{p}_{q,r}}
	+  \sup_{|z| \le 1/C(\eta)} \norm{(T(z)-1)f}_{M^{p}_{q,r}}
	\le \eta
\end{equation}
for any $f\in K$, where $T(a)$ is a translation $(T(a)f)(x)=f(x-a)$, $a\in \R^d$.
\end{theorem}
\begin{remark}
This kind of characterization for $L^p$ space for $1\le p < \I$ is due to Kolmogorov, Tamarkin, Tulajkov, and Riesz
\cite{Kol,Rie,Tam,Tula} (see also \cite{HOH}).
Further, when $p=2$, a characterization in terms of Fourier transformation is given by Pego \cite{Pe}.
\end{remark}
The proof relies on the following lemma. % by Hanche-Olsen and Holden \cite{HOH}.
\begin{lemma}[\cite{HOH}]\label{lem:tbdd}
Let $X$ be a metric space. Assume that, for every $\eps$, there exist some $\delta$, a metric space
$W$, and a mapping $P: X\to W$ so that $P(X)$ is totally bounded, and if $x,y \in X$ are such that
 $d_W(P(x),P(y))< \delta$ then $d_X(x,y)< \eps$. Then $X$ is totally bounded.
\end{lemma}
\begin{proof}[Proof of Theorem \ref{thm:totallybounded}]
Let us first prove if a bounded set $K\subset M^{p}_{q,r}$ satisfies \eqref{eq:tbdd} then $K$ is totally bounded.

Given $\eps>0$, take $C(\eps)>0$ so that \eqref{eq:tbdd} holds.
Let $j_0 \in \Z $ be the maximum number that satisfies $2^{-j_0} > C(\eps)$. 
Define $D_0 := [-2^{-j_0},2^{-j_0})^d \subset \R^d$.
Remark that $D_0 \supset \{|x| \le C(\eps)\}$.
Let $j_1=j_1(\eps)\ge j_0$ be an integer to be chosen later.
Recall the notation $\tau^j_k := 2^{-j} ([0,1)^d + k) \in \mathcal{D}_j$ with $j\in\Z$ and $k\in\Z^d$.
Let $K_0 = \Z^d \cap [-2^{j_1-j_0},2^{j_1-j_0})^d$ so that $D_0 = \bigcup_{k \in K_0} \tau_k^{j_1}$.
We define a projection operator $\mathcal{P}=\mathcal{P}(\eps):M^{p}_{q,r}\to M^{p}_{q,r}$ as 
\[
	\mathcal{P} f (x)
	:=
	\left\{
	\begin{aligned}
	&\frac1{|\tau_k^{j_1}|} \int_{\tau_k^{j_1}} f(y) dy,& & x \in \tau_k^{j_1},\, \exists k\in K_0,\\
	&0 ,&&\text{otherwise}.
	\end{aligned}
	\right.
\]
Remark that $\supp \mathcal{P}f \subset \overline{D_0}$. 
$\mathcal{P}$ is a bounded operator. Indeed,
by the embedding $L^p \hookrightarrow M^p_{q,r}$ for $q<p<r$, $\supp \mathcal{P}f \subset \overline{D_0}$, and triangle inequality,
we have
\begin{align*}
	\norm{\mathcal{P}f}_{M^{p}_{q,r}}
	&{}\le C\norm{\mathcal{P}f}_{L^p(D_0)}
	\le C\sum_{k \in K_0 %[-2^{j_1-j_0},2^{j_1-j_0})^d
	} 
	\norm{\frac1{|\tau_k^{j_1}|}\int_{\tau^{j_1}_k}fdy}_{L^p(\tau_k^{j_1})} \\
	&{}
	\le C\sum_{k \in K_0 %[-2^{j_1-j_0},2^{j_1-j_0})^d
	} 
	|\tau_k^{j_1}|^{\frac1p-\frac1q}\norm{f}_{L^q(\tau^{j_1}_k)}
	\le C 2^{d(j_1-j_0)(1-\frac1r)} \norm{f}_{M^p_{q,r}}.
\end{align*}

By the assumption \eqref{eq:tbdd}, 
\begin{equation}\label{eq:pf_tbdd1}
	\norm{f-\mathcal{P}f}_{{M}^p_{q,r}} \le
	\norm{f{\bf 1}_{D_0^c}}_{{M}^p_{q,r}}
	+\norm{f{\bf 1}_{D_0} -\mathcal{P}f}_{{M}^p_{q,r}}
	\le \eps + \norm{f{\bf 1}_{D_0} -\mathcal{P}f}_{{M}^p_{q,r}}
\end{equation}
for any $f\in K$.
We estimate $\norm{f{\bf 1}_{D_0} -\mathcal{P}f}_{{M}^p_{q,r}}$.
Remark that
\begin{align*}
	\norm{f{\bf 1}_{D_0} -\mathcal{P}f}_{{M}^p_{q,r}}^r
	=
	\sum_{j \in\Z} \sum_{k \in \Z^d} |\tau_k^j|^{\(\frac1p-\frac1q\)r} 
	\norm{f{\bf 1}_{D_0}-\mathcal{P}f}_{L^q(\tau_k^j)}^r.
\end{align*}

We first consider the case $j \le j_0$. In this case, $\tau_k^j \cap D_0 \neq \emptyset$ if and only if
$k\in \{-1,0\}^d \subset \Z^d$. Further, for $k\in \{-1,0\}^d$, $\tau_k^j \cap D_0=\tau_k^{j_0}$.
Therefore,
\begin{equation}\label{eq:pf_tbdd2}
\begin{aligned}
	&\sum_{j \le j_0} \sum_{k \in \Z^d} |\tau_k^j|^{\(\frac1p-\frac1q\)r} 
	\norm{f{\bf 1}_{D_0}-\mathcal{P}f}_{L^q(\tau_k^j)}^r \\
	&{}= \sum_{j \le j_0} \sum_{k \in \{-1,0\}^d} |\tau_k^j|^{\(\frac1p-\frac1q\)r} 
	\norm{f-\mathcal{P}f}_{L^q(\tau_k^{j_0})}^r \\
	&{}= \sum_{j \le j_0} 2^{-d(j-j_0)\(\frac1p-\frac1q\)r} \sum_{k \in \{-1,0\}^d} |\tau_k^{j_0}|^{\(\frac1p-\frac1q\)r} 
	\norm{f-\mathcal{P}f}_{L^q(\tau_k^{j_0})}^r \\
	&{}\le C_{p,q,r,d} \sum_{k \in \{-1,0\}^d} |\tau_k^{j_0}|^{\(\frac1p-\frac1q\)r} 
	\norm{f-\mathcal{P}f}_{L^q(\tau_k^{j_0})}^r
\end{aligned}
\end{equation}
since $q<p$.
We next consider the case $j_0 \le j \le j_1$.
If $j \ge j_0$ then, for each $\tau_k^j \in \mathcal{D}_j$, either $\tau_k^j \subset \tau_l^{j_0}\subset D_0$ for some $l\in\{-1,0\}^d$ or $\tau_k^j \cap D_0 =\emptyset$ holds.
Further, $\tau_k^j \subset \tau_l^{j_0}$ if and only if $k \in [-2^{j-j_0},2^{j-j_0})^d$. We have
\begin{multline*}
	\sum_{j_0 \le j \le j_1} \sum_{k \in \Z^d} |\tau_k^j|^{\(\frac1p-\frac1q\)r} 
	\norm{f{\bf 1}_{D_0}-\mathcal{P}f}_{L^q(\tau_k^j)}^r \\
	= \sum_{j_0 \le j \le j_1} \sum_{k \in [-2^{j-j_0},2^{j-j_0})^d} |\tau_k^j|^{\(\frac1p-\frac1q\)r} 
	\norm{f-\mathcal{P}f}_{L^q(\tau_k^j)}^r.
\end{multline*}
% For such $(j,k)$ in the sum on the right hand side, 
For each $\tau_k^j \in \mathcal{D}_j$ in the summand of the right hand side,
H\"older's inequality yields
\begin{align*}
\norm{f-\mathcal{P}f}_{L^q(\tau_k^j)}^q
&{}= \sum_{m\in\Z^d,\, \tau_{m}^{j_1} \subset \tau_{k}^{j}}
\int_{\tau_{m}^{j_1}}\abs{ \frac1{|\tau_m^{j_1}|} \int_{\tau_m^{j_1}} (f(x) -f(y)) dy}^q dx\\
&{}\le \sum_{m\in\Z^d,\, \tau_{m}^{j_1} \subset \tau_{k}^{j}}
\frac1{|\tau_m^{j_1}|} \int_{\tau_{m}^{j_1}}  \int_{\tau_m^{j_1}} |f(x) -f(y)|^q dy dx.
\end{align*}
We introduce the change of variable $y=x-z$.
Since it holds for any $m\in\Z^d$ that $|x-y|\le C_d 2^{-j_1}$ as long as $x,y \in \tau_m^{j_1}$,
one has
\begin{align*}
\norm{f-\mathcal{P}f}_{L^q(\tau_k^j)}^q
&{} \le \sum_{m\in\Z^d,\, \tau_{m}^{j_1} \subset \tau_{k}^{j}}
\frac1{|\tau_m^{j_1}|} \int_{|z| \le C 2^{-j_1}}  \( \int_{\tau_m^{j_1}} |f(x) -T(z)f(x)|^q dx \)dz\\
&{}=
2^{j_1d} \int_{|z| \le C 2^{-j_1}} \norm{(T(z)-1)f}_{L^q(\tau_k^j)}^q dz\\
&{}\le C 2^{j_1d\frac{q}{r}}\(\int_{|z| \le C 2^{-j_1}} \norm{(T(z)-1)f}_{L^q(\tau_k^j)}^r dz\)^{\frac{q}{r}},
\end{align*}
where we have used H\"older's inequality in $z$ to obtain the last line.
Hence, combining above estimates, we reach to
\begin{equation}\label{eq:pf_tbdd3}
	\begin{aligned}
	&\sum_{j_0 \le j \le j_1} \sum_{k \in \Z^d} |\tau_k^j|^{\(\frac1p-\frac1q\)r} 
	\norm{f{\bf 1}_{D_0}-\mathcal{P}f}_{L^q(\tau_k^j)}^r	\\
	&{}\le C_{q,r,d}
	\sum_{j_0 \le j \le j_1} \sum_{k \in \Z^d} |\tau_k^j|^{\(\frac1p-\frac1q\)r} 
	2^{j_1d}\int_{|z| \le C 2^{-j_1}} \norm{(T(z)-1)f}_{L^q(\tau_k^j)}^r dz \\
	&{}= C_{q,r,d} 2^{j_1d}\int_{|z| \le C 2^{-j_1}}
	\(\sum_{j_0 \le j \le j_1} \sum_{k \in \Z^d} |\tau_k^j|^{\(\frac1p-\frac1q\)r} 
	 \norm{(T(z)-1)f}_{L^q(\tau_k^j)}^r\) dz.
	\end{aligned}
\end{equation}
We finally consider $j \ge j_1$.
As in the previous case,
\begin{multline*}
	\sum_{ j \ge j_1} \sum_{k \in \Z^d} |\tau_k^j|^{\(\frac1p-\frac1q\)r} 
	\norm{f{\bf 1}_{D_0}-\mathcal{P}f}_{L^q(\tau_k^j)}^r \\
	= \sum_{ j \ge j_1} \sum_{k \in [-2^{j-j_0},2^{j-j_0})^d} |\tau_k^j|^{\(\frac1p-\frac1q\)r} 
	\norm{f-\mathcal{P}f}_{L^q(\tau_k^j)}^r.
\end{multline*}
For each $\tau_k^j \in \mathcal{D}_j$ in the summand of the right hand side, $\tau_k^j \subset \tau_l^{j_1}$ for some $l\in K_0$.
Denoting $l=l(k) \in K_0$ be a (unique) vector such that $\tau_k^j \subset \tau_{l}^{j_1}$, we have
\[
\norm{f-\mathcal{P}f}_{L^q(\tau_k^j)}^q
\le \frac1{|\tau_{l(k)}^{j_1}|} \int_{\tau_{k}^{j}} \( \int_{\tau_{l(k)}^{j_1}} |f(x) -f(y)|^q dy\) dx
\]
as in the previous case. Now, we again introduce change of variable $y=x-z$. In this case,
since $x \in \tau_k^j$ and $y \in \tau_{l(k)}^{j_1}$, we have
$|z| \le C_d(2^{-j}+2^{-j_1}) \le C_d2^{-j_1}$. 
The rest of the estimate is similar to the previous case. We obtain
\begin{equation}\label{eq:pf_tbdd4}
	\begin{aligned}
	&\sum_{ j \ge j_1} \sum_{k \in \Z^d} |\tau_k^j|^{\(\frac1p-\frac1q\)r} 
	\norm{f{\bf 1}_{D_0}-\mathcal{P}f}_{L^q(\tau_k^j)}^r	\\
	&{}\le C_{q,r,d} 2^{j_1d}\int_{|z| \le C 2^{-j_1}}
	\(\sum_{ j \ge j_1} \sum_{k \in \Z^d} |\tau_k^j|^{\(\frac1p-\frac1q\)r} 
	 \norm{(T(z)-1)f}_{L^q(\tau_k^j)}^r\) dz.
	\end{aligned}
\end{equation}

By \eqref{eq:pf_tbdd2}, \eqref{eq:pf_tbdd3}, and \eqref{eq:pf_tbdd4}, we have
\begin{align*}
	\norm{f{\bf 1}_{D_0} -\mathcal{P}f}_{{M}^p_{q,r}}^r
	&{}\le C_{p,q,r,d} 2^{j_1d} \int_{|z|\le C_d 2^{-j_1}} \norm{(T(z)-1)f}_{M^{p}_{q,r}}^r dz\\
	&{}\le C_{p,q,r,d} \(\sup_{|z| \le C_d 2^{-j_1}}\norm{(T(z)-1)f}_{M^{p}_{q,r}}\)^r.
\end{align*}
Now, we choose $j_1$ so that $C_d 2^{-j_1} \le 1/C(\eps(C_{p,q,r,d})^{-1/r})$, where $C(\cdot)$ is the function in
the assumption \eqref{eq:tbdd}. Then, plugging the above estimate to \eqref{eq:pf_tbdd1} and using the assumption \eqref{eq:tbdd}, we conclude that $\norm{f-\mathcal{P}f}_{M^p_{q,r}} \le 2\eps$ for any $f\in K$ and so that 
$\norm{f-g}_{M^p_{q,r}} \le 4\eps + \norm{\mathcal{P}f-\mathcal{P}g}_{M^p_{q,r}}$ for any $f,g \in K$.
This shows that if $\norm{\mathcal{P}f-\mathcal{P}g}_{M^p_{q,r}} \le \eps$ then $\norm{f-g}_{M^p_{q,r}} \le 5\eps$.
Since $\mathcal{P}$ is bounded and since images of $\mathcal{P}$ is finite dimensional,
$\mathcal{P}K$ is totally bounded. Thus, we conclude from Lemma \ref{lem:tbdd} that $K$ is totally bounded.

Conversely, assume that a bounded set $K$ is totally bounded and prove \eqref{eq:tbdd}.
By a standard argument, it suffices to show that each $f \in K \subset M^p_{q,r}$ satisfies
\[
	\norm{f {\bf 1}_{\{|x|\ge R\}}}_{M^{p}_{q,r}} + \sup_{|z|\le 1/R}
	\norm{(T(z)-1)f}_{M^p_{q,r}} \to 0
\]
as $R\to\I$. Fix $f\in K$.
Since $r<\I$, for any $\eps>0$, there exists a finite set $\Omega \in \Z \times \Z^d$ such that
\[
	\sup_{|z|\le 1} \(\sum_{(j,k) \in \Z^{1+d}\setminus \Omega} |\tau_k^j|^{(\frac1p-\frac1q)r}\norm{T(z)f}_{L^q(\tau_k^j)}^r\)^{1/r} \le \eps.
\]
Hence, the proof is reduced to showing that
$$\norm{f {\bf 1}_{\{|x|\ge R\}}}_{L^q(\tau)} + \sup_{|z|\le 1/R}\norm{(T(z)-1)f}_{L^q(\tau)} \to 0$$
as $R\to\I$ for each dyadic cube $\tau$ and $f\in L^q_{\mathrm{loc}}$.
This is obvious.
\end{proof}

\section{Compactness tool}\label{sec:compactness}

In this section, we treat a compactness result, a linear profile decomposition.
We first collect notations and elementary facts in Section \ref{subsec:pd1} and Section \ref{subsec:pd1.5}.
The main result of this section is Theorem \ref{thm:pd} in Section \ref{subsec:pd2}.
Throughout this section, we suppose $d\ge1$ and \eqref{asmp:alpha}.

\subsection{Deformations}\label{subsec:pd1}
We introduce a dilation
\[
	(D(h)f)(x) = h^{\frac1\alpha} f(hx), \quad h\in 2^\Z,
\]
translation in physical space
\[
	(T(a)f)(x) = f(x-a), \quad a \in \R^d,
\]
translation in Fourier space
\[
	(P(b)f)(x) = e^{-ix\cdot b} f(x), \quad b\in \R^d,
\]
and Schr\"odinger group $U(s) = e^{is\Delta}$, $s\in \R$.
Each of them forms a group and inverses of them are summarized as follows:
\[
	D(h)^{-1} = D(h^{-1}),\,
	T(a)^{-1} = T(-a),\,
	P(b)^{-1} = P(-b),\,
	U(s)^{-1} = U(-s).
\]
It is easy to see that $T(a)$ and $U(s)$ are isometric bijection on $\hat{M}^{d\alpha}_{2,r}$ since
they are just multiplication by $e^{-ia\cdot \xi}$ and $e^{-is|\xi|^2}$ respectively in the Fourier side.
Similarly, $D(h)$ is also an isometric bijection on $\hat{M}^{d\alpha}_{2,r}$ as long as $h$ is a dyadic number.
For $P(b)$, we have the following.
\begin{lemma}\label{lem:Paction}
It holds that
\[
	2^{-d} \norm{ f}_{\hat{M}^{d\alpha}_{2,r}}
	\le \norm{P(b)f}_{\hat{M}^{d\alpha}_{2,r}}
	\le 2^d \norm{ f}_{\hat{M}^{d\alpha}_{2,r}}
\]
for any $b\in \R^d$ and $f \in \hat{M}^{d\alpha}_{2,r}$.
\end{lemma}
For the proof, see \cite[Lemma 2.3]{MS2}.
\begin{remark}\label{rmk:quasi-norm}
Let $\ell_r:\hat{M}^{d\alpha}_{2,r} \to \R$ be as in \eqref{eq:M_size}.
As an immediate consequence of Lemma \ref{lem:Paction},
we see that $\ell_r(\cdot)$ is a quasi norm of $\hat{M}^{d\alpha}_{2,r}$ which is
equivalent to $\norm{\cdot}_{\hat{M}^{d\alpha}_{2,r}}$.
\end{remark}

Next we collect commutation of the above operators:
\begin{align*}
	D(h)T(a) &{}= T(h^{-1}a) D(h), &
	D(h)P(b) &{}= P(hb)D(h), \\
	D(h)U(t) &{}= U(h^{-2}t)D(h), &
	T(a)P(b) &{}= e^{ia\cdot b}P(b)T(a), \\
	T(a)U(s) &{}= U(s)T(a), &
\end{align*}
and 
\[
	U(s)P(b) = e^{-is|b|^2}P(b)U(s) T(-2sb).
\]
The last one is nothing but a Galilean transform.

\begin{definition}\label{def:deformation}
We call a bounded operator on $\hat{M}^{d\alpha}_{2,r}$ of the form
\begin{equation}\label{eq:nrep}
	\mathcal{G} =e^{i\theta} D(h) P(b) U (s) T(a) , \quad
	(\theta,h,s,a,b) \in \R \times 2^\Z \times \R \times \R^d \times \R^d
\end{equation}
as a deformation of $\hat{M}^{d\alpha}_{2,r}$. 
We define $G \subset \mathcal{L}(\hat{M}^{d\alpha}_{2,r})$ as a set of 
all deformations.
$G$ forms a group with functional composition as a multiplication.
$\mathrm{Id} \in G$ is the identity element.
\end{definition}
It follows from Lemma \ref{lem:Paction} and the above commutation that
\[
	2^{-d} \norm{ f}_{\hat{M}^{d\alpha}_{2,r}}
	\le \norm{\mathcal{G}f}_{\hat{M}^{d\alpha}_{2,r}}
	\le 2^d \norm{ f}_{\hat{M}^{d\alpha}_{2,r}}
\]
and
\[
	\ell_r( \mathcal{G} f) = \ell_r (f)
\]
for any $\mathcal{G} \in G$ and $f \in \hat{M}^{d\alpha}_{2,r}$.

\begin{remark}[Normal representation]
By the commutation relations above, we can freely change the order of four operators
$D$, $P$, $U$, and $T$ in the representation \eqref{eq:nrep} by a suitable change of parameters.
We refer the representation as in \eqref{eq:nrep} to as a
\emph{normal representation of $\mathcal{G}\in G$}.
% However, in most of the rest part, we work with the following different representation
% \begin{equation}\label{eq:rrep}
% 	\mathcal{G} = e^{i\tilde{\theta}}
% 	U (\tilde{s}) D(\tilde{h}) P(\tilde{b}) T(\tilde{a}).
% \end{equation}
% \begin{equation}\label{eq:rrep2}
% 	\mathcal{G} = e^{i\tilde{\theta}}
% 	D(\tilde{h}) T(\tilde{a}) P(\tilde{b}) U (\tilde{s}).
% \end{equation}
% The correspondence of this representation and the normal representation \eqref{eq:nrep}
% is given by
% \begin{equation}\label{eq:changeparameteres}
% 	\theta=\tilde{\theta}-\tilde{h}^2\tilde{s}|\tilde{b}|^2,\quad h=\tilde{h},\quad
% 	b=\tilde{b},\quad s=\tilde{h}^2\tilde{s},\quad a= \tilde{a} - 2(\tilde{h})^2 \tilde{s} \tilde{b}.
% \end{equation}
% \begin{equation}\label{eq:changeparameteres2}
% 	\theta=\tilde{\theta}+\tilde{a}\cdot\tilde{b},\quad h=\tilde{h},\quad
% 	b=\tilde{b},\quad s=\tilde{s},\quad a= \tilde{a} .
% \end{equation}
\end{remark}

\subsection{Orthogonality of families of deformations}\label{subsec:pd1.5}
We next introduce several notions on families of deformations.
\begin{definition}[a vanishing family]
We say a family of deformations $\{\mathcal{G}_n\}_n \subset G$ is \emph{vanishing}
if, in the normal representation
\[
	\mathcal{G}_n = e^{i\theta_n} D(h_n) P(b_n) U (s_n) T(a_n),
\]
it holds that
\[
	|\log h_n| + |b_n| + |s_n| + |a_n| \to \I
\]
as $n\to\I$.
\end{definition}
\begin{lemma}\label{lem:inv_vanishing}
A family $\{\mathcal{G}_n\}_n \subset G$ is vanishing if and only if
$\{\mathcal{G}_n^{-1}\}_n$ is vanishing.
\end{lemma}
\begin{proof}
If we denote $\mathcal{G}_n = e^{i\theta_n} D(h_n) P(b_n) U (s_n) T(a_n)$ then 
\[
	\mathcal{G}_n^{-1} = e^{i(-\theta_n + a_n\cdot b_n + s_n |b_n|^2)} D(h_n^{-1}) P(-h_n b_n) U \(-\frac{s_n}{h_n^2}\)
	T\(-\frac{a_n+2 s_n b_n}{h_n}\).
\]
It is obvious that if $\mathcal{G}_n$ is not vanishing then $\mathcal{G}_n^{-1}$ is not vanishing.
The other direction follows from the same argument by the relation $(\mathcal{G}_n^{-1})^{-1}=\mathcal{G}_n$.
\end{proof}
\begin{lemma}\label{lem:nonvanishing}
If a family $\{\mathcal{G}_n\}_n \subset G$ is not vanishing then
there exist a subsequence $n_k$ of $n$ and $\mathcal{G} \in G$
such that $\mathcal{G}_{n_k} \to \mathcal{G}$ strongly in $\mathcal{L}(\hat{M}^{d\alpha}_{2,r})$ as $k \to\I$,
i.e., for any $\phi \in \hat{M}^{d\alpha}_{2,r}$, $\mathcal{G}_{n_k} \phi \to \mathcal{G}\phi$ (strongly)
in $\hat{M}^{d\alpha}_{2,r}$ as $k\to\I$.
\end{lemma}
\begin{proof}
We denote $\mathcal{G}_n = e^{i\theta_n} D(h_n) P(b_n) U (s_n) T(a_n)$.
Since $\mathcal{G}_n$ is not vanishing, there exists a subsequence $n_k$ such that
$(e^{i\theta_{n_k}},h_{n_k},b_{n_k},s_{n_k},a_{n_k})$ converges to
$(e^{i\theta},h,b,s,a)\in \{z\in\C\ |\ |z|=1\} \times 2^\Z \times \R^d \times \R \times \R^d$ as $k\to\I$.
The conclusion is obvious by taking $\mathcal{G}:=e^{i\theta} D(h) P(b) U (s) T(a)$.
\end{proof}
\begin{proposition}\label{prop:vanishing}
For a family $\{\mathcal{G}_n\}_n \subset G$ of deformations,
the following three statements are equivalent:
\begin{enumerate}
\item $\{\mathcal{G}_n\}_n$ is vanishing;
\item For any $\phi \in \hat{M}^{d\alpha}_{2,r}$, $\mathcal{G}_n \phi \rightharpoonup 0$ weakly-$*$ in $\hat{M}^{d\alpha}_{2,r}$ as $n\to\I$;
\item For any subsequence $n_k$ of $n$
there exist a sequence $\{ u_k \}_k \subset \hat{M}^{d\alpha}_{2,r}$ 
and subsequence $k_l$ of $k$ such that
$ u_{k_l} \rightharpoonup 0$ and $\mathcal{G}_{n_{k_l}}^{-1} u_{k_l} \rightharpoonup \phi \neq 0 $
weakly-$*$ in $\hat{M}^{d\alpha}_{2,r}$ as $l\to\I$.
\end{enumerate}
\end{proposition}
\begin{proof}
``(2)$\Rightarrow$(3)'' is obvious by taking $u_k=\mathcal{G}_{n_k} \phi$ for some $\phi \neq0$.
``(3)$\Rightarrow$(1)'' is also immediate because the contraposition is Lemma \ref{lem:nonvanishing}.
% Finally, ``(1)$\Rightarrow$(2)'' follows by a standard argument (see \cite[Lemma 5.3]{MS2}, for instance).

Let us prove ``(1)$\Rightarrow$(2)''.
By density argument, it suffices to show that $(\F \mathcal{G}_n \phi,\F \psi) \to 0$ as $n\to\I$
for any $\phi,\psi \in \F (C_0^\I) \subset \mathcal{S}$.
If $|\log h_n| \to\I$ then we use H\"older's inequality to obtain
$|(\F \mathcal{G}_n \phi,\F \psi)| \le \norm{\F \mathcal{G}_n \phi}_{L^{r'}} \norm{\F \psi}_{L^{r}} =
C_{\phi,\psi} (h_n)^{\frac1\alpha-\frac{d}r}$.
We obtain the result by taking $r>d\alpha$ if $h_n\to0$ and $r<d\alpha$ if $h_n\to\I$.
Let us next suppose that $|\log h_n|$ is bounded and $|b_n|\to\I$ as $n\to\I$. In this case, we have 
$(\F \mathcal{G}_n \phi,\F \psi)=0$ for large $n$ because $\F \phi$ and $\F \psi$ have compact support
and because $|\log h_n|$ is bounded.
Let us suppose that $|\log h_n|+|b_n|$ is bounded and $|s_n|\to\I$ as $n\to\I$. In this case, the result follows from
\[
	|(\F \mathcal{G}_n \phi,\F \psi)| =
	|(\mathcal{G}_n \phi, \psi)| \le C \norm{U(s_n) \phi}_{L^\I} \norm{\psi}_{L^1}
	\le C |s_n|^{-\frac{d}2} \norm{\phi}_{L^1}\norm{\psi}_{L^1},
\]
where the constant $C$ depends on the bound of $|\log h_n|$. Finally, let us consider
the case where $|\log h_n|+|b_n| + |s_n|$ is bounded and $|a_n|\to\I$ as $n\to\I$.
Thanks to the boundedness of $|\log h_n|+|b_n| + |s_n|$, the proof boils down to showing that
$(T(a_n)\phi,\psi)\to0$ as $n\to\I$, which is obvious.
\end{proof}

Let us now introduce a notion of orthogonality of two families of deformations.
\begin{definition}[Orthogonality]
Let $\{\mathcal{G}_n\}_n, \{\tilde{\mathcal{G}}_n\}_n \subset G$ be
two families of deformations.
We say $\{\mathcal{G}_n\}_n$ and $ \{\tilde{\mathcal{G}}_n\}_n$ are \emph{orthogonal}
if $\{\mathcal{G}_n^{-1} \tilde{\mathcal{G}}_n \}_n$ is vanishing.
\end{definition}
\begin{proposition}
We introduce the following relation $\backsim$
for families of deformations: For
$\{\mathcal{G}_n\}_n, \{\tilde{\mathcal{G}}_n\}_n \subset G$,
$\{\mathcal{G}_n\}_n \backsim \{\tilde{\mathcal{G}}_n\}_n$
is true if  $\{{\mathcal{G}}_n\}_n$ and $\{\tilde{\mathcal{G}}_n\}_n$ are not orthogonal.
Then, $\backsim$ defines an equivalent relation.
\end{proposition}
\begin{proof}
The reflexivity of $\backsim$ follows from the fact that sequence of the identity $\{\mathcal{G}_n = \mathrm{Id}\}_n$
is not vanishing. The symmetry of $\backsim$ follows from Lemma \ref{lem:inv_vanishing}.
The transitivity of $\backsim$ is a consequence of Lemma \ref{lem:nonvanishing}.
Indeed, if $\{\mathcal{G}_n^1\}_n \backsim  \{\mathcal{G}^2_n\}_n$ and $\{\mathcal{G}_n^2\}_n \backsim  \{\mathcal{G}^3_n\}_n$
then there exists a subsequence $n_k$ such that
\[
	(\mathcal{G}_{n_k}^1)^{-1}\mathcal{G}_{n_k}^2 \to \mathcal{G} \in G,\quad
	(\mathcal{G}_{n_k}^2)^{-1}\mathcal{G}_{n_k}^3 \to \tilde{\mathcal{G}} \in G
\]
strongly in $\mathcal{L}(\hat{M}^{d\alpha}_{2,r})$ as $k\to\I$, in light of Lemma \ref{lem:nonvanishing}.
For the same subsequence $n_k$, we have
\[
	(\mathcal{G}_{n_k}^1)^{-1}\mathcal{G}_{n_k}^3 
	= [(\mathcal{G}_{n_k}^1)^{-1}\mathcal{G}_{n_k}^2][(\mathcal{G}_{n_k}^2)^{-1}\mathcal{G}_{n_k}^3]
	\to \mathcal{G}\tilde{\mathcal{G}} \in G
\]
strongly in $\mathcal{L}(\hat{M}^{d\alpha}_{2,r})$ as $k\to\I$.
This implies that the sequence $\{(\mathcal{G}_{n}^1)^{-1}\mathcal{G}_{n}^3\}_n$ does not satisfy
the third assertion of Proposition \ref{prop:vanishing}.
% Therefore, $\{(\mathcal{G}_{n}^1)^{-1}\mathcal{G}_{n}^3\}_n$ is not vanishing. 
\end{proof}
We conclude this section with an explicit representation of orthogonality.
Let $\{\mathcal{G}_n^j = e^{i\theta_n^j} D(h_n^j) P(b_n^j)  U (s_n^j) T(a_n^j)\}_n \subset G$
($j=1,2$) be families of deformations in the normal representation. 
$\{ \mathcal{G}_n^1 \}_n$ and $\{ \mathcal{G}_n^2\}_n$ are orthogonal if and only if
\begin{multline}\label{eq:orthty}
	\abs{\log \frac{h_n^1}{h_n^2}} +
	\abs{b_n^1 - \frac{ h_n^2 }{h_n^1}b_n^2} +
	\abs{s_n^1  -\(\frac{h_n^1}{h_n^2}\)^2 s_n^2 } \\
	+\abs{a_n^1  -\frac{h_n^1}{h_n^2} a_n^2 +2 \(\frac{h_n^1}{h_n^2}\)^2 s_n^2 \( b_n^1 - \frac{ h_n^2 }{h_n^1}b_n^2 \)} \to \I
\end{multline}
as $n\to\I$. This is immediate from the identity
\begin{multline*}
	(\mathcal{G}^2_n)^{-1}\mathcal{G}^1_n = e^{i\theta_n}
	D\( \frac{h^1_n}{h^2_n} \) P\( b_n^1 - \frac{ h_n^2 }{h_n^1}b_n^2 \)\\
	U\( s_n^1  -\(\frac{h_n^1}{h_n^2}\)^2 s_n^2 \)
	T\( a_n^1  -\frac{h_n^1}{h_n^2} a_n^2 +2 \(\frac{h_n^1}{h_n^2}\)^2 s_n^2 \( b_n^1 - \frac{ h_n^2 }{h_n^1}b_n^2 \) \) 
\end{multline*}
with suitable $\theta_n \in \R$. 
% Further, if we parameterize $\{ \mathcal{G}_n^1 \}_n$ and $\{ \mathcal{G}_n^2\}_n$
% as in \eqref{eq:rrep} then the orthogonality is
% \begin{multline}\label{eq:orthty}
% 	\abs{\log \frac{h_n^1}{h_n^2}} +
% 	\abs{b_n^1 - \frac{ h_n^2 }{h_n^1}b_n^2} +
% 	\abs{(h_n^1)^2(s_n^1 - s_n^2) } \\
% 	+\abs{a_n^1  -\frac{h_n^1}{h_n^2} a_n^2 +2 (h_n^1)^2 (s_n^2-s_n^1) b_n^1} \to \I
% \end{multline}
% by means of \eqref{eq:changeparameteres}.
% Further, by means of \eqref{eq:changeparameteres2}, the representation is the same 
% even if $\{ \mathcal{G}_n^1 \}_n$ and $\{ \mathcal{G}_n^2\}_n$ are parameterized
%  as in \eqref{eq:rrep2}.
\subsection{Linear profile decomposition}\label{subsec:pd2}

\begin{theorem}[Linear profile decomposition]\label{thm:pd}
Let $\frac{2}{d+\frac{2}{d+3}}<\alpha<2/d$ and $(d\alpha)'<r < ((d+2)\alpha)^*$.
For any bounded sequence $\{u_n\}_n \subset \hat{M}^{d\alpha}_{2,r}$,
there exist $\phi^j \in \hat{M}^{d\alpha}_{2,r} $, $R_n^j \in \hat{M}^{d\alpha}_{2,r} $ and 
pairwise orthogonal families of deformations $\{\mathcal{G}^j_n\}_n \subset G$ 
($j=1,2,\dots$) parametrized as in \eqref{eq:nrep} by $\{ \Gamma_n^j = (0, h_n^j, \xi_n^j, s_n^j , y_n^j) \}_n$
such that, extracting a subsequence in $n$,
\begin{equation}\label{eq:pf:decomp}
	u_n = \sum_{j=1}^l \mathcal{G}^j_n \phi^j + R_n^l
\end{equation}
for all $n,l\ge1$. Moreover, $\{R_n^j\}_{n,j}$ satisfies
\begin{equation}\label{eq:pf:swlimit}
	(\mathcal{G}_n^{k})^{-1}R_n^j \rightharpoonup 
	\begin{cases}
	\phi^k & j <k,\\
	0 & j \ge k
	\end{cases}
\end{equation}
weakly-$*$ in $\hat{M}^{d\alpha}_{2,r}$ as $n\to\I$ for all $j\ge0$ and $k\ge1$, with a convention $R_n^0=u_n$,
and
\begin{equation}\label{eq:pf:smallness}
	\limsup_{n\to\I} \norm{ e^{it\Delta} R_n^l }_{L^{(d+2)\alpha}_{t,x} (\R^{1+d})} \to 0
\end{equation}
as $l\to\I$. 
Furthermore, a decoupling inequality
\begin{equation}\label{eq:pf:Pythagorean}
	\limsup_{n\to\I} \ell_r(u_n) \ge \(\sum_{j=1}^J \ell_r (\phi^j)^r \)^{1/r}
	+ \limsup_{n\to\I} \ell_r( R_n^J ) 
\end{equation}
holds for all $J\ge1$.
\end{theorem}
The proof is done by modifying the argument in the $L^2$ case \cite{BV,CK}.
The modification to the $\hat{M}^{d\alpha}_{2,r}$-framework
is essentially the same as for the Airy equation, see \cite{MS2}.
For self-containedness, we give a proof in Appendix \ref{sec:pf_profile}.

\section{Proof of Theorem \ref{thm:main3}}\label{sec:main3}

We first introduce a function $L(E)$ for $E\ge 0$ by
\[
	L(E) =\sup \left\{
	%\norm{u}_{L^{(d+2)\alpha}_{t,x}(I \times \R^d)}
	S_I(u)
	\middle|
	\begin{aligned}
	&u(t): I\times \R^d \to \C\text{ : sol.\ to }\eqref{eq:NLS},\\
	& \sup_{t\in I} \ell_r(u(t)) \le E
 	\end{aligned}
	\right\} \in [0,\I].
\]
Remark that, in the above definition, $u(t)$ is not always a maximal-lifespan solution. 
Small data scattering implies that $L(E) \le CE $ for $E \le \delta$.
Further, since $Q_\alpha(t,x)$ is a nonscattering solution,
$L(\ell_r(Q)) = \I$. 
By the long time stability, we see that $L(E)$ is continuous.
Combining these facts, one sees that there exists a critical value 
\begin{align*}
	E_c = E_c(\alpha,\hat{M}^{d\alpha}_{2,r},\ell_r) &{}:= \sup \{E \ |\ L(E) <\I\} \\
	&{} = \min \{ E \ |\ L(E)=\I\} \in [\delta, \ell_r(Q)].
\end{align*}
By definition, one has
\begin{equation}\label{eq:EcE2}
	E_c  \le E_2.
\end{equation}
Indeed, by definition of $E_2$, for any $\eps>0$, there exists
a solution $v(t)$ with maximal interval $I$
that does not scatter for positive time
direction and $$\limsup_{t\uparrow \sup I} \ell_r(v(t)) \le E_2 + \eps.$$
Then, one can choose $t_0 \in I$ so that
\[
	\sup_{t\in [t_0,\sup I) } \ell_r(v(t)) \le E_2 + 2\eps.
\]
On the other hand, since $v(t)$ does not scatter for positive time direction,
$L(\sup_{t\in [t_0,\sup I) } \ell_r(v(t)))=\I$, proving
\[
	E_c \le \sup_{t\in [t_0,\sup I) } \ell_r(v(t))
\]
and so $E_c \le E_2 + 2\eps$.
Since $\eps>0$ is arbitrary, we obtain \eqref{eq:EcE2}.

Our task is now to show
\begin{theorem}\label{thm:main3b}
Let $d\ge1$ and suppose \eqref{asmp:alpha}.
There exists a solution $v(t)$ to \eqref{eq:NLS} with 
maximal existence interval such that
\begin{enumerate}
\item $v(t)$ does not scatter for both time directions.
\item $\displaystyle\sup_{I_{\max}(v) \cap \{t\ge 0\}} \ell_r(v(t))=\sup_{I_{\max}(v) \cap \{t\le 0\}} \ell_r(v(t)) =E_c(\alpha,\hat{M}^{d\alpha}_{2,r},\ell_r)$.
\item $v(t)$ is almost periodic modulo symmetry as in \eqref{eq:apms}.
\end{enumerate}
\end{theorem}
As an immediate consequence of this theorem, we obtain $E_2 =E_c$.
Indeed, once we obtain a solution $v(t)$ with the first two
properties of the above theorem, it follows that
\[
	E_2 \le \limsup_{t\uparrow \sup I_{\max}(v)} \ell_r(v(t))
	\le \sup_{I_{\max}(v) \cap \{t\ge 0\}} \ell_r(v(t))=E_c.
\]
By means of \eqref{eq:EcE2}, we obtain the desired result.
\subsection{The key convergence result}
For $u:I\times \R^d \to \C$ and $\tau \in I$,
we denote $S_{\ge \tau} (u) := S_{I \cap \{t \ge \tau \}}(u)$, $S_{\le \tau} (u) := S_{I \cap \{t \le \tau \}}(u)$.
\begin{proposition}\label{prop:key_convergence}
Let $u_n:I_n \times \R^d \to \C$ be a sequence of solutions to \eqref{eq:NLS}
such that
\begin{equation}\label{asmp:PS1}
	\limsup_{n\to\I} \sup_{t\in I_n} \ell_r (u_n(t)) = E_c,
\end{equation}
and let $t_n\in I_n$ be a sequence of times such that
\begin{equation}\label{asmp:PS2}
	\lim_{n\to\I} S_{\ge t_n} (u_n) = \lim_{n\to\I} S_{\le t_n} (u_n) =\I.
\end{equation} 
Then, there exist a sequence of deformations $\mathcal{G}_n=\{ D(h_n)P(b_n) T(a_n) \}_n$ and 
a subsequence of $n$ such that $(\mathcal{G}_{n})^{-1}u_n(t_n)$  converges strongly in $\hat{M}^{d\alpha}_{2,r}$ to
a function $\phi\in\hat{M}^{d\alpha}_{2,r}$ along the subsequence.
Further, a solution $\Phi(t)$ of \eqref{eq:NLS} with $\Phi(0) = \phi$ satisfies the first two properties 
in Theorem \ref{thm:main3b}.
\end{proposition}

In the rest of this section, we prove this proposition.
Our argument is in the same spirit as in \cite{KV}. 
By the time translation symmetry of \eqref{eq:NLS}, we may let $t_n=0$.
We apply profile decomposition lemma to yield a decomposition
\[
	u_n(0) = \sum_{j=1}^J \mathcal{G}_n^j \phi^j + w_n^J
\] 
up to subsequence, where $\mathcal{G}_n^j$ is parameterized as in \eqref{eq:nrep} with $\theta_n^j \equiv 0$.

Refining the subsequence and changing notations, we may assume that for each $j$,
the sequence $\{s_n^j\}$ converges to some $s^j \in \{0, \pm \I\}$.
Further, if $s^j=0$ then we may let $s^j_n \equiv 0$.
Let $\Phi^j: I^j \times \R^N \to \C$ be a nonlinear profile associated with $(\phi^j,\{s^j_n\}_n)$,
i.e., 
\begin{itemize}
\item If $s^j=0$ then $\Phi^j(t)$ is a solution to \eqref{eq:NLS} with $\Phi^j(0) = \phi^j$.
\item If $s^j=\I$ (resp. $s^j=-\I$) then $\Phi^j(t)$ is a solution to \eqref{eq:NLS} that scatters to $\phi^j$
for positive time direction (resp. negative time direction).
\end{itemize}
Define
\[
	v_n^j(t) := e^{-i(h_n^j)^2|b_n^j|^2 t} D(h_n^j) P(b_n^j) T(a_n^j- 2b_n^j(h_n^j)^2t  )   \Phi^j((h_n^j)^2t + s_n^j)
\]
and
\[
	\widetilde{u}^J_n(t) := \sum_{j=1}^J v_n^j(t) + U(t)w_n^J.
\]
Remark that $v_n^j$ solves \eqref{eq:NLS} and $v_n^j(0)=D(h_n^j) P(b_n^j) T(a_n^j)  \Phi^j(s_n^j)$.
\begin{lemma}\label{lem:pf2_1}
There exists $j$ such that $\Phi^j(t)$ does not scatter for positive time direction.
\end{lemma}
\begin{proof}[Proof of Lemma \ref{lem:pf2_1}]
Assume for contradiction that $\Phi^j$ scatters for positive time direction for all $j$.
We apply long time stability 
with $\widetilde{u}(t) = \widetilde{u}_n^J(t)$ for large $J$ and $n\ge N(J)$.

We first demonstrate that $\widetilde{u}_n^J(0) - u_{n}(0)\to 0$ in $\hat{M}^{d\alpha}_{2,r}$
as $n\to\I$.
Indeed, we have
\begin{align*}
	\widetilde{u}_n^J(0) - u_{n}(0)
	={}& \sum_{j=1}^J v_n^j(0) - \mathcal{G}_n^j \phi^j \\
	={}& \sum_{j=1}^J 
	D(h_n^j) P(b_n^j) T(a_n^j) U(s_n^j) (U(-s_n^j)\Phi^j(s_n^j)- \phi^j).
\end{align*}
Hence, by definition of nonlinear profile $\Phi^j(t)$ and Lemma \ref{lem:Paction},
\[
	\norm{\widetilde{u}_n^J(0) - u_{n}(0)}_{\hat{M}^{d\alpha}_{2,r}}
	\le 2^d \sum_{j=1}^J \norm{U(-s_n^j)\Phi^j(s_n^j)-\phi^j}_{\hat{M}^{d\alpha}_{2,r}}
	\to 0
\]
as $n\to\I$.

We will show that $\norm{\widetilde{u}_n^J}_{L^{(d+2)\alpha}_{t,x}(\R_+ \times \R)}$ is uniformly bounded
and that the error $e:= i\d_t \widetilde{u}_n^J + \Delta \widetilde{u}_n^J + |\widetilde{u}_n^J|^{2\alpha}\widetilde{u}_n^J$ tends to zero as $n\to \I$ for each $J$.
By the orthogonality, it follows that
\[
	\norm{|v_n^j|^\theta |v_n^k|^{1-\theta}}_{L^{\frac{(d+2)\alpha}2}_{t,x}} \to 0
\]
as $n\to\I$ for any $j\neq k$ and $0<\theta<1$ (see \cite{BG,MV}).
Hence, we see that 
$$\norm{e}_{L^{\frac{(d+2)\alpha}{2\alpha+1}}_{t,x}(\R_+ \times \R)} \to 0$$
as $n\to\I$.
On the other hand, since $\sum_{j=1}^{\I} \norm{\phi^j}_{\hat{M}^{d\alpha}_{2,r}}^r <\I$,
for any $\eps>0$ there exists $J_0=J_0(\eps)$ such that
\[
	\norm{\sum_{j=J_0}^{J} \mathcal{G}_n^j \phi^j}_{\hat{M}^{d\alpha}_{2,r}} \le \eps
\]
for any $J>J_0$ and $n > N(J)$, where we have used Lemma \ref{lem:Paction} and orthogonality of $\mathcal{G}_n^j$. This implies that
\[
	\norm{ \sum_{j=J_0}^{J} v_n^j(0) }_{\hat{M}^{d\alpha}_{2,r}} \le 2\eps
\]
for any $J>J_0(\eps)$ and $n>N(J)$. By small data result, one then sees that
\[
	\norm{ \widetilde{v}_n^J }_{L^\I(\R,\hat{M}^{d\alpha}_{2,r})
	\cap L^{(d+2)\alpha}_{t,x}(\R\times \R^d)} \le C \eps
\]
for such $J$ and $n$, provided $\eps>0$ is sufficiently small,
where $\widetilde{v}_n^J(t)$ is a solution to \eqref{eq:NLS} with $\widetilde{v}_n^J(0)
=\sum_{j=J_0(\eps)}^{J} v_n^j(0)$. As in the long time perturbation, it follows that
\[
	\norm{ \sum_{j=J_0}^{J} v_n^j}_{L^\I(\R,\hat{M}^{d\alpha}_{2,r})
	\cap L^{(d+2)\alpha}_{t,x}(\R\times \R^d)} \le C'\eps
\]
for any $J>J_0$ and $n>N'(J)$. Hence,
\[
	\norm{u_n^J}_{L^{(d+2)\alpha}_{t,x}} \le 
	\sum_{j=1}^{J_0} \norm{\Phi^J}_{L^{(d+2)\alpha}_{t,x}} + 
	\norm{ \sum_{j=J_0}^{J} v_n^j}_{L^{(d+2)\alpha}_{t,x}} + C \norm{w_n^J}_{\hat{M}^{d\alpha}_{2,r}} < \I
\]
for any $J > J_0$ and $n\ge N''(J)$.
This contradicts with the assumption \eqref{asmp:PS2}.
% Thus, $u_n$ scatters for positive time if $n$ is sufficiently large,
% which is a contradiction.
\end{proof}

By the previous result, there exists at least one $\Phi^j$ that blows up for positive time direction.
Renumbering the index $j$ if necessary, we may assume that 
$\Phi^j$ does not scatter for positive time direction if and only if $1 \le j \le J_1$.
Remark that the number $J_1$ is finite because of decoupling inequality and small data scattering.
Also remark that $s^j \neq \I$ for $1 \le j \le J_1$ otherwise it scatters for positive time direction by definition
of $\Phi^j$.

We now prove that $J_1=1$.
For each $m,n \ge 1$ let us define an integer $j(m.n) \in \{1,2,\dots,J_1\}$ and an 
interval $K_n^m$ of the form $[0,\tau]$ by
\[
	\sup_{1\le j \le J_1} S_{K_n^m}(v_n^j) = S_{K_n^m}(v_n^{j(m,n)}) = m.
\]
By the pigeonhole principle, there is a $j_1 \in \{1,2,\dots,J_1\}$ so that for infinitely
many $m$ one has $j(m,n) = j_1$ for infinitely many $n$.
By reordering the indices, we may assume that $j_1=1$.
Then, by definition of $E_2$ and \eqref{eq:EcE2},
\begin{equation}\label{eq:1stprofile}
	\limsup_{m\to\I} \limsup_{n\to\I} \sup_{t\in K_n^m} \ell_r(v_n^1(t))
	\ge E_2 \ge E_c.
\end{equation}

\begin{lemma}\label{lem:pf2_2}
$\psi^j\equiv 0$ for $j \ge 2$. And, $w_n^1 \to 0$ in $\hat{M}^{d\alpha}_{2,r}$ as 
$n\to\I$.
\end{lemma}
\begin{proof}[Proof of Lemma \ref{lem:pf2_2}]
In light of long time stability,
it holds for each $m$ that
\begin{equation}\label{eq:approximation}
	\limsup_{J\to\I} \limsup_{n\to\I}\sup_{t\in K_n^m}\norm{u_n(t) - \widetilde{u}_n^J(t)}_{\hat{M}^{d\alpha}_{2,r}} = 0.
\end{equation}
% This follows by means of long time stability.
% To see this,
Here we remark that, by definition of $K_n^m$, we have
\[
	\sup_{n} %\norm{v_n^j}_{L^{(d+2)\alpha}_{t,x}(K_n^m \times \R^d)} 
	S_{K_n^m} (v_n^j)
	\le m
\]
for $1 \le j \le J_1$ and so the assumption of long time stability is fulfilled.
For $j>J_1$, we have $S_{\ge0}(v^n_j) \le C_j < \I$.
Define
\[
	c_j := \inf_{t \in I_{\max}(\Phi^j)} \ell_r(\Phi^j(t)).
\]
We shall show that $c_j=0$ for $j\ge2$.
% We claim that for any $\eps>0$, $m\ge1$, $J\ge1$, and $\gamma\in (0,1)$,
% there exists $N$ such that
% \begin{equation}\label{eq:decouple}
% 	\sup_{t \in K_n^m} \ell(\widetilde{u}_n^J)^\sigma
% 	\ge \gamma \sup_{t\in K_n^m}\ell(v_n^1(t))^\sigma +  \sum_{j=2}^J \gamma^j c_j^\sigma 
% 	+\gamma^J \ell(w_n^J)^\sigma -\eps
% \end{equation}
% holds for $n\ge N$.

% Here, we will see that \eqref{eq:1stprofile}, \eqref{eq:approximation}, and \eqref{eq:decouple}
% yield $c_j=0$ for $j\ge2$. 
Assume for contradiction that $c_{j_0} >0$ for some $j_0\ge2$.
By means of \eqref{eq:1stprofile}, for any $\eps>0$ there exists $m=m(\eps)$ such that 
\[
	\sup_{t \in K_n^m} \ell_r(v_n^1(t))^r \ge E_c^r -\eps
\]
for a subsequence of $n$. 
The subsequence depends on $m$ and is denoted again by $n$.

Fix $m$.
Then, for the same $\eps$, we can choose $J=J(\eps,m)=J(\eps)$ so that
\[
	\sup_{t\in K^m_n}\norm{u_n(t)-\widetilde{u}_n^J}_{\hat{M}^{d\alpha}_{2,\sigma}}
	\le \eps
\]
as long as $n\ge N(\eps,m,J)=N(\eps)$ by using \eqref{eq:approximation}.
Without loss of generality, we may assume that $J>j_0$. 

Fix $J$. 
Set
\[
	\widetilde{w}_n^k(t) := \sum_{j=k+1}^J v_n^j(t) + e^{it\Delta} w_n^J.
\]
for $k=0,1,2,\dots,J-1$ and $\widetilde{w}_n^J (t) = e^{it\Delta} w_n^J$.
Remark that $\widetilde{w}_n^0 = \widetilde{u}_n^J$.
For each $ j =0,1,2,\dots,J$, we may show that 
for any given sequence $\{t_n\}_n$ such that $t_n \in K_n^m$
there exists a subsequence of $n$, which is again denoted by $n$, such that
\begin{equation}\label{eq:decouple2}
	\ell_r(\widetilde{w}_n^j(t_n))^r  
	- \gamma \ell_r(v_n^{j+1}(t_n))^r - \gamma \ell_r(\widetilde{w}_n^{j+1}(t_n))^r \ge o_n(1)
\end{equation}
for any $0<\gamma<1$. %In what follows we assume that $1<\gamma <2$.

Before the proof of \eqref{eq:decouple2}, we shall complete the proof of the lemma. 
Once inequality \eqref{eq:decouple2} is proven, we deduce for any $\{t_n\}_n$ with $t_n \in K_n^m$
that
\begin{equation}\label{eq:decouple1}
	\ell_r (\widetilde{u}_n^J(t_n))^r  
	- \sum_{j=1}^J \gamma^{j}\ell_r(v_n^j(t_n))^r - \gamma^J \ell_r(\widetilde{w}_n^{J}(t_n))^r \ge o_n(1)
\end{equation}
holds up to subsequence. 
Now, choose a sequence $\{t_n\}_n$ so that 
\[
	\ell_r(v^1_n(t_n))^r \ge \sup_{t\in K_n^m} \ell_r(v_n^1)^r -\frac{\eps}{2}.
\]
Then, by means of \eqref{eq:decouple1}, extracting subsequence of $n$, one verifies that
\[
	\sup_{t\in K_n^m}\ell_r(\widetilde{u}_n^J)^r 
	\ge \gamma \sup_{t\in K_n^m} \ell_r(v_n^1)^r + \sum_{j=2}^J \gamma^j c_j^r
	+ \gamma^J \ell_r(w_n^J)^r - \frac{\gamma }2\eps + o_n(1),
\]
where we have used 
\[
	\ell_r(\widetilde{u}_n^J(t_n))^r \le \sup_{t\in K_n^m}\ell_r(\widetilde{u}_n^j(t))^r,
\]
$\ell_r(v_n^j(t)) \ge c_j$ for $j\ge2$, and $\ell_r(\widetilde{w}_n^{J}(t_n))=\ell_r({w}_n^{J})$.
Hence, by definition of $m$, for large $n$, 
\[
	\sup_{t\in K_n^m} \ell_r(\widetilde{u}_n^J)^r
	\ge \gamma \sup_{t\in K_n^m}\ell_r(v_n^1)^r +\sum_{j=2}^J \gamma^j c_j^r - \eps
	\ge \gamma E_c^r + \gamma^{j_0} c_{j_0}^r -2\eps.
\]
By assumption \eqref{asmp:PS1}, we also have
\[
	\sup_{t\in K_n^m} \ell_r(u_n(t)) \le E_c + \eps
\]
for large $n$.
Thus, for sufficiently large $n$, we have
\begin{align*}
	E_c+ \eps \ge {}& \sup_{t\in K_n^m} \ell_r(u_n(t)) \\
	\ge {}& \sup_{t\in K_n^m} \ell_r(\widetilde{u}_n^J(t)) - C \sup_{t\in K_n^m} \norm{u_n(t)-\widetilde{u}_n^J(t)}_{\hat{M}^{d\alpha}_{2,r}} \\
	\ge {}& (\gamma E_c^r  + \gamma^{j_0} c_{j_0}^r -2\eps)^{\frac1r}-C \eps,
\end{align*}
that is,
\[
	c_{j_0}^r \le C \eps + \gamma^{-j_0}(1-\gamma) E_c^r,
\]
which is a contradiction when $\eps$ is sufficiently small and $\gamma$ is sufficiently close to one.
Hence, $\phi^j\equiv 0$ for $j\ge2$.
Once we know $\phi^j \equiv 0 $ for $j\ge 2$, we see that $w_n^J = w_n^1$ for all $J$.
Arguing as above, one sees that
\[
	\limsup_{n\to\I}\norm{w_n^1}_{\hat{M}^{d\alpha}_{2,r}} 
	= \limsup_{J\to\I}\limsup_{n\to\I}\norm{w_n^J}_{\hat{M}^{d\alpha}_{2,r}} \le \eps
\] 
for any $\eps>0$. Hence, $\lim_{n\to\I}\norm{w_n^1}_{\hat{M}^{d\alpha}_{2,r}} = 0$.

Thus, it suffices to show \eqref{eq:decouple2} to complete the proof. 
We first note that
$\widetilde{w}_n^j = v_n^j + \widetilde{w}_n^{j+1}$ and so
\begin{align*}
	\norm{\F \widetilde{w}_n^j(t_n)}_{L^2(\tau)}^2
	={}& \norm{\F v_n^{j+1}(t_n)}_{L^2(\tau)}^2 + \norm{\F \widetilde{w}_n^{j+1}(t_n)}_{L^2(\tau)}^2\\
	&{}+2\Re \Jbr{\F v_n^{j+1}(t_n),\F \widetilde{w}_n^{j+1}(t_n)}_{\tau}
\end{align*}
for each dyadic cube $\tau\in \mathcal{D}$, where $\Jbr{f,g}_\tau = \int_{\tau} f(x) \overline{g}(x) dx$.
By an elementary inequality
\[
	(a-b)^{\frac{\sigma}{2}} \ge \(\frac{m}{m+1}\)^{\frac{\sigma-2}{2}} a^{\frac{\sigma}2} - m^{\frac{\sigma-2}{2}} b^{\frac{\sigma}{2}}
\]
for any $a\ge b\ge0$ and $m>0$ and by embedding
$\ell^2_{\mathcal{D}} \hookrightarrow \ell^{r}_{\mathcal{D}}$, it follows that
\begin{align*}
	&	\sum_{\tau \in \mathcal{D}} |\tau|^{r(\frac12-\frac1{d\alpha})} \norm{\widetilde{w}_n^j}_{L^2(\tau)}^r \\
	={}& \sum_{\tau \in \mathcal{D}} |\tau|^{r(\frac12-\frac1{d\alpha})} \( \norm{\F v_n^{j+1}}_{L^2(\tau)}^2 + \norm{\F \widetilde{w}_n^{j+1}}_{L^2(\tau)}^2 + 2\Re
	\Jbr{\F v_n^{j+1}, \F \widetilde{w}_n^{j+1}}_{\tau} \)^{\frac{r}2}  \\
	\ge{}& \sum_{\tau \in\mathcal{D}} |\tau|^{r(\frac12-\frac1{d\alpha})} \( \norm{\F v_n^{j+1}}_{L^2(\tau)}^2 + \norm{\F \widetilde{w}_n^{j+1}}_{L^2(\tau)}^2 - 2
	\abs{\Jbr{\F v_n^{j+1}, \F \widetilde{w}_n^{j+1}}_{\tau}} \)^{\frac{r}{2}} \\
	\ge{}& \(\frac{m}{m+1}\)^{\frac{r-2}{2}} \sum_{\tau \in\mathcal{D}} |\tau|^{r(\frac12-\frac1{d\alpha})} \( \norm{\F v_n^{j+1}}_{L^2(\tau)}^2 + \norm{\F \widetilde{w}_n^{j+1}}_{L^2(\tau)}^2 \)^{\frac{r}{2}}\\
	&{}- 2^{\frac{r}{2}}m^{\frac{r-2}{2}} \sum_{\tau \in\mathcal{D}} |\tau|^{r(\frac12-\frac1{d\alpha})} \abs{\Jbr{\F v_n^{j+1}, \F \widetilde{w}_n^{j+1}}_{\tau}}^{\frac{r}{2}} \\
	\ge{}& \(\frac{m}{m+1}\)^{\frac{r-2}{2}}\( \sum_{\tau \in\mathcal{D}} |\tau|^{r(\frac12-\frac1{d\alpha})} \norm{\F v_n^{j+1}}_{L^2(\tau)}^r + \sum_{\tau \in\mathcal{D}} |\tau|^{r(\frac12-\frac1{d\alpha})} \norm{\F \widetilde{w}_n^{j+1}}_{L^2(\tau)}^{r} \) \\
	&{}- 2^{\frac{r}{2}}m^{\frac{r-2}{2}} \sum_{\tau \in\mathcal{D}} |\tau|^{r(\frac12-\frac1{d\alpha})} \abs{\Jbr{\F v_n^{j+1}, \F \widetilde{w}_n^{j+1}}_{\tau}}^{\frac{r}{2}},
\end{align*}
where we have omitted the time variable $(t_n)$ in the above estimate.
Hence, the equation \eqref{eq:decouple1} follows if we show
\begin{equation}\label{eq:pf_last}
	\sum_{\tau \in\mathcal{D}} |\tau|^{r(\frac12-\frac1{d\alpha})} \abs{\Jbr{\F v_n^{j+1}(t_n), \F \widetilde{w}_n^{j+1}(t_n)}_{\tau}}^{\frac{r}{2}} \to 0
\end{equation}
as $n\to\I$ up to subsequence. %For simplicity, we drop upper index $j+1$.

We now claim that it suffices to show the above convergence with replacing
$v_n^{j+1}(t_n)$ with
\[
	D(h_n^{j+1})T(a_n^{j+1}- 2b_n^{j+1}(h_n^{j+1})^2t_n^{j+1}  )  P(b_n^{j+1})  U((h_n^{j+1})^2t_n^{j+1} + s_n^{j+1}) f
\]
% if $s=\pm \I$, 
for suitable $f$.
For simplicity, we drop upper index $j+1$ for a while.
To this end, we shall recall that
\[
	v_n(t_n) = e^{i\theta_n}D(h_n) P(b_n) T(a_n- 2b_n(h_n)^2t_n  )  \Phi((h_n)^2t_n + s_n),
\]
with suitable $\theta_n \in \R$. In view of \eqref{eq:pf_last}, we may neglect $e^{i\theta_n}$.
By extracting subsequence, we may suppose that $(h_n)^2 t_n + s_n$ converges
to $T \in \overline{I_{\max}(\Phi)} \subset [-\I,\I]$.
We first consider the case $T$ is interior of $\overline{I_{\max}(\Phi)}$.
In this case, $\Phi((h_n)^2t_n + s_n)$ converges strongly to $\Phi(T)$ in $\hat{M}^{d\alpha}_{2,r}$.
Hence, we may replace $\Phi((h_n)^2t_n + s_n)$ by $U((h_n)^2t_n + s_n) (U(-T)\Phi(T))$.
Namely, we take $f=U(-T)\Phi(T)$.
% Further, $s_n \equiv 0$ since $K_n^m$ is of the form $[0,\tau]$.
If $T=\sup I_{\max}(\Phi)$ then $T=\I$ and $\Phi$ must scatters for positive time direction
because $t_n$ is taken from $K_n^m$.
Hence, we may replace $\Phi((h_n)^2t_n + s_n)$ by $U((h_n)^2t_n + s_n) \Phi_+$
for some $\Phi_+ \in \hat{M}^{d\alpha}_{2,r}$.
This implies that the choice $f=\Phi_+$ works.
The case $T= \inf I_{\max}(\Phi)$ is handled similarly.
Since $t_n\ge0$ and $s_n\in I_{\max}(\Phi)$,
this case occurs only if $s_n\to-\I$ as $n\to\I$ and $T=-\I$.
So, we may replace $\Phi((h_n)^2t_n + s_n)$ by $U((h_n)^2t_n + s_n) \Phi_-$
for some $\Phi_- \in \hat{M}^{d\alpha}_{2,r}$. The claim is true.
% By extracting subsequence, we may suppose that $(h_n)^2 t_n + s_n$ converges
% to $T \in \overline{I_{\max}(\Phi)} \subset [-\I,\I]$.
% We first consider the case $T$ is interior of $\overline{I_{\max}(\Phi)}$.
% In tis case, $\Phi((h_n)^2t_n + s_n)$ converges strongly to $\Phi(T)$ in $\hat{M}^{d\alpha}_{2,r}$.
% Hence, we may replace $\Phi((h_n)^2t_n + s_n)$ by $\Phi(T)$.
% Namely, we take $f=\Phi(T)$.
% Further, $s_n \equiv 0$ since $K_n^m$ is of the form $[0,\tau]$.
% If $T=\sup I_{\max}(\Phi)$ then $\Phi$ scatters forward in time (and hence $j \ge J_1$)
% and $s_n \to \I$.
% Hence, we may replace $\Phi((h_n)^2t_n + s_n)$ by $S((h_n)^2t_n + s_n) \Phi_+$
% for some $\Phi_+ \in \hat{M}^{d\alpha}_{2,r}$.
% This implies that the choice $f=\Phi_+$ works.
% The case $T= \inf I_{\max}(\Phi)$ is similar.
% The claim is true.
% In what follows, we consider the case $T=\sup I_{\max}(\Phi)$.
% The other cases are handled in a similar way.
Remark that
\[
	D(h_n)  P(b_n)  T(a_n- 2b_n(h_n)^2t_n  )  U((h_n)^2t_n + s_n) f
	= e^{i\gamma_n}U( t_n) \mathcal{G}_n f
\]
for suitable $\gamma_n \in \R$.

We shall show
\[
	\sum_{\tau \in\mathcal{D}} |\tau|^{r(\frac12-\frac1{d\alpha})} \abs{\Jbr{\F f,
 \F r_n  }_{\tau}}^{\frac{r}{2}} \to 0
\]
as $n\to\I$, where
\[
	r_n := (\mathcal{G}_n^{j+1})^{-1} U(-t_n)\widetilde{w}_n^{j+1}(t_n).
\] 
Since $f \in \hat{M}^{d\alpha}_{2,r}$ and $r_n$ is uniformly bounded in $\hat{M}^{d\alpha}_{2,r}$,
for any $\eps>0$ there exists a finite set of dyadic cubes 
$\Omega \subset \mathcal{D}$ independent of $n$ such that
\begin{align*}
		\sum_{\tau \in\mathcal{D} \setminus \Omega} |\tau|^{r(\frac12-\frac1{d\alpha})} \abs{\Jbr{\F f,
 \F r_n  }_{\tau}}^{\frac{r}{2}}
&{}\le \( \sum_{\tau \in\mathcal{D} \setminus \Omega} |\tau|^{r(\frac12-\frac1{d\alpha})} \norm{\F{f}}_{L^2(\tau)}^r\)^{\frac12}
\norm{r_n}_{\hat{M}^{d\alpha}_{2,r}}^{\frac{r}2} \\
 &{}\le \eps.
\end{align*}
Hence, the proof is reduced to showing that
\[
	\Jbr{\F f, \F r_n}_{\tau} \to 0
\]
as $n\to\I$ for each dyadic cube $\tau$.
A similar argument as in the previous paragraph allows us to replace $\widetilde{w}_n^{j+1}$ with
\[
	 U(t_n) \(\sum_{k=j+2}^J \mathcal{G}_n^k f^k + w_n^J \).
\]
With this replacement, it suffices to show
\[
	\Jbr{\F f, \F \( \sum_{k=j+2}^J (\mathcal{G}_n^{j+1})^{-1}\mathcal{G}_n^k f^k + (\mathcal{G}_n^{j+1})^{-1} w_n^J\)}_{\tau}
	\to 0
\]
as $n\to\I$. The desired convergence now follows 
form mutual orthogonality of families $\{\mathcal{G}_n^{j}\}_n \subset G$ ($j=1,2,3,\dots$) and
from weak-$*$ convergence 
$(\mathcal{G}_n^{j+1})^{-1} w^J_n \rightharpoonup 0$ as $n\to\I$ in $\hat{M}^{d\alpha}_{2,r}$.
\end{proof}

Let us finish the proof of Proposition \ref{prop:key_convergence}. So far, we obtain
\[
	u_n(0) = \mathcal{G}_n^1 \phi^1 + o(1)
\]
in $\hat{M}^{d\alpha}_{2,r}$ as $n\to\I$.
Further, $s_n^1\neq +\I$ and the nonlinear profile $\Phi^1(t)$ %with $\Phi^1(0)=\psi$ 
does not scatter for positive time direction and so
$\sup_{I_{\max}(\Phi^1 )\cap \{t\ge 0\} } \ell_r(\Phi^1(t)) \ge E_c$.
By the assumption \eqref{asmp:PS1}, we see that
\[
	\sup_{I_{\max}(\Phi^1 )\cap \{t\ge 0\} } \ell_r(\Phi^1(t)) = E_c
\]
by the stability (or arguing as in Lemma \ref{lem:pf2_2}).

Remark that we did not yet use the assumption $\lim_{n\to\I} S_{\le 0} (u_n) = \I$.
Arguing as in Lemmas \ref{lem:pf2_1} and \ref{lem:pf2_2},
this assumption implies that $\Phi^1(t)$ have the same property for negative time direction.
Namely, $s^1 \neq -\I$, $\Phi^1(t)$ does not scatter for negative time direction, and
\[
	\sup_{I_{\max}(\Phi^1 )\cap \{t\le 0\} } \ell_r(\Phi^1(t)) = E_c.
\]
We conclude that $\Phi^1(t)$ is the desired solution.

\subsection{Almost periodicity modulo symmetry}
In this section, we complete the proof of Theorem \ref{thm:main3b}.
By definition of $E_c$, we can take a sequence of solutions $u_n: I_n \times \R^d \to \C$ that satisfies the
assumption of Proposition \ref{prop:key_convergence}.
Therefore, thanks to Proposition \ref{prop:key_convergence}, we obtain a maximal-lifespan 
solution $v(t)$ that satisfies
that satisfies the first two properties of Theorem \ref{thm:main3b}.
Then, $E_2=E_c$ follows.
Thus, let us prove that $v(t)$ is almost periodic modulo symmetry as in \eqref{eq:apms}.
To this end, the main step is the following.
\begin{proposition}
There exists $a(t):I_{\max}(v)\to\R^d$, $b(t):I_{\max}(v)\to\R^d$,
$\lambda(t):I_{\max}(v) \to 2^\Z$ such that the set 
\[
	\left\{ (D(\lambda(t)) P(b(t)) T(a(t)) )^{-1} v(t)\ \middle|\ t\in I_{\max}(v) \right\} \subset \hat{M}^{d\alpha}_{2,r}
\]
is totally bounded in $\hat{M}^{d\alpha}_{2,r}$. 
\end{proposition}
Indeed, the property \eqref{eq:apms} for $v(t)$ then follows from the characterization of total boundedness (Theorem \ref{thm:totallybounded})
with $N(t)=\lambda(t)$, $y(t)=a(t)/\lambda(t)$, and $z(t)=\lambda(t)b(t)$.
\begin{remark}
By using Proposition \ref{prop:key_convergence},
we see that
for any sequence $\{\tau_n\}_n \subset I_{\max}(v)$ there exists a sequence of parameters $(\lambda_n,a_n,b_n) \in 2^\Z \times \R^d \times \R^d$
such that $(D(\lambda_n) P(b_n)T(a_n) )^{-1}v(\tau_n)$ possess a convergence subsequence.
However, this statement is weaker.
We have to choose parameters $a(t)$, $b(t)$, and $\lambda(t)$ independently of choice of a sequence $\{\tau_n\}_n$.
\end{remark}

\begin{proof}
{\bf Step 1}. We first construct $\lambda(t)$, $a(t)$, and $b(t)$.
Fix $\sigma \in I_{\max}(v)$.  
For simplicity, we omit time variable $\sigma$ and write $v=v(\sigma)$ in this step.
Since $v(t)$ does not scatter for positive time direction,
it holds that
\begin{equation}\label{eq:aspm_pf0}
% 	\norm{e^{it\Delta} v}_{L^{(d+2)\alpha}_{t,x}(\R\times \R^d)}
	S_\R (e^{it\Delta} v)
	\ge \delta,
\end{equation}
where $\delta$ is the number given in Lemma \ref{lem:LWP}.
Moreover, we have
\[
	\norm{v}_{\hat{M}^{d\alpha}_{2,r}} \le 2^d E_c <\I.
\]

Mimicking the proof of Lemma \ref{lem:pd:step1}, we see that for any $\eps>0$ there exist
a sequence of dyadic cubes $\{ \tau_m \}_{m=1}^M \subset \mathcal{D}$, $M=M(\eps)$, 
and constant $C_j=C_j(\eps)>0$ ($j=1,2$)
such that if we define $f^m(x)$ by
\[
	\F f^m(\xi) := \F v (\xi) \times  {\bf 1}_{\tau_m \setminus (\cup_{k=1}^{m-1} \tau_k)} (\xi)
	\times {\bf 1}_{\{|\F v(\xi)| \le C_1 |\tau_m|^{-1/(d\alpha)'}\}}(\xi)
\]
then it holds that
\[
	|\tau_m|^{\frac1{(d\alpha)'}-\frac12} \norm{\F f^m}_{L^2(\tau_m) } \ge C_2
\]
for each $m=1,2,\dots,M$ and that
\[
	\norm{e^{it\Delta}\(v-\sum_{m=1}^M f^m\)}_{L^{(d+2)\alpha}_{t,x}(\R \times \R^d)} \le \eps.
\]
Choose $\eps = \delta/3$. Then, at least one $f^m$ satisfies
\[
% 	\norm{e^{it\Delta} f^m}_{L^{(d+2)\alpha}_{t,x}(\R \times \R^d)} 
	S_\R(e^{it\Delta} f^m)
	\ge \frac{\delta}{2M(\delta/3)} =:\delta_1.
\]
Indeed, otherwise we obtain a contraction with \eqref{eq:aspm_pf0}.
Pick such $m=m_0$ and define $\lambda(\sigma) \in 2^\Z$ and $b(\sigma) \in \Z^d$ by the relation
\[
	\tau_{m_0} = \lambda(\sigma) ([0,1)^d + b(\sigma)).
\]

Set $g(x)=[P(b(\sigma))^{-1} D(\lambda(\sigma))^{-1} f^{m_0}](x)$. Then,
\[
	|\F g(\xi)| = (\lambda(\sigma))^{d-\frac1{\alpha}} |\F f^{m_0}(\lambda(\sigma)(\xi + b(\sigma)))|
	\le C_1 {\bf 1}_{[0,1)^d} (\xi).
\]
Arguing as in \cite[Lemma 22]{Bo2}, we see that for any $\eps>0$ there exists a sequence of disjoint unit cubes
$\{Q_k \}_{k=1}^{K} \subset \R \times \R^d$, $K=K(\eps)$, such that
\[
	\norm{e^{it\Delta}g}_{L^{(d+2)\alpha}_{t,x} \((\R \times \R^d) \setminus \bigcup_{k=1}^K Q_k \)} \le \eps.
\]
Take $\eps = \delta_1/3$. Then, at least for one $k$, we have
\[
	\norm{e^{it\Delta}g}_{L^{(d+2)\alpha}_{t,x} (Q_k)} \ge \frac{\delta_1}{2K(\delta_1/3)} =:\delta_2
\]
Choose such $k=k_0$ and define $a(\sigma)$ as the $x$-coordinate of the center of $Q_{k_0}$.

{\bf Step 2}. We now prove that the above $\lambda(t)$, $a(t)$, and $b(t)$ give the desired conclusion.
Take any sequence $\{t_n\} \subset I_{\max}(v)$, since the assumption of Proposition \ref{prop:key_convergence}
is satisfied with $u_n\equiv v$ and this $\{t_n\}$, up to subsequence, we have 
\[
	v(t_n) = D(\lambda_n) P(b_n) T(a_n) \phi + o(1)
\]
in $\hat{M}^{d\alpha}_{2,r}$ as $n\to\I$.
Changing parameters and notations, and refining subsequence if necessary, we can rewrite 
the above convergence as
\begin{equation}\label{eq:construction_pf}
	T(a(t_n))^{-1}P(b(t_n))^{-1} D(\lambda(t_n))^{-1} v(t_n) =  T(a_n)^{-1} P(b_n)^{-1}  D(\lambda_n)^{-1} \phi + o(1)
\end{equation}
in $\hat{M}^{d\alpha}_{2,r}$ as $n\to\I$. Without loss of generality, we may assume that $a_n, b_n\in \Z^d$.
Now, it suffices to show that $\abs{\log \lambda_n } + |a_n| + |b_n|$ is bounded (under the new representation).

Set $w(t)=T(a(t))^{-1}P(b(t))^{-1} D(\lambda(t))^{-1} v(t)$.
By definition of $\lambda(t)$ and $b(t)$, we have
\[
	\norm{ {\bf 1}_{\{|\F w(t)| \le C_1\}} (\F w(t))   }_{L^2([0,1)^d)} \ge C_2
\]
for any $t \in I_{\max}(v)$. By \eqref{eq:construction_pf}, there exists $n_0 $ such that
\begin{align*}
	\frac{C_2}2
	&{}\le \norm{ {\bf 1}_{\{|\F w(t_n)| \le C_1\}} \F [ T(a_n)^{-1} P(b_n)^{-1}  D(\lambda_n)^{-1}  \phi ] }_{L^2([0,1)^d)} \\
	&{}\le \norm{\F [ T(a_n)^{-1} P(b_n)^{-1} D(\lambda_n)^{-1}  \phi] }_{L^2([0,1)^d)} \\
	&{}= \lambda_n^{d(\frac12 -\frac1{d\alpha})} \norm{\F \phi }_{L^2(\lambda_n([0,1)^d+b_n))}
\end{align*}
for $n \ge n_0$. Let $N_0$ be the number of the dyadic cubes $\tau \in \mathcal{D}$ such that
\[
	|\tau|^{\frac12-\frac1{d\alpha}} \norm{\F \phi }_{L^2(\tau)} \ge \frac{C_2}{2}.
\]
Then, $N_0$ is bounded because
\[
	N_0 \(\frac{C_2}{2}\)^r \le \norm{\phi}_{\hat{M}^{d\alpha}_{2,r}}^r \le (2^d E_c)^r.
\]
Therefore, $\# \{ \lambda_n ([0,1)^d + b_n ) \in \mathcal{D}\ |\ n \ge 1\} \le n_0 + N_0 <\I$.
Hence, $\log |\lambda_n| + |b_n|$ is bounded.

Refining subsequence and changing notations, we may suppose that $\lambda_n\equiv 1$ and $b_n \equiv 0$.
By definition of $a(t)$, we have
\[
 \norm{e^{it\Delta} \F^{-1} \( {\bf 1}_{A_n} \F w(t_n) \) }_{L^{(d+2)\alpha}_{t,x} (\R \times [-1/2,1/2]^d))} \ge \delta_2,
\]
where $A_n \subset [0,1)^d$ is a suitable set depending only on $v(t_n)$.
By \eqref{eq:construction_pf} and Strichartz' estimate, there exists $n_1$ such that
\[
 \norm{e^{it\Delta} \F^{-1} \( {\bf 1}_{A_n} \F (T(a_n)^{-1} \phi) \) 
}_{L^{(d+2)\alpha}_{t,x} (\R \times [-1/2,1/2]^d)} \ge \frac{\delta_2}2
\]
holds for all $n \ge n_1$.
Let $N_1(n)$ be the number of vectors $a \in \Z^d$ such that
\[
	 \norm{e^{it\Delta} \F^{-1} \( {\bf 1}_{A_n} \F \phi \) }_{L^{(d+2)\alpha}_{t,x} (\R \times ([-1/2,1/2]^d + a))} \ge \frac{\delta_2}2.
\]
Then, $N_1(n)$ is uniformly bounded because
\begin{align*}
	N_1(n) \(\frac{\delta_2}2\)^{(d+2)\alpha}
	&{}\le \norm{e^{it\Delta} \F^{-1} \( {\bf 1}_{A_n} \F \phi \) }_{L^{(d+2)\alpha}_{t,x} (\R \times \R^d)}^{(d+2)\alpha}\\
	&{}\le C \norm{ \F^{-1} \( {\bf 1}_{A_n} \F \phi \) }_{\hat{M}^{d\alpha}_{2,r}}^{(d+2)\alpha}\\
	&{}\le C \norm{ \phi }_{\hat{M}^{d\alpha}_{2,r}}^{(d+2)\alpha} 
	\le C (2^d E_c)^{(d+2)\alpha}.
\end{align*}
Hence, $\# \{a_n\ |\ n\ge1 \} \le n_1 + \sup_{n} N_1(n) <\I$. In particular,
$a_n$ is bounded. %We obtain the result.
\end{proof}

\section{Proof of Theorem \ref{thm:main2}}\label{sec:main2}

By definition of $E_1$, we can take a sequence $u_{0,n} \in \hat{M}^{d\alpha}_{2,r}$
of initial data which satisfies
the following two properties:
\begin{enumerate}
\item a solution $u_n(t)$ such that $u_n(0) = u_{0,n}$ does not scatter for positive time direction.
Or equivalently,
\begin{equation}\label{eq:pf1_1}
% 	\norm{u_n(t)}_{S((0,\sup I_{\max}(u_n)))} =\I.
	S_{\ge 0} (u_n) = \I.
\end{equation}
\item It has a size slightly bigger than $E_1$ i.e.
\begin{equation}\label{eq:pf1_2}
	\ell_r (u_{0,n}) \le E_1 + \frac1n.
\end{equation}
\end{enumerate}
We argue as in the proof of Proposition \ref{prop:key_convergence} to obtain the result.
See \cite{M1,M2,MS2} for details.
We only comment the following two respects.

The first is that the case $s_n^j\to-\I$ may happen.
This difference comes from the fact that we do not necessarily have 
\[
	\sup_n S_{\le 0}(u_n)=\I.
\]
Remark that it is possible to choose a sequence $\{t_n\}_n$ so that a new sequence
$\widetilde{u}_{0,n} := u_n( t_n)$ satisfies this assumption and \eqref{eq:pf1_1}.
% as in the proof of Proposition \ref{prop:key_convergence}. 
However, we then lose the assumption \eqref{eq:pf1_2} in general.
In other words, a time translation argument is forbidden by the assumption \eqref{eq:pf1_2}.
We also remark that if $s_n^1 \equiv 0$ then (2)-(a) of Theorem \ref{thm:main2} occurs and
if $s_n^1\to-\I$ then (2)-(b) of Theorem \ref{thm:main2} takes place.

The second is that the proofs of $\phi^j \equiv0$ for $j\ge2$ and $w_n^1\to0$ in $\hat{M}^{d\alpha}_{2,r}$ as $n\to\I$
are much simpler.
By the linear profile decomposition, we obtain 
\[
	u_{0,n} = \sum_{j=1}^J \mathcal{G}_n^j \phi^j + w_n^J
\] 
Then, the same argument as in the proof of Lemma \ref{lem:pf2_1} shows that
one of $\Phi^j(t)$, a nonlinear profile associated with $(\phi^j,s_n^j)$, does not scatter for positive time direction.
Suppose that $\Phi^{1}(t)$ is the nonlinear profile does not scatter.
Then, by definition of $E_1$, we have $\ell_r(\phi^j) \ge E_1$.
Combining the decoupling inequality, we immediately obtain the conclusion.

\appendix

\section{Proof of Theorem \ref{thm:previous2}}\label{sec:energycrit}

In this section, we prove Theorem \ref{thm:previous2}.
\begin{proof}
By \cite[Theorem 1.7]{KV}, we see that $E_2 = \norm{W}_{\dot{H}^1}$ for $d\ge 5$.
For $d=4$, the same conclusion is obtained by \cite[Theorem 1.16]{KV}
and \cite[Theorem 1.7]{Do2}.

Let us prove that $E_1=\sqrt{2/d} \norm{W}_{\dot{H}^1}$. We first remark that $\frac1d \norm{W}_{\dot{H}^1}^2 = E[W]$.
If a nontrivial $u_0 \in \dot{H}^1$ satisfies
$\norm{u_0}_{\dot{H}^1} \le \sqrt{2/d}\norm{W}_{\dot{H}^1}$ then it holds that
\[
	E[u_0] = \frac12 \norm{u_0}_{\dot{H}^1}^2 -\frac{d-2}{2d} \norm{u_0}_{L^{\frac{2d}{d-2}}}^{\frac{2d}{d-2}}
	< \frac12 \norm{u_0}_{\dot{H}^1}^2 \le \frac1d \norm{W}_{\dot{H}^1}^2 = E[W].
\]
Thus, $\norm{u_0}_{\dot{H}^1}<\norm{W}_{\dot{H}^1}$ and $E[u_0]<E[W]$ hold.
Then, in light of \cite[Corollary 1.9]{KV} and \cite[Theorem 1.5]{Do2},
a corresponding solution $u(t)$ is global and scatters for both time directions.
Hence, we obtain $E_1 \ge \sqrt{2/d} \norm{W}_{\dot{H}^1}$.

On the other hand, \cite[Theorem 1]{DM} and \cite[Theorem 1.3 and Proposition 1.5]{LZ} show that
there exists a radial global solution $W_-(t)$ such that
\begin{enumerate}
	\item $E[W_-] = E[W]$;
	\item $W_-(t)$ scatters for negative time direction;
	\item $W_-(t)$ converges to $W$ exponentially as $t\to\I$, that is,
	there exists positive constants $c$ and $C$ such that
	\[
		\norm{W_-(t) - W}_{\dot{H}^1(\R^d)} \le C e^{-ct}
	\]
	for all $t\ge0$. In particular, $\lim_{t\to\I} \norm{W_-(t)}_{\dot{H}^1}=\norm{W}_{\dot{H}^1}=E_2$.
\end{enumerate}
for $d\ge 3$. 
Since $W_-(t)$ scatters for negative time direction, we have $\norm{W_-(t)}_{L^{\frac{2d}{d-2}}} \to 0$
as $t\to-\I$. Combining with the energy conservation, we deduce that
\[
	\frac12 \norm{W_-(t)}_{\dot{H}^1}^2 = E[W] + \frac{d-2}{2d} \norm{W_-(t)}_{L^{\frac{2d}{d-2}}}^{\frac{2d}{d-2}}
	\to E[W]
\]
as $t\to-\I$. Namely,
\[
	\norm{W_-(t)}_{\dot{H}^1} \to \sqrt{2E[W]} = \sqrt{\frac2d} \norm{W}_{\dot{H}^1}
\]
as $t\to-\I$. Since $\norm{W_-(t)}_{\dot{H}^1}$ is continuous in time, this shows
\[
	E_1 \le \lim_{t\to-\I }\norm{W_-(t)}_{\dot{H}^1} = \sqrt{\frac2d} \norm{W}_{\dot{H}^1},
\]
which completes the proof.
\end{proof}
\begin{remark}
If we consider the problem under the radial symmetry then we obtain the same conclusion for $d=3$ by using a result
by Kenig and Merle \cite{KM}. Without the radial symmetry, we only have the upper bounds $E_1 \le \sqrt{2/d} \norm{W}_{\dot{H}^1}$
and $E_2 \le \norm{W}_{\dot{H}^1}$.
\end{remark}

\section{Proof of Theorem \ref{thm:pd}}\label{sec:pf_profile}

In this section, we prove a linear profile decomposition.
Throughout this section, we assume that $d\ge 1$ and 
\[
 \frac2d \cdot \frac{1}{1+\frac{2}{d(d+3)}}<\alpha<\frac2d, \quad (d\alpha)'<r < ((d+2)\alpha)^*.
\]
The proof consists of two parts.
The first is a decomposition of a bounded sequence of functions in $\hat{M}^{d\alpha}_{2,r}$
with a different notion of smallness of remainders.
The second is a concentration compactness type result,
which assures that the modified notion of smallness is stronger than the original one.
%which is sometimes referred to as a inverse Strihcartz' inequality.

\subsection{Decomposition of a sequence}
Let us first introduce notations.
For a bounded sequence $P=\{P_n\}_n \subset \hat{M}^{d \alpha}_{2,r} $,
we introduce
a \emph{set of weak-$*$ limits modulo deformations}
\[
	\mathcal{M}(P)
	:=\left\{ \phi \in \hat{M}^{d \alpha}_{2,r} \ \Biggl|\  
	\begin{aligned}
	&\phi= \lim_{k\to\I} \mathcal{G}_{n_k}^{-1} P_{n_k} \text{weakly-}* \text{ in }\hat{M}^{d \alpha}_{2,r}, \\
	&\exists \mathcal{G}_n \in G, \,\exists \text{subsequence }n_k
	\end{aligned}
	\right\}.
\]
and define
\[
	\eta (P):= \sup_{\phi \in \mathcal{M}(P)}  \ell_r(\phi),
\]
where $\ell_r(\cdot)$ is the size function introduced in \eqref{eq:M_size}.
The main result of this section is a decomposition with a smallness
of remainders with respect to $\eta$.
\begin{theorem}\label{thm:pd1}
Let $u=\{u_n\}_n$ be a bounded sequence of functions in $\hat{M}^{d\alpha}_{2,r}$.
Then, there exist $\phi^j \in \mathcal{M}(u) $, $R_n^l \in \hat{M}^{d\alpha}_{2,r}$, and 
pairwise orthogonal families of deformations $\{\mathcal{G}_n^j\}_{n} \subset G$ ($j=1,2,\dots$)
 such that, up to subsequence, a decomposition
% \begin{equation}\label{eq:pd:key_decomp}
% 	u_n = \sum_{j=1}^l \mathcal{G}^j_n \phi^j + R_n^l
% \end{equation}
\eqref{eq:pf:decomp} holds for any $l,n \ge 1$.
Moreover, $\{R_n^j\}_{n,j}$ satisfies the convergence \eqref{eq:pf:swlimit} and 
\begin{equation}\label{eq:pd:key_smallness}
	\eta (R^l) \to 0
\end{equation}
as $l\to\I$. Furthermore, a decoupling inequality
\eqref{eq:pf:Pythagorean} holds for any $J\ge1$.
% \begin{equation}\label{eq:pd:key_Pythagorean}
% 	\limsup_{n\to\I} \ell_r(u_n)^r \ge \sum_{j=1}^J \ell_r(\phi^j)^r	
% 	+\limsup_{n\to\I} \ell_r(R_n^J)^r
% \end{equation}
% holds for all $J\ge1$.
\end{theorem}

We first recall a decoupling inequality in \cite{MS2}.
\begin{lemma}[Decoupling inequality]\label{lem:decouple} 
Let $\{u_n\}_n$ be a bounded sequence in $ \hat{M}^{d\alpha}_{2,r}$.
Let $\{\mathcal{G}_n\}_n \subset G$ be a sequence of deformations.
Suppose that $\mathcal{G}_n^{-1} u_n $ converges to $\phi$ weakly-$*$ in $\hat{M}^{d\alpha}_{2,r}$
as $n\to\I$.
Set $R_n := u_n - \mathcal{G}_n \phi$.
Then, for any $\gamma>1$ and $b_0 \in \R^d$, it holds that
\begin{equation}\label{eq:pd:key_decouple}
	\gamma \norm{P(b_0) u_n}_{\hat{M}^{d\alpha}_{2,r}}^{r}
	\ge \norm{P(b_0) \mathcal{G}_{n} \phi}_{\hat{M}^{d\alpha}_{2,r}}^{r}
	 + \norm{P(b_0) R_n}_{\hat{M}^{d\alpha}_{2,r}}^{r}
	+o_\gamma(1)
\end{equation}
as $n\to\I$.
\end{lemma}
The idea of the proof is to sum up the local (in the Fourier side) $L^2$ decoupling
with respect to intervals. We do not repeat details.

\begin{proof}[Proof of Theorem \ref{thm:pd1}]
We may suppose $\eta(u)>0$, otherwise the result holds with $\phi^j \equiv 0$ and
$R_n^j = u_n$ for all $j\ge1$.
Then, we can choose $\phi^1\in \mathcal{M}(u)$ so that $\ell_r(\phi^1) \ge \frac12 \eta(u)$
by definition of $\eta$.
Then, by definition of $\mathcal{M}(u)$, one finds $\mathcal{G}^1_n \in G$ such that
\[
	(\mathcal{G}^1_n)^{-1} u_n \rightharpoonup \phi^1 \swIN \hat{M}^{d\alpha}_{2,r}
\]
as $n\to\I$ up to subsequence.
Define $R^1_n := u_n - \mathcal{G}^1_n \phi^1 $.
Then, \eqref{eq:pf:decomp} holds for $l=1$.
It is obvious that
\begin{equation}\label{eq:pd:key_pf_wlim1}
	(\mathcal{G}_{n}^1)^{-1} R^1_n \rightharpoonup  \phi^1 -\phi^1 =0 \swIN \hat{M}^{d\alpha}_{2,r}
\end{equation}
as $n\to\I$.
By Lemma \ref{lem:decouple},
\begin{equation}\label{eq:pd:key_pf_Pythagoreanp}
	\gamma \norm{P(b_0) u_n}_{\hat{M}^{d\alpha}_{2,r}}^{r}
	\ge \norm{P(b_0)\mathcal{G}^1_{n} \phi^1}_{\hat{M}^{d\alpha}_{2,r}}^{r}
	 + \norm{P(b_0) R_n^1}_{\hat{M}^{d\alpha}_{2,r}}^{r}
	+o_\gamma (1)
\end{equation}
as $n\to\I$ for any constant $\gamma>1$ and $b_0\in\R^d$.
Since $\gamma>1$ and $b_0$ are arbitrary, the decoupling inequality \eqref{eq:pf:Pythagorean} holds for $J=1$.

% Let us go back to the decomposition procedure.
If $\eta(R^1)=0$ then the proof is completed by taking $\phi^j\equiv0$ for $j\ge2$.
Otherwise, we can choose $\phi^2 \in \mathcal{M}(R^1)$ so that $\ell_r(\phi^2) \ge \frac12 \eta(R^1)$.
Then, as in the previous step, one can take
% $\Gamma^2 \in 2^\Z \times \R \times \R \times \R$
$\mathcal{G}_n^2 \in G$ so that
\[
	(\mathcal{G}^2_n)^{-1} R^1_n \rightharpoonup \phi^2 \swIN \hat{M}^{d\alpha}_{2,r}
\]
as $n\to\I$ up to subsequence.
In particular, $\phi^2\not\equiv0$.
Together with \eqref{eq:pd:key_pf_wlim1}, Proposition \ref{prop:vanishing} (3) gives us that
$\{ (\mathcal{G}^2_n)^{-1}\mathcal{G}^1_n \}_n$ is vanishing.
Hence, $\mathcal{G}^1_n$ and $\mathcal{G}^2_n$ are orthogonal. 
Then, Proposition \ref{prop:vanishing} (2) implies that
\[
	(\mathcal{G}^2_n)^{-1} u_n
	= (\mathcal{G}^2_n)^{-1} \mathcal{G}^1_n \phi^1
	+ (\mathcal{G}^2_n)^{-1} R^1_n \rightharpoonup 0+\phi^2
	\swIN \hat{M}^{d\alpha}_{2,r}
\]
as $n\to\I$. Hence, we obtain $\phi^2 \in \mathcal{M}(u)$.
Set $R_n^2:=R_n^1-\mathcal{G}^2_n \phi^2$.
Then, \eqref{eq:pf:decomp} holds for $l=2$.
Further, one deduces from Lemma \ref{lem:decouple} that
\[
	\gamma \norm{P(b_0) R_n^1}_{\hat{M}^{d\alpha}_{2,r}}^r \\
	\ge  \norm{P(b_0) \mathcal{G}^2_n \phi^2}_{\hat{M}^{d\alpha}_{2,r}}^r
+ \norm{P(b_0) R_n^2}_{\hat{M}^{d\alpha}_{2,r}}^r +o_\gamma (1)
\]
as $n\to\I$ for any $\gamma>1$ and $b_0 \in \R^d$. This implies 
\eqref{eq:pf:Pythagorean}
for $J=2$ with the help of \eqref{eq:pd:key_pf_Pythagoreanp}.

Repeat this argument and construct $\phi^j \in \mathcal{M}(R^{j-1})$ and  $\mathcal{G}^j_n\in G$, inductively.
If we have $\eta(R^{j_0})=0$ for some $j_0$ then, for 
$j\ge j_0+1$, we take $\phi^j \equiv 0$ and define suitable $\mathcal{G}_{n}^j$ so that mutual orthogonality holds.
In what follows, we may suppose that $\eta(R^{j})>0$ for all $j\ge1$.
In each step, 
$R_n^j$ is defined by the formula $R_n^j = R_n^{j-1} - \mathcal{G}^j_n \phi^j$.
The decomposition \eqref{eq:pf:decomp} is obvious by construction. 

Let us now prove that pairwise orthogonality.
To this end, we demonstrate by induction that 
$\mathcal{G}^j_n$ is orthogonal to  $\mathcal{G}^k_n$ for $1\le k \le j-1$.
Since $(\mathcal{G}^j_n)^{-1} r_n^j \rightharpoonup \phi^j$
and $(\mathcal{G}^{j-1}_n)^{-1} r_n^j \rightharpoonup 0$ 
weakly-$*$ in $\hat{M}^{d\alpha}_{2,r}$ as $n\to\I$,
Proposition \ref{prop:vanishing} implies that $\mathcal{G}^j_n$ and $\mathcal{G}^{j-1}_n$ are orthogonal.
Suppose that $\mathcal{G}^j_n$ is orthogonal to  $\mathcal{G}^k_n$ for $k_0\le k \le j-1$.
Proposition \ref{prop:vanishing} (2) yields
\[
	(\mathcal{G}^j_n)^{-1} R_n^{k_0-1} = \sum_{k=k_0}^{j-1} (\mathcal{G}^j_n)^{-1}\mathcal{G}^k_n \phi^{k} + (\mathcal{G}^j_n)^{-1} R_n^{j-1}
	\rightharpoonup \phi^j \swIN \hat{M}^{d\alpha}_{2,r}
\]
as $n\to\I$. On the other hand, $(\mathcal{G}^{k_0-1}_n)^{-1} R_n^{k_0-1} \rightharpoonup 0$ as $n\to\I$.
We therefore see that $\mathcal{G}^j_n$ and $\mathcal{G}^{k_0-1}_n$ are orthogonal.
Hence, by induction, $\mathcal{G}^j_n$ is orthogonal to $\mathcal{G}^k_n$ for $1\le k \le j-1$.

The above argument also proves that the convergence \eqref{eq:pf:swlimit}.
and $\phi^j \in \mathcal{M}(u)$ for all $j\ge1$.

To conclude the proof,
we shall show \eqref{eq:pd:key_smallness} and \eqref{eq:pf:Pythagorean}.
Notice that the inductive construction gives us
\begin{equation}\label{eq:pd:key_pf_complete1}
	\ell_r (\phi^{j+1}) \ge \frac12 \eta(R^j)	
\end{equation}
for $j\ge1$ and
\begin{eqnarray}
	\lefteqn{\gamma \norm{P(b_0) R_n^j}_{\hat{M}^{d\alpha}_{2,r}}^r}
	\label{eq:pd:key_pf_complete2}\\
	&\ge& \norm{P(b_0) \mathcal{G}^{j+1}_n \phi^{j+1}}_{\hat{M}^{d\alpha}_{2,r}}^r
	 + \norm{P(b_0) R_n^{j+1}}_{\hat{M}^{d\alpha}_{2,r}}^r
	+o_{\gamma,j} (1).\nonumber
\end{eqnarray}
as $n\to\I$ for (fixed) $j\ge1$ and any $\gamma>1$ and $b_0 \in \R^d$.
Combining \eqref{eq:pd:key_pf_Pythagoreanp} and 
\eqref{eq:pd:key_pf_complete2} for $1 \le j \le J$, we have
\begin{align*}
	\gamma^J 
	\norm{P(b_0) u_n}_{\hat{M}^{d\alpha}_{2,r}}^r
	&{}\ge \sum_{j=1}^J \gamma^{J-j} 
	\norm{P(b_0) \mathcal{G}^j_n \phi^{j}}_{\hat{M}^{d\alpha}_{2,r}}^r + \norm{P(b_0) R_n^{J}}_{\hat{M}^{d\alpha}_{2,r}}^r +o_{\gamma,J} (1) \\
	&{}\ge \sum_{j=1}^J \gamma^{J-j} \ell_r(\phi^j)^r + \ell_r(R_n^J)^r +o_{\gamma,J} (1) .
\end{align*}
Take first infimum with respect to $b_0$ and then limit supremum in $n$ to obtain
\[
	\limsup_{n\to\I}\ell_r(u_n)^r \ge \sum_{j=1}^J \gamma^{-j} \ell_r(\phi^j)^r
	+\gamma^{-J} \limsup_{n\to\I}\ell_r(r_n^J)^r.
\]
Since $\gamma>1$ is arbitrary, we obtain \eqref{eq:pf:Pythagorean}.
Finally, \eqref{eq:pf:Pythagorean} and \eqref{eq:pd:key_pf_complete1} imply \eqref{eq:pd:key_smallness}.
\end{proof}

\subsection{Concentration compactness}
The second part of the proof of Theorem \ref{thm:pd}
is concentration compactness.
Intuitively, the meaning of the concentration compactness is as follows.
Let us consider a bonded sequence $\{ u_n \}_n \subset \hat{M}^{d\alpha}_{2,r}$.
Without any additional assumption, we may not expect to find any \emph{nonzero} weak-$*$ limit of the sequence.
Such a sequence is easily constructed by considering an orbit of \emph{general deformations}\footnote{
Here, general deformation means a set of bounded linear operators on $\hat{M}^{d\alpha}_{2,r}$ that the norms of itself and its inverse are uniformly bounded. The set $G$ given in Definition \ref{def:deformation} is an example.
There exist infinitely many such a set in view of multiplier type operators
$\{ e^{it\phi(i\nabla)}\}_{t\in\R}$, where $\phi$ is a real valued function.
The notion of general deformation is in the same spirit as what is called as \emph{displacement} in \cite{ST}.}.
So, to find a nonzero limit, we make some additional assumption on the sequence.
If the additional assumption is so strong that it removes all possible 
deformations that $\{ u_n \}_n$ may possess with few exceptions, say $G$, then
we can find a nonzero weak-$*$ limit modulo $G$.
In our case, the additional assumption is \eqref{eq:belowm} below.
Recall that $S_I(u) := \norm{u}_{L^{(d+2)\alpha}_{t,x}(I\times \R^d)}$.
\begin{theorem}[Concentration compactness]\label{thm:cc}
Let a bounded sequence $\{u_n\} \subset \hat{M}^{d\alpha}_{2,r}$ satisfy
\begin{equation}\label{eq:upM}
	\norm{u_n}_{\hat{M}^{d\alpha}_{2,r}}
	\le M
\end{equation}
and
\begin{equation}\label{eq:belowm}
% 	\norm{ e^{it\Delta} u_n}_{L^{(d+2)\alpha}_{t,x}(\R \times \R^{d})} 
	S_\R ( e^{it\Delta} u_n )
	\ge m	
\end{equation}
for some positive constants $m,M$. Then, 
$\eta(u) \ge \beta(m,M)$ holds for some positive constant $\beta(m,M)$ depending only on $m,M$.
\end{theorem}
\begin{remark}
The choice of a set $G$ given in Definition \ref{def:deformation} is the best one.
Remark that the assumptions \eqref{eq:upM} and \eqref{eq:belowm} are ``preserved'' under $D(h),T(a),P(b)$, and $U(s)$,
and so we cannot remove any of these actions from $G$.
The main point of the above theorem is that a set $G$ with these four actions is ``enough.''
\end{remark}

Plugging Theorem \ref{thm:cc} to Theorem \ref{thm:pd1},
we obtain desired decomposition result.
Before proceeding to the proof of Theorem \ref{thm:cc},
we complete the proof Theorem \ref{thm:pd}.

\begin{proof}[Proof of Theorem \ref{thm:pd}]
By means of Theorem \ref{thm:pd1},
it suffices to show \eqref{eq:pf:smallness} as $l\to\infty$
Assume for contradiction that a sequence $R_n^l$ given in Theorem \ref{thm:pd1} satisfies
\[
	\limsup_{l\to\I} \limsup_{n\to\I} 
% 	\norm{e^{it\Delta} R_n^l}_{L_{t,x}^{(d+2)\alpha}(\R \times \R^{d})} 
	S_\R (e^{it\Delta} R_n^l)
	>0.
\]
Then, we can choose $m>0$ and a subsequence $l_k$ with $l_k\to\I$ as $k\to\I$ such that
the assumption of Theorem \ref{thm:cc} is fulfilled for each $k$.
Then, Theorem \ref{thm:cc} implies $\eta(R^{l_k})\ge C \beta>0$,
which contradicts to \eqref{eq:pd:key_smallness}.
\end{proof}

The proof of Theorem \ref{thm:cc} consists of three steps.
The argument is very close to that in the mass-critical case $\alpha=2$
such as \cite{MV,CK,BV} or that for generalized KdV equation \cite{Shao,MS2}.
Throughout the proof, we fix $\tilde{q}$ and $\tilde{r}$ so that
\[
	2<\tilde{q} < \(d+\frac2{d+3}\)\alpha, \quad
	\tilde{r} = ((d+2)\alpha)^*.
\]
Remark that this choice gives
\begin{equation}\label{eq:ST3}
	c
% 	\norm{e^{it\Delta}f}_{L^{(d+2)\alpha}_{t,x}(\R \times \R^d)}
	S_\R (e^{it\Delta}f)
	\le  \norm{f}_{\hat{M}^{d\alpha}_{\tilde{q},\tilde{r}}}
	\le  \norm{f}_{\hat{M}^{d\alpha}_{2,r}}
\end{equation}
by Proposition \ref{prop:Strichartz}, where $c$ is a positive constant.

{\bf Step 1 -- Decomposition into a sum of scale pieces. } 
Let us start the proof of Theorem \ref{thm:cc}
with a decomposition of bounded sequence into some pieces 
of which Fourier transforms 
have mutually disjoint compact supports and are bounded.

\begin{lemma}\label{lem:pd:step1}
Suppose that a sequence 
$\{ u_n\}_n \subset \hat{M}^{d\alpha}_{2,r} $ satisfy $\norm{u_n}_{\hat{M}^{d\alpha}_{2,r}}\le M$.
Then, for any $\eps>0$,
there exist 
a subsequence of $n$ which denoted still by $n$, 
a number $J$, a sequence of dyadic cubes
$\{ \tau_n^j=h_n^j([0,1)^d + b_n^j) \}_{n} \subset \mathcal{D}$ ($h_n^j \in 2^\Z$, $b_n^j \in \Z^d$, $1\le j \le J$),
$\{ f_n^j\}_{n} \subset \hat{M}^{d\alpha}_{2,r}$ ($1 \le j \le J$), 
 and $ q_n \in \hat{M}^{d\alpha}_{2,r}$
such that 
\[
	\abs{\log \frac{h_n^j}{h_n^k}}
	+ \abs{ b_n^j - \frac{{h}_n^k}{h_n^j} {b}_n^k }  \to \I
\]
as $n\to \I$ for $1\le j<k\le J$,
and $u_n$ is decomposed into
\begin{eqnarray}
	u_n = \sum_{j=1}^J f_n^j + q_n
	\label{ufq}
\end{eqnarray}
for all $n \ge 1$.
Moreover, it holds that
\[
	\norm{u_n}_{\hat{M}^{d\alpha}_{2,r}}^r
 	\ge \sum_{j=1}^J \norm{f_n^j}_{\hat{M}^{d\alpha}_{2,r}}^r
 	+ \norm{q_n}_{\hat{M}^{d\alpha}_{2,r}}^r
\]
for all $n\ge 1$ and
\begin{eqnarray}\label{eq:qe}
	\limsup_{n\to\I} 
	\norm{q_n}_{\hat{M}^{d\alpha}_{\tilde{q},\tilde{r}}} \le \eps. 
\end{eqnarray}
Further, 
there exists a bounded and compactly supported function $F_j$ such that
 $\widehat{f_n^j}$ satisfies 
\begin{equation}\label{eq:pd:step1_ptwisebdd}
	|\tau_n^j|^{\frac1{\alpha'}} \abs{\F f_n^j ( h_n^j (\xi +b_n^j)) } \le F_j (\xi)
\end{equation}
for any $n\ge1$.
\end{lemma}

\begin{proof} Pick $\eps>0$.
If $\limsup_{n\to\I} \norm{u_n}_{\hat{M}^{d\alpha}_{\tilde{q},\tilde{r}}} \le \eps$ then there is nothing to 
be proved.
Otherwise, we can extract a subsequence so that
$$
	\norm{u_n}_{\hat{M}^{d\alpha}_{\tilde{q},\tilde{r}}} > \eps
$$
for all $n$. 
By H\"older's inequality, the embedding $\hat{M}^{d\alpha}_{2,r} \hookrightarrow \hat{M}^{d\alpha}_{q,r}$ ($q>2$), and assumption, one sees that
\[
	\eps \le 
	\norm{u_n}_{\hat{M}^{d\alpha}_{\tilde{q},r}}^{1-\theta}
	\norm{u_n}_{\hat{M}^{d\alpha}_{\tilde{q},\I}}^{\theta}
	\le  
	\norm{u_n}_{\hat{M}^{d\alpha}_{2,r}}^{1-\theta}
	\norm{u_n}_{\hat{M}^{d\alpha}_{\tilde{q},\I}}^{\theta} 
	\le  M^{1-\theta} \norm{u_n}_{\hat{M}^{d\alpha}_{\tilde{q},\I}}^{\theta},
% 	\left[ \sup_{I \in \mathcal{D}}  |I|^{-\frac1{r}} \norm{\hat{u_n}}_{L^{\frac{r}{r-2}}(I)} \right]^{1-\theta}.
\]
where $\theta = \theta(r,\tilde{r})\in (0,1)$.
Hence, %by definition of $\dot{M}^{\alpha}_{3\alpha/2,\I}$ norm,
there exists a dyadic cube $\tau_n^1 :=h_n^1([0,1)^d+b_n^1) \in \mathcal{D}$ with $h_n^1 \in 2^\Z$ and $b_n^1 \in \Z^d$
such that
\begin{equation}\label{eq:up1}
\begin{aligned}
	\int_{\tau_n^1} |\hat{u}_{n}|^{\tilde{q}'} \,d\xi 
	&{}\ge 
	C_1  \eps^{\frac{\tilde{q}'}{\theta}} |\tau_n^1|^{(\frac{1}{\tilde{q}'} - \frac1{(d\alpha)'})\tilde{q}'},
\end{aligned}
\end{equation}
where $C_1=C_1(r,\tilde{r},M)$ is a positive constant.
On the other hand,
it holds for any $A>0$ that
\begin{equation}\label{eq:up2}
\begin{aligned}
	\int_{\tau_n^1 \cap \{ |\hat{u}_{n}| \ge A\}} |\hat{u}_{n}|^{\tilde{q}'} d\xi
	&\le A^{\tilde{q}'-2} \norm{\hat{u}_{n}}_{L^2(\tau_n^1)}^{2} 
	\\
	&\le A^{\tilde{q}'-2}  |\tau_n^1|^{2(\frac1{d\alpha}-\frac1{2})}
	\norm{u_n}_{ \hat{M}^{d\alpha}_{2,r}}^2 \\
	&\le M^2 A^{\tilde{q}'-2}
	|\tau_n^1|^{\frac{2}{d\alpha}-1}.
\end{aligned}
\end{equation}
since $\tilde{q}'<2$.
We choose $A$ so that
\[
	M^2 A^{\tilde{q}'-2}
	|\tau_n^1|^{\frac{2}{d\alpha}-1} = \frac{C_1}2 \eps^{\frac{\tilde{q}'}{\theta}} |\tau_n^1|^{(\frac1{\tilde{q}'}-\frac{1}{(d\alpha)'})\tilde{q}'}.
\]
or more explicitly,
\[
	A = \(\frac{2}{C_1M^2}\)^{\frac1{2-\tilde{q}'}} \eps^{\frac1\theta \(\frac{\tilde{q}'}{2-\tilde{q}'}\)} |\tau_n^1|^{-\frac1{(d\alpha)'}}=:C_\eps |\tau_n^1|^{-\frac1{(d\alpha)'}}.
\]
From \eqref{eq:up1} and \eqref{eq:up2}, we have
\begin{equation}\label{eq:do1}
	\int_{\tau_n^1 \cap \{ |\hat{u}_{n} | \le
	C_\eps |\tau_n^1|^{-1/(d\alpha)'} \}} |\hat{u}_{n}|^{\tilde{q}'} d\xi \\
	\ge \frac{C_1}2 \eps^{\frac{\tilde{q}'}{\theta}} |\tau_n^1|^{(\frac{1}{\tilde{q}'} - \frac1{(d\alpha)'})\tilde{q}'}.
\end{equation}
H\"older's inequality implies that
\begin{multline}\label{eq:do2}
	\lefteqn{\int_{\tau_n^1 \cap \{ |\hat{u}_{n} | \le C_\eps |\tau_n^1|^{-1/(d\alpha)'} \} } |\hat{u}_{n}|^{\tilde{q}'} d\xi}
	\\
	\le \(\int_{\tau_n^1 \cap \{ |\hat{u}_{n} | \le C_\eps |\tau_n^1|^{-1/(d\alpha)'} \} } |\hat{u}_{n}|^2 d\xi \)^{\frac{\tilde{q}'}{2}} |\tau_n^1|^{1 -\frac{\tilde{q}'}2}.
\end{multline}
Combining \eqref{eq:do1} and \eqref{eq:do2}, 
we reach to the estimate
\begin{equation}\label{eq:d3}
\begin{aligned}
	|\tau_n^1|^{\frac1{(d\alpha)'}-\frac12} \(\int_{\tau_n^1 \cap \{ |\hat{u}_{n} | \le C_\eps |\tau_n^1|^{-1/(d\alpha)'} \} } |\hat{u}_{n}|^2 d\xi\)^{\frac12}
	\ge \(\frac{C_1}2\)^{\frac{1}{\tilde{q}`}}  \eps^{\frac{1}{\theta}}.
\end{aligned}
\end{equation}
We define $v_n^1$ by $\widehat{v_n^1}:= \hat{u}_{n} {\bf 1}_{\tau_n^1\cap \{  |
\hat{u}_{n} | \le C_\eps |\tau_n^1|^{-1/(d\alpha)'}\}}$ and $q_n^1:=u_n-v_n^1$.
Then, \eqref{eq:d3} implies that 
$\norm{v_n^1}_{\hat{M}^{d\alpha}_{2,r}}
	\ge C \eps^{\frac{1}{\theta}}$.
Further, we have 
\[
	|\tau_n^1|^{\frac{1}{(d\alpha)'}} \abs{ \widehat{v_n^1} (h_n^1 (\xi + b_n^1))} \le C_\eps
	{\bf 1}_{[0,1]^d}(\xi).
\]

If
$\limsup_{n\to\I}
\norm{q_n^1}_{\hat{M}^{d\alpha}_{\tilde{q},\tilde{r}}}
\le \eps$
then we have done.
Otherwise, the same argument with $u_n$ being replaced by $q_n^1$
enables us to define $\tau_n^2:=h_n^2([0,1)^d+b_n^2)$,
$v_n^2$, and $q_n^2$ (up to subsequence).
We repeat this argument and define $\tau_n^j:=h_n^j([0,1)^d+b_n^j)$, $v_n^j$, and $q_n^j$ inductively.
It is easy to see that
\[
	\norm{u_n}_{\hat{M}^{d\alpha}_{2,r}}^r \ge
	\sum_{j=1}^N \norm{v_n^j}_{\hat{M}^{d\alpha}_{2,r}}^r
	+ \norm{q_n^N}_{\hat{M}^{d\alpha}_{2,r}}^r
\]
since supports of $\{v_n^j\}_{1\le j \le N}$ and $q_n^N$  are disjoint in the Fourier side and since $r>2$.
Since $\|v_n^j\|_{\hat{M}^{d\alpha}_{2,r}}\ge C \eps^{\frac{1}{\theta}}$
for each $j$, together with \eqref{eq:ST3},
we see that
\[
	\limsup_{n\to\I} \norm{q_n^J}_{\hat{M}^{d\alpha}_{\tilde{q},\tilde{r}}}
	\le \eps
\]
holds in at most $J=J(\eps)$ steps. Set $q_n:=q_n^J$.

We reorganize $v_n^j$ to obtain
mutual asymptotic orthogonality.
It is done as follows;
We collect all $k\ge2$ such that $|\log \frac{h_n^1}{h_n^k}| + \abs{b_n^1-\frac{h_n^k}{h_n^1}b_n^k}$ is bounded, and define $f_n^1:=v_n^1+\sum_{k} v_n^k$.
Since 
\begin{align*}
	&|\tau_n^1|^{\frac{1}{(d\alpha)'}} \abs{ \widehat{v_n^k} (h_n^1 (\xi + b_n^1))} \\
	&{} =  \(\frac{h_n^1}{h_n^k}\)^{\frac{1}{(d\alpha)'}} |\tau_n^k|^{\frac{1}{(d\alpha)'}} \abs{ \widehat{v_n^k} \(  h_n^k \left[ \frac{h_n^1}{h_n^k} \left\{ \xi  + \( b_n^1 - \frac{h_n^k}{h_n^1} b_n^k \) \right\}   + b_n^k \right] \)} \\
	&{} \le  C_\eps \(\frac{h_n^1}{h_n^k}\)^{\frac{1}{(d\alpha)'}} {\bf 1}_{[0,1]^d} \( \frac{h_n^1}{h_n^k} \left\{ \xi  + \( b_n^1 - \frac{h_n^k}{h_n^1} b_n^k \) \right\} \) ,
\end{align*}
we see that $|\tau_n^1|^{\frac{1}{(d\alpha)'}} \F{f_n^1} (h_n^1 (\xi + b_n^1)) \le F_1(\xi)$
for some bounded and compactly supported function $F_1$.
Similarly, we define $f_n^j$ inductively.
It is easy to see that $f_n^j$ possesses all properties we want. 
This completes the proof of Lemma \ref{lem:pd:step1}.
\end{proof}

{\bf Step 2 -- Decomposition of each scale pieces.} 
We next decompose functions obtained in the previous decomposition.

\begin{lemma}\label{lem:pd:step2}
Let $F(\xi)$ be a nonnegative bounded function with compact support.
Suppose that a sequence $R_n \in \hat{M}^{d\alpha}_{2,r}$ satisfy 
\begin{equation}\label{eq:pd:step2_support_assumption}
	|\widehat{R_n} (\xi) | \le F (\xi).
\end{equation}
Then, up to subsequence, there exist $\{\phi^l\}_l \subset \hat{M}^{d\alpha}_{2,r} $ with $ |\widehat{\phi^l} (\xi)| \le F(\xi) $, 
$(a_n^l,s_n^l) \in \R^d \times \R^d$ with
\[
	\lim_{n\to\I} (|s_n^l - s_n^{\tilde{l}} |+ 
	|a_n^l -a_n^{\tilde{l}}|) = \I
\]
for any $\tilde{l}\neq l$, and $\{r_n^l\}_{n,l}\subset \hat{M}^{d\alpha}_{2,r}$ with $|\widehat{r_n^l} (\xi)| \le F(\xi)$ such that
\[
	R_n(x) = \sum_{l=1}^L U(s_n^l) T({a_n^l}) \phi^l(x) + r_n^L(x)
\]
for any $L\ge1$. Moreover, it holds that
\begin{equation}\label{eq:pd:step2_Pythagorean}
	\sum_{l=1}^L 
	\norm{\phi^l}_{\hat{M}^{q}_{2,r}}^r
	+\limsup_{n\to\I} 
	\norm{r_n^L}_{\hat{M}^{q}_{2,r}}^r
	\le \limsup_{n\to\I} 
	\norm{R_n}_{\hat{M}^{q}_{2,r}}^r
	<\I
\end{equation}
for any $2<q'<r<\I$ and $L\ge1$. 
Furthermore, we have
\begin{equation}\label{eq:pd:step2_smallness}
	\limsup_{n\to\I} 
% 	\norm{ e^{it\Delta} r_n^L}_{L^{(d+2)\alpha}_{t,x}(\R \times \R^{d})}
	S_\R (e^{it\Delta} r_n^L)
	\to 0
\end{equation}
as $L\to\I$.
\end{lemma}

For the proof, see \cite[Proposition 3.4]{CK}.
A key restriction estimate in our case is established in Proposition \ref{prop:Strichartz}.

{\bf Step 3 -- Completion of the proof of Theorem \ref{thm:cc}.} 
We are now ready to prove Theorem \ref{thm:cc}.
For the proof,
we recall the following space-time nonresonant property.
\begin{lemma}\label{lem:realorthty2}
Let $\phi^j \in \hat{M}^{d\alpha}_{2,r}$ ($1\le j \le J$).
Let $\{\mathcal{G}^j_n \}_n \subset G$ ($1\le j \le J$)
be mutually orthogonal families.
Then,
\[
	\norm{\sum_{j=1}^J e^{it\Delta} \mathcal{G}^j_{n} \phi^j}_{L_{t,x}^{(d+2)\alpha}(\R\times\R^d)}^{(d+2)\alpha}
	\le \sum_{j=1}^J \norm{e^{it\Delta} \mathcal{G}^j_{n} \phi^j}_{L_{t,x}^{(d+2)\alpha}(\R\times \R^d)}^{(d+2)\alpha}
	+o(1)
\]
as $n\to\I$.
\end{lemma}
The proof is standard. For instance, see \cite[Lemma 5.5]{BV}. Note that $(d+2)\alpha > \frac{2(d+3)}{d+1}>2$.

\begin{proof}[Proof of Theorem \ref{thm:cc}]
Let $\{u_{n}\}\subset\hat{M}^{d\alpha}_{2,r}$ be a bounded 
sequence satisfying \eqref{eq:upM} and \eqref{eq:belowm}. 
Let $\eps=\eps(m,M)>0$ to be chosen later.
Let $J=J(\eps)\ge1$,
$\{ I_n^j=h_n^j([0,1)^d + b_n^j) \}_n \subset \mathcal{D}$ ($1\le j\le J$),
 $\{ f_n^j \}_n \subset \hat{M}^{d\alpha}_{2,r}$ ($1\le j\le J$), and $q_n$ be sequences 
given in Lemma \ref{lem:pd:step1}. 
Set
\[
	\widehat{R_n^j}(\xi):=|h_n^j|^{\frac{d}{(d\alpha)'}} \widehat{f_n^j} (h_n^j (\xi + b_n^j)).
\]
Namely, $R_n^j = P({b_n^j})^{-1} {D} (h_n^j)^{-1}  {f}_n^j$.
Then, 
by means of \eqref{eq:pd:step1_ptwisebdd}, $\{R_n^j\}_n$ satisfies assumption of
Lemma \ref{lem:pd:step2} for each $j$.
Then, thanks to Lemma \ref{lem:pd:step2}, 
for every $1\le j \le J$, there exists a family $\{ \phi^{j,l} \}_{l} \subset \hat{M}^{d\alpha}_{2,r}$,
and a family $\{ (a^{j,l}_n,s^{j,l}_n) \}_{n,l} \in \R^d \times \R^d$
such that
\[
	R_n^j = \sum_{l=1}^L U(s_n^{j,l}) T(a_n^{j,l}) \phi^{j,l} +  r_n^{j,L}
\]
with
\[
	\limsup_{n\to\I} 
% 	\norm{e^{it\Delta} r_n^{j,L} }_{L_{t,x}^{(d+2)\alpha}(\R \times \R^d)}
	S_\R (e^{it\Delta} r_n^{j,L})
	\to 0
\]
as $L\to\I$
and that 
\[
	\lim_{n\to\I} (|s_n^{j,l} - s_n^{j,\tilde{l}} |
	+|a_n^{j,l} -a_n^{j,\tilde{l}}|) = \I
\]
for any $l\neq\tilde{l}$.
Remark that 
\begin{eqnarray}
	f_n^j &=& D(h_n^j) P(b_n^{j})  R_n^j
	\label{fff}\\
	&=& \sum_{l=1}^L  D(h_n^j)  P(b_n^j) U(s_n^{j,l}) T(a_n^{j,l})\phi^{j,l} +  D(h_n^j) P(b_n^{j}) r_n^{j,L}.
	\nonumber
\end{eqnarray}
We choose $L=L(\eps)$ so that 
\begin{equation}\label{eq:ej}
	\limsup_{n\to\I} 
% 	\norm{e^{it\Delta} D(h_n^j) P(b^j_n) r_n^{j,L}}_{L_{t,x}^{(d+2)\alpha}(\R\times\R^d)}
	S_\R ( e^{it\Delta} D(h_n^j) P(b^j_n) r_n^{j,L} )
	\le \frac{\eps}{J}
\end{equation}
holds for any $1 \le j \le J$.
Notice that this is possible by means of the scale invariance and Galilean transform
% \[
% 	\norm{e^{it\Delta} D(h_n^j)P(b^j_n) r_n^{j,L}}_{L^{(d+2)\alpha}_{t,x}(\R\times \R^d)}
% 	=\norm{ e^{it\Delta} r_n^{j,L}}_{L^{(d+2)\alpha}_{t,x}(\R\times \R^d)}.
% \]
\[
	S_\R( e^{it\Delta} D(h_n^j)P(b^j_n) r_n^{j,L} )
	=S_\R ( e^{it\Delta} r_n^{j,L} ) .
\]
Let $r_n:=\sum_{j=1}^J  D(h_n^j) P(b_n^{j}) r_n^{j,L}+ q_n$.
By Lemma \ref{lem:pd:step1} (\ref{ufq}) and (\ref{fff}), we have
\begin{equation}\label{eq:ufqr1}
	u_n= \sum_{j=1}^J f_n^j + q_n = \sum_{j=1}^J \sum_{l=1}^{L}  \mathcal{G}_n^{j,l} \phi^{j,l} + r_n,
\end{equation}
where $\mathcal{G}_n^{j,l}:=D(h_n^j) P(b_n^j) U(s_n^{j,l}) T(a_n^{j,l}) $.
% It is easy to see that $\{\mathcal{G}_n^{j,a}\}_n \subset G$ 
% are pairwise orthogonal families.
We renumber $(j,l) \in \{ 1\le j \le J,\, 1\le l \le L\}$ as $k=1,2,\dots,K$ and rewrite
\eqref{eq:ufqr1} as
\begin{equation}\label{eq:ufqr}
	u_n= \sum_{k=1}^{K}  \mathcal{G}_n^{k} \phi^{k} + r_n,
\end{equation}
Note that $\{\mathcal{G}_n^{k}\}_n$ ($k=1,2,\dots,K$) are mutually orthogonal families of 
deformations. Indeed, if $k_1 \neq k_2$ we have either $j(k_1)\neq j(k_2)$, or 
$j(k_1)=j(k_2)$ and $l(k_1)\neq l(k_2)$, where $j(k)$ and $l(k)$ are numbers given 
by the above renumbering procedure, $k=(j,l)$.
In the first case, the orthogonality follows from Lemma \ref{lem:pd:step1}.
In the second case, Lemma \ref{lem:pd:step2} gives the orthogonality, since
$h_n^{k_1}\equiv h_n^{k_2}$ and $b_n^{k_1} \equiv b_n^{k_2}$ in this case.

By \eqref{eq:ej}, \eqref{eq:qe}, and Proposition \ref{prop:Strichartz}, we have
\[
% 	\norm{ e^{it\Delta} u_n}_{L^{(d+2)\alpha}_{t,x}(\R\times \R^d)}
	S_\R ( e^{it\Delta} u_n )
	\le 
% 	\norm{e^{it\Delta} (u_n - r_n)}_{L^{(d+2)\alpha}_{t,x}(\R\times \R^d)}
	S_\R ( e^{it\Delta} (u_n - r_n) )
	+C\eps.
\]
By assumption and Proposition \ref{prop:Strichartz},
\[
% 	\norm{e^{it\Delta} (u_n - r_n)}_{L^{(d+2)\alpha}_{t,x}(\R\times \R^d)} 
	S_\R ( e^{it\Delta} (u_n - r_n) )
	\le CM.
\] 
Combining the above inequality and Lemma \ref{lem:realorthty2}, 
one can verify that
\begin{align*}
% 	\norm{ e^{it\Delta} u_n}_{L^{(d+2)\alpha}_{t,x}(\R\times \R^d)}^{(d+2)\alpha}
	S_\R ( e^{it\Delta} u_n )^{(d+2)\alpha}
	\le{}&  \sum_{k=1}^{K} 
% \norm{e^{it\Delta} \mathcal{G}^{k}_{n} \phi^{k}}_{L^{(d+2)\alpha}_{t,x}(\R\times \R^d)}^{(d+2)\alpha}
	S_\R ( e^{it\Delta} \mathcal{G}^{k}_{n} \phi^{k} )^{(d+2)\alpha}
	 +  C\eps +o(1)
\end{align*}
as $n\to\I$. 
Notice that
% \[
% 	\norm{e^{it\Delta} \mathcal{G}^{k}_{n} \phi^{k}}_{L^{(d+2)\alpha}_{t,x}(\R\times \R^d)}
% 	= \norm{e^{it\Delta} \phi^{k}}_{L^{(d+2)\alpha}_{t,x}(\R\times \R^d)}.
% \]
$S_\R ( e^{it\Delta} \mathcal{G}^{k}_{n} \phi^{k} )
	= S_\R ( e^{it\Delta} \phi^{k} )$.
By \eqref{eq:belowm}, we can take $\eps=\eps(m,M)$ small
and $n$ large enough to get
\[
	C_\alpha m^{(d+2)\alpha} \le \sum_{k=1}^K 
% 	\norm{ e^{it\Delta} \phi^{k} }_{L^{(d+2)\alpha}_{t,x}(\R \times \R^d)}^{(d+2)\alpha} .
	S_\R ( e^{it\Delta} \phi^{k} )^{(d+2)\alpha} .
\]
On the other hand, by Proposition \ref{prop:Strichartz},
\[
% 	\norm{ e^{it\Delta} \phi^{k} }_{L_{t,x}^{(d+2)\alpha}(\R \times \R^d)} 
	S_\R ( e^{it\Delta} \phi^{k} )
	\le C 
	\norm{ \phi^{k}  }_{\hat{M}^{d\alpha}_{2,r}}.
\]
Since $(d+2)\alpha \ge \tilde{r} > r $, we have
\begin{align*}
	C_\alpha m^{(d+2)\alpha} &{}\le C\( \sup_{1\le k\le K} 
% 	\norm{ e^{it\Delta} \phi^{k} }_{L^{(d+2)\alpha}_{t,x}(\R \times \R^d)}
	S_\R ( e^{it\Delta} \phi^{k} )
	\)^{(d+2)\alpha-r}  \sum_{k=1}^K \norm{ \phi^{k} }_{\hat{M}^{d\alpha}_{2,r}}^r \\
	&{}\le C\( \sup_{1\le k\le K} 
% 	\norm{ e^{it\Delta} \phi^{k} }_{L^{(d+2)\alpha}_{t,x}(\R \times \R^d)} 
	S_\R ( e^{it\Delta} \phi^{k} )
	\)^{(d+2)\alpha-r}  M^r.
\end{align*}
Thus, there exists $k_0$ such that
\begin{equation}\label{eq:pf_lbound}
% 	\norm{ e^{it\Delta} \phi^{k_0} }_{L^{(d+2)\alpha}_{t,x}(\R \times \R^d)} 
	S_\R ( e^{it\Delta} \phi^{k_0} )
	\ge C_\alpha\(\frac{m^{(d+2)\alpha}}{M^r}\)^{\frac1{(d+2)\alpha-r}}.
\end{equation}

Now, up to subsequence, we have
\[
	(\mathcal{G}^{k_0}_{n})^{-1} u_n \rightharpoonup \phi^{k_0} + q_0=:\psi \swIN \hat{M}^{d\alpha}_{2,r}
\]
as $n\to\I$, where $q_0$ is a weak-$*$ limit of $(\mathcal{G}^{k_0}_{n})^{-1} q_n$.
Indeed, by Lemma \ref{lem:pd:step1} (\ref{ufq}), we have
\[
	u_n = \sum_{1\le j \le J ,\,j\neq j_0} f^j_n + f^{j_0}_n + q_n ,
\]
where $(j_0,l_0)$ is a pair corresponds to $k_0$.
By orthogonality obtained in Lemma \ref{lem:pd:step1},
one has
$(\mathcal{G}^{j_0,l_0}_{n})^{-1} f^j_n \rightharpoonup 0$ weakly-$*$ in $\hat{M}^{d\alpha}_{2,r}$
as $n\to\I$ for $j\neq j_0$.
Further, $(\mathcal{G}^{j_0,l_0}_{n})^{-1} f^{j_0}_n \rightharpoonup \phi^{j_0,l_0}$ weakly-$*$ in $\hat{M}^{d\alpha}_{2,r}$ as $n\to\I$.
Therefore, we have the above limit and so $\psi \in \mathcal{M}(u)$.

Let us estimate $\norm{\psi}_{\hat{M}^{d\alpha}_{2,r}}$ from below.
Since $(\mathcal{G}^{k_0}_{n})^{-1} q_n$ is also bounded in $\hat{M}_{\tilde{q},\tilde{r}}^{d\alpha}$
in light of \eqref{eq:ST3},
it converges to $q_0$ also weakly-$*$ in $\hat{M}_{\tilde{q},\tilde{r}}^{d\alpha}$ as $n\to\I$.
One then sees from lower semi-continuity that
\[
	\norm{q_0}_{\hat{M}^{d\alpha}_{\tilde{q},\tilde{r}}} \le \limsup_{n\to\I} 
	\norm{q_n}_{\hat{M}^{d\alpha}_{\tilde{q},\tilde{r}}} \le C\eps.
\]
Thanks to \eqref{eq:ST3}, we have
\[
% 	\norm{e^{it\Delta}q_0}_{L_{t,x}^{(d+2)\alpha}(\R \times \R^d )} 
	S_\R (e^{it\Delta}q_0 )
	\le \norm{q_0}_{\hat{M}^{d\alpha}_{\tilde{q},\tilde{r}}}
	\le C \eps.
\]
Finally, using Proposition \ref{prop:Strichartz} and \eqref{eq:pf_lbound}, and
choosing $\eps=\eps(m,M)>0$ even smaller if necessary, we reach to the estimate
\begin{align*}
	\norm{\psi}_{\hat{M}^{d\alpha}_{2,r}}
	&{}\ge C
% 	\norm{ e^{it\Delta} \mathcal{G}^{k_0}_{n} \psi}_{L_{t,x}^{(d+2)\alpha}(\R \times \R^d )} \\
	(S_\R ( e^{it\Delta} \mathcal{G}^{k_0}_{n} \phi^{k_0} ) - S_\R ( e^{it\Delta} \mathcal{G}^{k_0}_{n} q_0 )) \\
	&{}\ge C
% 	\norm{ e^{it\Delta} \mathcal{G}^{k_0}_{n} \phi^{k_0}}_{L_{t,x}^{(d+2)\alpha}(\R \times \R^d )}
	S_\R ( e^{it\Delta}  \phi^{k_0} )
	-C\eps \\
	&{}\ge \frac{C}2\(\frac{m^{(d+2)\alpha}}{M^r}\)^{\frac1{(d+2)\alpha-r}} =: \beta(m,M),
\end{align*}
which completes the proof of Theorem \ref{thm:cc}.
\end{proof}

\subsection*{Acknowledgments}
The author expresses his deep gratitude
to Professors Rowan Killip and Monica Visan for fruitful discussions and valuable comments on 
the preliminary version of the manuscript. 
In particular, the evaluation of $E_1$ in the energy critical case (Theorem \ref{thm:previous2}) is due to them.
The part of this work was done while the author was 
visiting at Department of Mathematics at the University of
California, Los Angeles.  The author gratefully acknowledges their hospitality. 
The stay is supported by JSPS KAKENHI, Scientific Research (S) 23224003.
The author is supported by JSPS KAKENHI, Grant-in-Aid for Young Scientists (B) 24740108.

% \bibliographystyle{amsplain}
% \bibliography{subcrit3}

\providecommand{\bysame}{\leavevmode\hbox to3em{\hrulefill}\thinspace}
\providecommand{\MR}{\relax\ifhmode\unskip\space\fi MR }
\providecommand{\MRhref}[2]{%
  \href{http://www.ams.org/mathscinet-getitem?mr=#1}{#2}
}
\providecommand{\href}[2]{#2}

\end{document}